\documentclass[11pt,reqno]{article}
\usepackage{amsmath,amssymb,amsthm}
\usepackage{color,url}
\newtheorem{theorem}{Theorem}[section]
\newtheorem{proposition}[theorem]{Proposition}
\newtheorem{corollary}[theorem]{Corollary}
\newtheorem{definition}[theorem]{Definition}
\newtheorem{lemma}[theorem]{Lemma}
\newtheorem{example}[theorem]{Example}
\newtheorem{remark}[theorem]{Remark}
\newtheorem{problem}[theorem]{Problem}
\numberwithin{equation}{section}

\usepackage
{hyperref}

\def\bR{\mathbb{R}}
\def\bN{\mathbb{N}}
\def\bM{\mathbb{M}}
\def\bC{\mathbb{C}}
\def\cA{\mathcal{A}}
\def\cH{\mathcal{H}}
\def\cM{\mathcal{M}}
\def\cN{\mathcal{N}}

\def\Tr{\mathrm{Tr}}
\def\<{\langle}
\def\>{\rangle}

\def\eps{\varepsilon}
\def\dom{\mathrm{dom}}
\def\cl{\mathrm{cl}}
\def\diag{\mathrm{diag}}

\def\id{\mathrm{id}}
\def\supp{\mathrm{supp}}
\def\q{\mathrm{q}}
\def\meas{\mathrm{meas}}
\def\cE{\mathcal{E}}
\def\cP{\mathcal{P}}
\def\cC{\mathcal{C}}

\begin{document}
\allowdisplaybreaks

\centerline{\LARGE Log-majorizations between quasi-geometric type means for matrices}

\bigskip
\bigskip
\centerline{\large
Fumio Hiai\footnote{{\it E-mail:} hiai.fumio@gmail.com
\ \ {\it ORCID:} 0000-0002-0026-8942}}

\begin{center}
$^1$\,Graduate School of Information Sciences, Tohoku University, \\
Aoba-ku, Sendai 980-8579, Japan
\end{center}

\medskip

\begin{abstract}
In this paper, for $\alpha\in(0,\infty)\setminus\{1\}$, $p>0$ and positive semidefinite matrices $A$ and $B$,
we consider the quasi-extension $\mathcal{M}_{\alpha,p}(A,B):=\mathcal{M}_\alpha(A^p,B^p)^{1/p}$ of
several $\alpha$-weighted geometric type matrix means $\mathcal{M}_\alpha(A,B)$ such as the
$\alpha$-weighted geometric mean in Kubo--Ando's sense, the R\'enyi mean, etc. The log-majorization
$\mathcal{M}_{\alpha,p}(A,B)\prec_{\log}\mathcal{N}_{\alpha,q}(A,B)$ is examined for pairs
$(\mathcal{M},\mathcal{N})$ of those $\alpha$-weighted geometric type means. The joint concavity/convexity
of the trace functions $\mathrm{Tr}\,\mathcal{M}_{\alpha,p}$ is also discussed based on theory of quantum
divergences.


\end{abstract}

{\small
\tableofcontents}

\section{Introduction}\label{Sec-1}

Several types of majorizations have been well developed for the eigenvalues and the singular values of
matrices, as is fully described in the book \cite{MOA} (see also \cite{An2,Bh,Hi} for example), giving rise to
powerful tools in deriving various matrix norm and trace inequalities. Among others, the notion of
multiplicative type of majorization called log-majorization (see Section \ref{Sec-2} for definition) has played
an important role in matrix analysis, mathematical physics, quantum information, etc. For instance, for
positive semidefinite matrices $A,B\ge0$, Araki's log-majorization \cite{Ar} of Golden--Thompson type is
\begin{align}\label{F-1.1}
(A^{p/2}B^pA^{p/2})^{1/p}\prec_{\log}(A^{q/2}B^qA^{q/2})^{1/q},\qquad0<p\le q,
\end{align}
which is a stronger version of Araki--Lieb--Thirring trace inequality. Ando--Hiai's one \cite{AH} of
complementary Golden–Thompson type is
\begin{align}\label{F-1.2}
(A^p\#_\alpha B^p)^{1/p}\prec_{\log}(A^q\#_\alpha B^q)^{1/q},\qquad0<q\le p,
\end{align}
where $A\#_\alpha B$ is the $\alpha$-weighted geometric mean of $A,B$. On the other hand, for positive
definite matrices $A,B>0$, Kian--Seo \cite{KS} obtained the counterpart of \eqref{F-1.2} for
$\alpha\in[-1,0)\cup(1,2]$ as
\begin{align}\label{F-1.3}
(A^p\#_\alpha B^p)^{1/p}\prec_{\log}(A^q\#_\alpha B^q)^{1/q},\qquad0<p\le q,
\end{align}
where the definition of $A\#_\alpha B:=A^{1/2}(A^{-1/2}BA^{-1/2})^\alpha A^{1/2}$ for $A,B>0$ extends to all
$\alpha\in\bR$.

In our previous paper \cite{Hiai} we considered the quasi-arithmetic matrix mean
$\cA_{\alpha,p}(A,B):=\{(1-\alpha)A^p+\alpha B^p\}^{1/p}$ and several quasi-geometric type matrix means
such as the quasi $\alpha$-weighted geometric mean $G_{\alpha,p}(A,B):=(A^p\#_\alpha B^p)^{1/p}$,
the quasi versions $SG_{\alpha,p}$, $\widetilde SG_{\alpha,p}$ of two different $\alpha$-weighted spectral
geometric means, the R\'enyi mean $R_{\alpha,p}$, and the Log-Euclidean mean $LE_\alpha$ for
$0<\alpha<1$ and $p>0$ (see Section \ref{Sec-2} for the precise definitions of these). In \cite{Hiai} we
examined various quasi-arithmetic-geometric inequalities between each of those quasi-geometric means
and $\cA_{\alpha,q}$ with respect to different matrix orderings. In the present paper we continue to consider
the above mentioned quasi-geometric type means for all $\alpha\in(0,\infty)\setminus\{1\}$ and $p>0$.
In fact, \eqref{F-1.1} is rewritten as $R_{\alpha,p}\prec_{\log}R_{\alpha,q}$, and \eqref{F-1.2} and
\eqref{F-1.3} are nothing but $G_{\alpha,p}\prec_{\log}G_{\alpha,q}$. Furthermore, the log-majorizations
$G_{\alpha,p}\prec_{\log}R_{\alpha,q}$ and $R_{\alpha,q}\prec_{\log}G_{\alpha,p}$ were characterized in
\cite{Hi2}, and $R_{\alpha,q}(A,B)\prec_{\log}SG_{\alpha,p}(A,B)$ was addressed in a recent paper \cite{GT}.
In this paper we examine the log-majorization $\cM_{\alpha,p}\prec_{\log}\cN_{\alpha,q}$ for other pairs
$(\cM,\cN)$ from $\{R,G,SG,\widetilde SG,LE\}$. Our goal is to hopefully obtain the necessary
and sufficient condition on $p,q,\alpha$ under which $\cM_{\alpha,p}(A,B)\prec_{\log}\cN_{\alpha,q}(A,B)$
holds for all $A,B>0$, though we have not succeeded it for all cases. To do so, we need to address the
sufficiency condition and the necessity condition separately. For the sufficiency part, since
$\det\cM_{\alpha,p}(A,B)=\det\cN_{\alpha,q}(A,B)=(\det A)^{1-\alpha}(\det B)^\alpha$, the standard technique
using anti-symmetric tensor powers is utilized so that the problem boils down to prove that
\[
\cN_{\alpha,q}(A,B)\le I\implies\cM_{\alpha,p}(A,B)\le I.
\]
The necessity part is quite computation-oriented. Similarly to the previous paper \cite{Hiai}, we make
computations for specific pairs of $2\times2$ positive definite matrices to restrict the possible range of
$p,q,\alpha$ for which $\cM_{\alpha,p}\prec_{\log}\cN_{\alpha,q}$ holds.

The joint concavity/convexity of matrix trace functions in two (or more) variables has been one of major
subjects in matrix analysis since Lieb's seminal paper \cite{Lieb} in 1973. Lieb's concavity \cite{Lieb} and
Ando's convexity \cite{An0} are especially famous, and the joint concavity/convexity of the trace functions
of the forms $\Tr(A^p+B^q)^s$ and $\Tr(A^{p/2}B^qA^{p/2})^s$ was developed in, e.g., \cite{CL,Hi4}. The
problem of fully characterizing the joint concavity/convexity of the above latter trace function became very
important, because it settles the question of the monotonicity property under CPTP maps (or the
data-processing inequality) of the $\alpha$-$z$-R\'enyi divergence \cite{AD}, as explained in detail in
\cite{CFL}. Then the problem was finally solved by Zhang \cite{Zh} in 2020, who characterized the joint
concavity/convexity of a more general trace function $(A,B)\mapsto\Tr(A^{p/2}X^*B^qXA^{p/2})^s$.
Also, the joint concavity of the trace function $\Tr(A^p\sigma B^q)^s$ for Kubo--Ando's operator means
$\sigma$ \cite{KA} was shown in \cite{Hi4}. The current status of the subject has been summarized in
Carlen's recent survey paper \cite{Ca1} and his new book \cite{Ca2}. Our second aim of this paper is to
examine the joint concavity/convexity of the trace functions $\Tr\,\cM_{\alpha,p}$ for our quasi-geometric
type means $\cM_{\alpha,p}$ and for $\alpha\in(0,\infty)\setminus\{1\}$ and $p>0$, while the case
$\Tr\,R_{\alpha,p}$ is included in Zhang's result in \cite{Zh}, and $\Tr\,G_{\alpha,p}$ for $0<\alpha<1$ is a
special case in \cite{Hi4}. Our idea is to apply the relationship between the joint concavity/convexity and the
monotonicity under CPTP maps for geometric type matrix functions of two variables (see Theorem \ref{T-5.3})
similarly to the situation in the above mentioned monotonicity question of the $\alpha$-$z$-divergence. Then
as a necessary condition for joint concavity/convexity of $\Tr\,\cM_{\alpha,p}$ we have the following
sandwiched inequalities (see Theorem \ref{T-5.8})
\begin{align*}
&\Tr\,G_{\alpha,1}(A,B)\le\Tr\,\cM_{\alpha,p}(A,B)\le\Tr\,R_{\alpha,1/\alpha}(A,B)
\quad\mbox{for $0<\alpha<1$}, \\
&\Tr\,R_{\alpha,1/\alpha}(A,B)\le\Tr\,\cM_{\alpha,p}(A,B)\le\Tr\,G_{\alpha,1}(A,B)
\quad\mbox{for $\alpha>1$}.
\end{align*}
These inequalities can be analyzed from the log-majorizations between $\cM_{\alpha,p}$ and
$G_{\alpha,1}$, $R_{\alpha,1/\alpha}$. But it is left open to find a sufficient condition for $\cM_{\alpha,p}$
to be jointly concave/convex, which seems a difficult problem.

The structure of the paper is as follows. In Section \ref{Sec-2} we review the definitions of the
quasi-geometric matrix means $R_{\alpha,p}$, $G_{\alpha,p}$, $SG_{\alpha,p}$, $\widetilde SG_{\alpha,p}$
and $LE_\alpha$ mentioned above as well as the notion of log-majorization. We add a few basic properties
of them to those given in \cite{Hiai}. The main Section \ref{Sec-3} is divided into four subsections. In those
subsections we examine the log-majorizations $\cM_{\alpha,p}\prec_{\log}\cN_{\alpha,q}$ for all pairs
$(\cM,\cN)$, except for the already known cases, from $\{R,G,SG,\widetilde SG,LE\}$. In Section \ref{Sec-4}
we characterize the equality cases in the norm inequalities derived from the log-majorizations shown in
Section \ref{Sec-3} in terms of the commutativity of matrix variables $A,B$, expanding the former results in
\cite{Hi1,Hi3}. In Section \ref{Sec-5} we discuss the joint concavity/convexity of $\Tr\,\cM_{\alpha,p}$ for
$\cM=R,G,SG,\widetilde SG,LE$ though some cases are not new, based on theory of quantum divergences.
Finally in Section \ref{Sec-6} some concluding remarks and open problems are in order. The paper contains
three appendices. Appendix \ref{Sec-A} is the proof of a main part of Theorem \ref{T-4.1}.
Appendix \ref{Sec-B} is the proof of Theorem \ref{T-5.3} and Appendix \ref{Sec-C} is a supplement to
Theorems \ref{T-5.3} and \ref{T-5.8}.

\section{Preliminaries}\label{Sec-2}

For each $n\in\bN$ we write $\bM_n$ for the $n\times n$ complex matrices. Let $\bM_n^+$ and
$\bM_n^{++}$ be the positive semidefinite $n\times n$ matrices and the positive definite $n\times n$
matrices, respectively. We often write $A\ge0$ for $A\in\bM_n^+$ and $A>0$ for $A\in\bM_n^{++}$ (for some
$n\in\bN$). The $n\times n$ identity matrix is denoted by $I_n$ or simply $I$. Let $\Tr$ be the usual trace on
$\bM_n$ and $\|X\|_\infty$ be the operator norm of $X\in\bM_n$. For $A\ge0$ we write $s(A)$ for the support
projection of $A$. We write $A^{-1}$ for the generalized inverse of $A$, i.e., the inverse of $A$ under the
restriction to the support of $A$. Moreover, for $r<0$ we define $A^r:=(A^{-1})^{-r}$ via the generalized inverse.

In this preliminary section we recall several examples of quasi matrix means and a few notions of matrix
orders, in particular, log-majorization. We first enumerate the definitions of quasi-extensions of several binary
matrix means for matrices. Let $0<\alpha<1$ and $p>0$, and let $A,B\in\bM_n^+$.

(i)\enspace
The \emph{quasi $\alpha$-weighted arithmetic mean} is
\[
\cA_{\alpha,p}(A,B):=(A^p\triangledown_\alpha B^p)^{1/p}=((1-\alpha)A^p+\alpha B^p)^{1/p},
\]
which is also called the ($\alpha$-weighted) \emph{matrix $p$-power mean}.

(ii)\enspace
The \emph{$\alpha$-weighted harmonic mean} of $A,B>0$ is
$A!_\alpha B:=((1-\alpha)A^{-1}+\alpha B^{-1})^{-1}$, extended to general $A,B\ge0$ as
$A!_\alpha B:=\lim_{\eps\searrow0}(A+\eps I)!_\alpha(B+\eps I)$. The
\emph{quasi $\alpha$-weighted harmonic mean} is
\[
\cH_{\alpha,p}(A,B):=(A^p!_\alpha B^p)^{1/p}.
\]

(iii)\enspace
The \emph{$\alpha$-geometric mean} of $A,B>0$ is
\begin{align}\label{F-2.1}
A\#_\alpha B:=A^{1/2}(A^{-1/2}BA^{-1/2})^\alpha A^{1/2},
\end{align}
which is extended to $A,B\ge0$ as $A\#_\alpha B:=\lim_{\eps\searrow0}(A+\eps I)\#_\alpha(B+\eps I)$.
The \emph{quasi $\alpha$-weighted geometric mean} is
\begin{align}\label{F-2.2}
G_{\alpha,p}(A,B):=(A^p\#_\alpha B^p)^{1/p}.
\end{align}

(iv)\enspace
The \emph{spectral geometric mean} of $A,B>0$ due to Fiedler and Pt\'ak \cite{FP} is
\[
F(A,B):=(A^{-1}\#B)^{1/2}A(A^{-1}\#B)^{1/2},
\]
which was extended to the $\alpha$-weighted version in \cite{LL} as
\begin{align}\label{F-2.3}
F_\alpha(A,B):=(A^{-1}\#B)^\alpha A(A^{-1}\#B)^\alpha.
\end{align}
The $\alpha$-weighted spectral geometric mean has recently been studied in \cite{KL,GK,GT,GH} where
$F_\alpha(A,B)$ is denoted by $A\natural_\alpha B$. The above definition of $F(A,B)$ is also meaningful
for $A,B\ge0$ with $s(A)\ge s(B)$ where $A^{-1}$ is the generalized inverse. The
\emph{quasi $\alpha$-weighted spectral geometric mean} of $A,B\ge0$ with $s(A)\ge s(B)$ is
\begin{align}\label{F-2.4}
SG_{\alpha,p}(A,B):=F_\alpha(A^p,B^p)^{1/p}.
\end{align}

(v)\enspace
Another weighted version of the spectral geometric mean of $A,B>0$ recently introduced in \cite{DTV} is
\begin{align}\label{F-2.5}
\widetilde F_\alpha(A,B):=(A^{-1}\#_\alpha B)^{1/2}A^{2(1-\alpha)}(A^{-1}\#_\alpha B)^{1/2},
\end{align}
whose quasi-extension is
\begin{align}\label{F-2.6}
\widetilde SG_{\alpha,p}(A,B):=\widetilde F_\alpha(A^p,B^p)^{1/p}.
\end{align}
These $\widetilde F_\alpha(A,B)$ and $\widetilde SG_{\alpha,p}(A,B)$ are meaningful for $A,B\ge0$ with
$s(A)\ge s(B)$ as well.

(vi)\enspace
There is one more familiar quasi matrix mean defined for all $A,B\ge0$ by
\begin{align}\label{F-2.7}
R_{\alpha,p}(A,B):=\bigl(A^{{1-\alpha\over2}p}B^{\alpha p}A^{{1-\alpha\over2}p}\bigr)^{1/p}.
\end{align}
This is called the \emph{R\'enyi mean} in \cite{DF} because $\Tr\,R_{\alpha,p}$ appears as the main
component in the definition of a certain quantum R\'enyi divergence; see Section \ref{Sec-5.2} for more
details.

(vii)\enspace
The \emph{Log-Euclidean mean} of $A,B>0$ is
\[
LE_\alpha(A,B):=\exp((1-\alpha)\log A+\alpha\log B),
\]
which is extended to general $A,B\ge0$ as
\begin{align}\label{F-2.8}
LE_\alpha(A,B):=P_0\exp\{(1-\alpha)P_0(\log A)P_0+\alpha P_0(\log B)P_0\},
\end{align}
where $P_0:=s(A)\wedge s(B)$. There is no quasi-extension of $LE_\alpha$ because of
$LE_\alpha(A^p,B^p)^{1/p}=LE_\alpha(A,B)$ for all $p>0$.

In this paper we will consider quasi-geometric type matrix means $G_{\alpha,p}$, $SG_{\alpha,p}$,
$\widetilde SG_{\alpha,p}$, $R_{\alpha,p}$ and $LE_\alpha$ for not only $0<\alpha<1$ but also $\alpha>1$.
The above definitions in \eqref{F-2.1}--\eqref{F-2.8} are all available even for any $\alpha>0$, $p>0$ and
for all $A,B\ge0$ with $s(A)\ge s(B)$ under conventions of $A^{-1}$ and $A^r$ for $r<0$ mentioned in the
beginning of this section. To be precise, we here fix the domains of those quasi-geometric type means as
follows: The domain of $G_{\alpha,p},R_{\alpha,p},LE_\alpha$ for $0<\alpha<1$ is
$\bigsqcup_{n\ge1}(\bM_n^+\times\bM_n^+)$. The domain of $G_{\alpha,p},R_{\alpha,p},LE_\alpha$ for
$\alpha>1$ and that of $SG_{\alpha,p},\widetilde SG_{\alpha,p}$ for all $\alpha\in(0,\infty)\setminus\{1\}$ are
\[
\bigsqcup_{n\ge1}\{(A,B)\in\bM_n^+\times\bM_n^+:s(A)\ge s(B)\}.
\]
In this way, although our quasi-geometric type means for $\alpha>1$ are a bit less meaningful as matrix
means than those for $0<\alpha<1$, we consider those for all $\alpha\in(0,\infty)\setminus\{1\}$ in this paper.

Several basic facts on the quasi matrix means defined in (i)--(vii) have been collected, though restricted to
the case $0<\alpha<1$, in Section 2.1 of our previous paper \cite{Hiai}. A few more facts are added in the
next proposition and theorem including the case $\alpha>1$.

\begin{proposition}\label{P-2.1}
Let $\alpha\in(0,\infty)\setminus\{1\}$ and $p>0$. Let $\cM_{\alpha,p}$ be any of $G_{\alpha,p}$,
$R_{\alpha,p}$, $SG_{\alpha,p}$, $\widetilde SG_{\alpha,p}$ and $LE_\alpha$.
\begin{itemize}
\item[(1)] $\cM_{\alpha,p}(A^{-1},B^{-1})=\cM_{\alpha,p}(A,B)^{-1}$ for all $A,B>0$.
\item[(2)] For every $(A,B)$ in the domain of $\cM_{\alpha,p}$ we have
$\cM_{\alpha,p}(A,B)=\lim_{\eps\searrow0}\cM_{\alpha,p}(A+\eps I,B+\eps I)$.
\end{itemize}
\end{proposition}

\begin{proof}
(1) is easily verified by definition of each $\cM_{\alpha,p}$.

(2)\enspace
For $0<\alpha<1$ the assertion was shown in \cite[Proposition 2.2]{Hiai}. Let $\alpha>1$ and $A,B\ge0$ with
$s(A)\ge s(B)$. Then the proof of \cite[Proposition 2.2]{Hiai} for $SG_{\alpha,p}$ and
$\widetilde SG_{\alpha,p}$ with $0<\alpha<1$ can work for any of $G_{\alpha,p}$, $R_{\alpha,p}$,
$SG_{\alpha,p}$, $\widetilde SG_{\alpha,p}$ with $\alpha>1$ as well. The proof for $LE_\alpha$ is similar to
that of \cite[Lemma 4.1]{HP}, while we here give a short proof using \cite[Lemma A.3]{Hiai}. Let $P_0:=s(B)$
and write for $0<\eps<1$,
\[
Z(\eps):=(1-\alpha)\log(A+\eps I)+\alpha\log(B+\eps I)
=\begin{bmatrix}Z_0(\eps)&Z_2(\eps)\\Z_2^*(\eps)&Z_1(\eps)\end{bmatrix},
\]
where $Z_0(\eps):=P_0Z(\eps)P_0$, $Z_1(\eps):=P_0^\perp Z(\eps)P_0^\perp$ and
$Z_2(\eps):=P_0Z(\eps)P_0^\perp$. Since
\begin{align*}
Z_0(\eps)&=(1-\alpha)P_0(\log(A+\eps I))P_0+\alpha P_0(\log(B+\eps I))P_0, \\
Z_1(\eps)&=(1-\alpha)P_0^\perp(\log(A+\eps I))P_0^\perp+\alpha(\log\eps)P_0^\perp, \\
Z_2(\eps)&=(1-\alpha)P_0(\log(A+\eps I))P_0^\perp,
\end{align*}
it is easy to see that
\begin{align*}
&Z_0(\eps)\to Z_0:=(1-\alpha)P_0(\log A)P_0+\alpha P_0(\log B)P_0\quad\mbox{as $\eps\searrow0$}, \\
&{1\over-\log\eps}\,Z_1(\eps)={\alpha-1\over\log\eps}\,P_0^\perp(\log(A+\eps I))P_0^\perp
-\alpha P_0^\perp\to-\alpha P_0^\perp\quad\mbox{as $\eps\searrow0$}, \\
&\sup\{\|Z_2(\eps)\|_\infty:0<\eps<1\}<\infty.
\end{align*}
Hence we can apply \cite[Lemma A.3]{Hiai} to $Z(\eps)$ with parameter $p:={1\over-\log\eps}$
($\searrow0$ as $\eps\searrow0$) to obtain
\begin{align*}
&LE_\alpha(A+\eps I,B+\eps I)=e^{Z(\eps)} \\
&\quad\to P_0e^{Z_0}=P_0\exp\{(1-\alpha)P_0(\log A)P_0+\alpha P_0(\log B)P_0\}=LE_\alpha(A,B),
\end{align*}
as desired.
\end{proof}

\begin{theorem}\label{T-2.2}
Let $\alpha\in(0,\infty)\setminus\{1\}$ and $p>0$. Let $\cM_{\alpha,p}$ be any of $G_{\alpha,p}$,
$R_{\alpha,p}$, $SG_{\alpha,p}$ and $\widetilde SG_{\alpha,p}$. For every $(A,B)$ in the domain of
$\cM_{\alpha,p}$ we have $LE_\alpha(A,B)=\lim_{p\searrow0}\cM_{\alpha,p}(A,B)$.
\end{theorem}

To prove the theorem we give a lemma. Note that for $0<\alpha<1$ this is contained in \cite[(A.25)]{Hiai}
where $A,B\ge0$ are general and $P_0:=s(A)\wedge s(B)$.

\begin{lemma}\label{L-2.3}
Let $\alpha,p$ be as in Theorem \ref{T-2.2}. For every $A,B\ge0$ with $s(A)\ge s(B)$ we have
\[
A^p\#_\alpha B^p=P_0+p\{(1-\alpha)P_0(\log A)P_0+\alpha P_0(\log B)P_0\}+o(p)
\quad\mbox{as $p\searrow0$},
\]
where $P_0:=s(B)$.
\end{lemma}

\begin{proof}
Let $L:=P_0(-\log A)P_0+P_0(\log B)P_0$ and set $Y(p):=(A^{-p/2}B^pA^{-p/2})^{1/p}-P_0e^LP_0$. Then by
\cite[(A.2)]{Hiai} we have $Y(p)\to0$ as $p\searrow0$. We write
\begin{equation}\label{F-2.9}
\begin{aligned}
A^{-p/2}B^pA^{-p/2}&=P_0(e^L+Y(p))^pP_0=P_0\bigl\{\exp\bigl[p\log(e^L+Y(p))\bigr]\bigr\}P_0 \\
&=P_0\bigl\{I+p\log(e^L+Y(p))+o(p)\bigr\}P_0=P_0+pL+o(p)\quad\mbox{as $p\searrow0$},
\end{aligned}
\end{equation}
where the last equality follows since Taylor's theorem (see, e.g., \cite[Theorem 2.3.1]{Hi}) gives
\[
\log(e^L+Y(p))=L+D(\log x)(e^L)(Y(p))+o(1)=L+o(1)
\]
with the Fr\'echet derivative $D(\log x)(e^L)$ of the functional calculus by $\log x$ ($x>0$) at $e^L$ (see
\cite[p.~159]{Hi}). Hence the Taylor expansion implies that
$(A^{-p/2}B^pA^{-p/2})^\alpha=P_0+\alpha pL+o(p)$ as $p\searrow0$. Moreover, note that
$A^p\#_\alpha B^p=P_0(A^p\#_\alpha B^p)P_0$. This is obvious for $0<\alpha<1$ (even for general
$A,B\ge0$). For $\alpha>1$ this can be seen from
\begin{align*}
A^p\#_\alpha B^p&=A^{p/2}(A^{-p/2}B^pA^{-p/2})^\alpha A^{p/2} \\
&=B^pA^{-p/2}(A^{-p/2}B^pA^{-p/2})^{\alpha-1}A^{p/2}
=A^{p/2}(A^{-p/2}B^pA^{-p/2})^{\alpha-1}A^{-p/2}B^p.
\end{align*}
Therefore, we have
\begin{align*}
A^p\#_\alpha B^p&=P_0\bigl\{A^{p/2}(A^{-p/2}B^pA^{-p/2})^\alpha A^{p/2}\bigr\}P_0 \\
&=P_0\Bigl(s(A)+{p\over2}\,s(A)\log A+o(p)\Bigr)(P_0+pL+o(p))
\Bigl(s(A)+{p\over2}\,s(A)\log A+o(p)\Bigr)P_0 \\
&=P_0+pP_0(\log A)P_0+\alpha pL+o(p) \\
&=P_0+p\{(1-\alpha)P_0(\log A)P_0+\alpha P_0(\log B)P_0\}+o(p)\quad\mbox{as $p\searrow0$},
\end{align*}
as asserted.
\end{proof}

\begin{proof}[Proof of Theorem \ref{T-2.2}]
For $0<\alpha<1$ the assertion was shown in \cite[Theorem 2.3]{Hiai}. For $\alpha>1$ let $A,B\ge0$ with
$s(A)\ge s(B)$. For $G_{\alpha,p}$ Lemma \ref{L-2.3} gives
$G_{\alpha,p}(A,B)^p=P_0+pK+o(p)$, where $K:=(1-\alpha)P_0(\log A)P_0+\alpha P_0(\log B)P_0$. This
implies that with $\cH_0$ being the range of $P_0$,
\[
\log G_{\alpha,p}(A,B)\big|_{\cH_0}={1\over p}\log\{P_0+pK+o(p)\}\big|_{\cH_0}=(K+o(1))|_{\cH_0},
\]
showing the result for $G_{\alpha,p}$. For $R_{\alpha,p}$ apply \eqref{F-2.9} to $A^{\alpha-1}$ and
$B^\alpha$ in place of $A,B$ to have
$A^{{1-\alpha\over2}p}B^{\alpha p}A^{{1-\alpha\over2}p}=P_0+pK+o(p)$. This implies that
$\log R_{\alpha,p}(A,B)\big|_{\cH_0}=(K+o(1))|_{\cH_0}$, showing the result for $R_{\alpha,p}$.
The proof for $SG_{\alpha,p}$ is similar to that of \cite[Theorem 2.3]{Hiai} for $0<\alpha<1$. Finally, for
$\widetilde SG_{\alpha,p}$ by Lemma \ref{L-2.3} we have $A^{-p}\#_\alpha B^p=P_0+p\widetilde K+o(p)$,
where $\widetilde K:=(1-\alpha)P_0(-\log A)P_0+\alpha P_0(\log B)P_0$. Therefore,
\begin{align*}
&(A^{-p}\#_\alpha B^p)^{1/2}A^{2(1-\alpha)p}(A^{-p}\#_\alpha B^p)^{1/2} \\
&\qquad=\Bigl(P_0+{p\over2}\,\widetilde K+o(p)\Bigr)(s(A)+2(1-\alpha)ps(A)\log A+o(p))
\Bigl(P_0+{p\over2}\,\widetilde K+o(p)\Bigr) \\
&\qquad=P_0+2(1-\alpha)pP_0(\log A)P_0+p\widetilde K+o(p)=P_0+pK+o(p),
\end{align*}
which shows the result for $\widetilde SG_{\alpha,p}$.
\end{proof}

In our previous paper \cite{Hiai} we discussed arithmetic-geometric type inequalities for quasi-geometric
type matrix means in several different matrix orderings varying from the strongest Loewner order to the
weakest order determined by trace inequality. But in the present paper we are mostly concerned with
log-majorizations for quasi-geometric type matrix means stated in (iii)--(vii) above.

Let $X,Y\in\bM_n^+$. The most standard and the strongest order between $X,Y$ is the the
\emph{Loewner order} $X\le Y$, i.e., $Y-X\ge0$. Let $\lambda(X)=(\lambda_1(X),\dots,\lambda_n(X))$ be
the eigenvalues of $X$ in decreasing order with multiplicities. The \emph{entrywise eigenvalue order}
denoted as $X\le_\lambda Y$ is defined if $\lambda_i(X)\le\lambda_i(Y)$ for each $i=1,\dots,n$. The
\emph{weak log-majorization} $X\prec_{w\log}Y$ means that
\[
\prod_{i=1}^k\lambda_i(X)\le\prod_{i=1}^k\lambda_i(Y),\qquad1\le k\le n,
\]
and the \emph{log-majorization} $X\prec_{\log}Y$ means that $X\prec_{w\log}Y$ and
$\prod_{i=1}^n\lambda_i(X)=\prod_{i=1}^n\lambda_i(Y)$, i.e., $\det X=\det Y$. Details on
(weak) log-majorization are found in \cite{An2}, \cite[Chap.~II]{Bh} and \cite{MOA}. Some basic properties of
several matrix orderings including $\le,\le_\lambda,\prec_{(w)\log}$ mentioned above are also found in
\cite[Sec.~2.2]{Hiai}.

Concerning the quasi-geometric type matrix means in (iii)--(vii) above, some general facts are in order.

\begin{remark}\label{R-2.4}\rm
(1)\enspace
Let $(\cM,\cN)$ be any pair from $\{R,G,SG,\widetilde SG,LE\}$, and let $\alpha,\beta\in(0,\infty)\setminus\{1\}$
and $p,q>0$. When $\alpha\ne\beta$, $\cM_{\alpha,p}$ and $\cN_{\beta,q}$ are not definitively comparable
even for positive scalar variables, as explained in \cite[Remark 2.7(1)]{Hiai}. So we only compare between
$\cM_{\alpha,p}$ and $\cN_{\alpha,q}$ under the same weight parameter $\alpha$.

(2)\enspace
Let $\cM_{\alpha,p}$ and $\cN_{\alpha,q}$ be as in (1), and let $A,B>0$. Since
\[
\det\cM_{\alpha,p}(A,B)=\det\cN_{\alpha,q}(A,B)=(\det A)^{1-\alpha}(\det B)^\alpha
\]
independently of $p,q>0$, we have $\cM_{\alpha,p}(A,B)\prec_{\log}\cN_{\alpha,q}$ as long as
$\cM_{\alpha,p}(A,B)\prec_{w\log}\cN_{\alpha,q}(A,B)$. Furthermore, if
$\cM_{\alpha,p}(A,B)\le_\lambda\cN_{\alpha,q}(A,B)$, then we must have
$\lambda(\cM_{\alpha,p}(A,B))=\lambda(\cN_{\alpha,q}(A,B))$. These facts are reasons why we consider
only the log-majorization $\cM_{\alpha,p}\prec_{\log}\cN_{\alpha,q}$ in this paper.

(3)\enspace
For any $p>0$ and $A,B\ge0$ we note that
\[
\lambda(R_{1/2,2p}(A,B))=\lambda^{1/p}(A^{p/2}B^pA^{p/2})
=\lambda^{1/p}(B^{p/2}A^pB^{p/2})=\lambda(SG_{1/2,p}(A,B))
\]
thanks to \cite[Theorem 3.2, Item 8]{FP} and that $SG_{1/2,p}(A,B)=\widetilde SG_{1/2,p}(A,B)$ by definitions.
Moreover, for $A,B\ge0$ with $s(A)\ge s(B)$ we have
\[
\lambda(R_{2,p}(A,B))=\lambda^{1/p}(A^{-p/2}B^{2p}A^{-p/2})
=\lambda^{1/p}(B^pA^{-p}B^p)=\lambda(G_{2,p}(A,B)).
\]
Hence the three pairs $(R_{1/2,2p},SG_{1/2,p})$, $(R_{2,p},G_{2,p})$ and
$(SG_{1/2,p},\widetilde SG_{1/2,p})$ are rather trivial and exceptional cases in our study.

(4)\enspace
It is simple to see that if $X,Y\in\bM_n^+$ and $\det X=\det Y$, then
$X\prec_{\log}Y$$\iff$$\lambda_1(X)\le\lambda_1(Y)$$\iff$$\Tr\,X\le\Tr\,Y$. Therefore, for any pair
$(\cM_{\alpha,p},\cN_{\alpha,q})$ as in (1), if $(A,B)\in\bM_2^+\times\bM_2^+$ is in the common
domain of $\cM_{\alpha,p}$ and $\cN_{\alpha,q}$, then the following conditions are equivalent:
(a) $\cM_{\alpha,p}(A,B)\prec_{\log}\cN_{\alpha,q}(A,B)$;
(b) $\lambda_1(\cM_{\alpha,p}(A,B))\le\lambda_1(\cN_{\alpha,q}(A,B))$;
(c) $\Tr\,\cM_{\alpha,p}(A,B)\le\Tr\,\cN_{\alpha,q}(A,B)$. This simplifies the log-majorization between
$\cM_{\alpha,p}$ and $\cN_{\alpha,q}$ when restricted to $2\times2$ matrices.

(5)\enspace
For any pair $(\cM_{\alpha,p},\cN_{\alpha,q})$ except for the three cases mentioned in (3), note that there
exist $A,B\in\bM_2^{++}$ such that $\cM_{\alpha,p}(A,B)\not\le_\lambda\cN_{\alpha,q}(A,B)$. Indeed,
assume that this is not the case. Then by (2) we must have
$\lambda(\cM_{\alpha,p}(A,B))=\lambda(\cN_{\alpha,q}(A,B))$ for all $A,B\in\bM_2^{++}$. But, as verified
below, this fails to hold except for the above three cases and the trivial case of $\cM=\cN$ and $p=q$.
\end{remark}

For $x,y>0$ and $\theta\in\bR$ we define $2\times2$ positive definite matrices by
\begin{align}\label{F-2.10}
A_0:=\begin{bmatrix}1&0\\0&x\end{bmatrix},\qquad
B_\theta:=\begin{bmatrix}\cos\theta&-\sin\theta\\\sin\theta&\cos\theta\end{bmatrix}
\begin{bmatrix}1&0\\0&y\end{bmatrix}
\begin{bmatrix}\cos\theta&\sin\theta\\-\sin\theta&\cos\theta\end{bmatrix}.
\end{align}
These $A_0,B_\theta$ were repeatedly utilized in \cite{Hiai} and are also useful in the present paper. The
next lemma is used to verify the assertion of Remark \ref{R-2.4}(5), and it will be also used in several places
in Section \ref{Sec-3}. 

\begin{lemma}\label{L-2.5}
Let $\alpha\in(0,\infty)\setminus\{1\}$ and $p>0$. Let $A_0,B_\theta\in\bM_2^{++}$ be given in \eqref{F-2.10}
with $y=x\in(0,1)$. Then we have
\begin{align}
\lambda_1(R_{\alpha,p}(A_0,B_\theta))
&=1+\theta^2\,{-1-x^p+x^{\alpha p}+x^{(1-\alpha)p}\over p(1-x^p)}+o(\theta^2), \label{F-2.11}\\
\lambda_1(G_{\alpha,p}(A_0,B_\theta))
&=1+\theta^2\,{\alpha(1-\alpha)\over2p}(x^p-x^{-p})+o(\theta^2), \label{F-2.12}\\
\lambda_1(SG_{\alpha,p}(A_0,B_\theta))
&=1-\theta^2\,{2\alpha(1-\alpha)\over p}\cdot{1-x^p\over1+x^p}+o(\theta^2), \label{F-2.13}\\
\lambda_1(\widetilde SG_{\alpha,p}(A_0,B_\theta))
&=1-{\theta^2\over p}\biggl\{\alpha\,{1-x^p\over1+x^p}
+{x^p+x^{2p}-x^{(2\alpha+1)p}-x^{2(1-\alpha)p}\over(1-x^p)(1+x^p)^2}\biggr\}+o(\theta^2), \label{F-2.14}\\
\lambda_1(LE_\alpha(A_0,B_\theta))
&=1+\theta^2\alpha(1-\alpha)\log x+o(\theta^2). \label{F-2.15}
\end{align}
\end{lemma}

\begin{proof}
We can easily compute all the expressions in the lemma by applying \cite[Lemma 3.6]{Hiai} to the relevant
situations discussed in \cite{Hiai}. See \cite[Lemmas 3.4, 3.5 and 4.17]{Hiai} (with $y=x$) for \eqref{F-2.15},
\eqref{F-2.11} and \eqref{F-2.12} respectively, and see the proofs of \cite[Theorems 4.29 and 4.37]{Hiai} for
\eqref{F-2.13} and \eqref{F-2.14} respectively, while details of computations are left to the reader. Also we
note that all the relevant expressions in \cite{Hiai} for $0<\alpha<1$ remain valid for all
$\alpha\in(0,\infty)\setminus\{1\}$.
\end{proof}

To verify Remark \ref{R-2.4}(5) more explicitly, we here include brief discussions for completeness.
Consider $A_0,B_\theta\in\bM_2^{++}$ as in Lemma \ref{L-2.5}. First let $(\cM_{\alpha,p},\cM_{\alpha,q})$
be given for $\cM\in\{R,G,SG,\widetilde SG\}$. Assume that
$\lambda_1(\cM_{\alpha,p}(A_0,B_\theta))=\lambda_1(\cM_{\alpha,q}(A_0,B_\theta))$ for all $x\in(0,1)$,
which implies from \eqref{F-2.11}--\eqref{F-2.15} that ${-1-x^p+x^{\alpha p}+x^{(1-\alpha)p}\over p(1-x^p)}=
{-1-x^q+x^{\alpha q}+x^{(1-\alpha)q}\over q(1-x^q)}$ etc. Letting $x\searrow0$ gives $p=q$ in each case.
Next let $(\cM_{\alpha,p},\cN_{\alpha,q})$ be given for $\cM\ne\cN$, and assume that
$\lambda_1(\cM_{\alpha,p}(A_0,B_\theta))=\lambda_1(\cN_{\alpha,q}(A_0,B_\theta))$ for all $x\in(0,1)$.
When $(\cM_{\alpha,p},\cN_{\alpha,q})=(R_{\alpha,p},G_{\alpha,q})$, by \eqref{F-2.11} and \eqref{F-2.12}
we have
\begin{align}\label{F-2.16}
{-1-x^p+x^{\alpha p}+x^{(1-\alpha)p}\over p(1-x^p)}
={\alpha(1-\alpha)\over2q}(x^q-x^{-q}),\qquad x\in(0,1).
\end{align}
For $0<\alpha<1$, letting $x\searrow0$ gives $-{1\over p}=-\infty$, a contradiction. For $\alpha>1$, we have
$(1-\alpha)p=-q$ and ${1\over p}={\alpha(\alpha-1)\over2q}$ by comparing the leading terms as $x\searrow0$
of both sides of \eqref{F-2.16}. Hence $\alpha=2$ and $p=q$, which is a case excluded. For other pairs the
discussions are more or less similar, so we omit the details.

\section{Log-majorizations}\label{Sec-3}

Let us begin by surveying the log-majorizations known so far between quasi-geometric type means
$R_{\alpha,p}$, $G_{\alpha,p}$ and $S G_{\alpha,p}$ in the following theorem.

\begin{theorem}\label{T-3.1}
Let $\alpha\in(0,\infty)\setminus\{1\}$, $p,q>0$ and $A,B>0$.
\begin{itemize}
\item[(a)] $R_{\alpha,p}(A,B)\prec_{\log}R_{\alpha,q}(A,B)$ for any $\alpha\in(0,\infty)\setminus\{1\}$ if
$p\le q$.
\item[(b)] $G_{\alpha,q}(A,B)\prec_{\log}G_{\alpha,p}(A,B)$ if $0<\alpha<1$ and $p\le q$.
\item[(c)] $G_{\alpha,p}(A,B)\prec_{\log}G_{\alpha,q}(A,B)$ if $1<\alpha\le2$ and $p\le q$.
\item[(d)] $G_{\alpha,p}(A,B)\prec_{\log}G_{\alpha,q}(A,B)$ if $\alpha\ge2$ and
$p/q\le{\alpha\over2(\alpha-1)}$.
\item[(e)] $G_{\alpha,p}(A,B)\prec_{\log}R_{\alpha,q}(A,B)$ either if $0<\alpha<1$ and $p,q>0$ are arbitrary,
or if $\alpha>1$ and $p/q\le\min\{\alpha/2,\alpha-1\}$.
\item[(f)] $R_{\alpha,q}(A,B)\prec_{\log}G_{\alpha,p}(A,B)$ if $\alpha>1$ and
$p/q\ge\max\{\alpha/2,\alpha-1\}$.
\item[(g)] $R_{\alpha,q}(A,B)\prec_{\log}SG_{\alpha,p}(A,B)$ if $0<\alpha<1$ and
$p/q\ge\max\{\alpha,1-\alpha\}$.
\item[(h)] $SG_{\alpha,p}(A,B)\prec_{\log}R_{\alpha,q}(A,B)$ either if $1<\alpha\le3/2$ and $p=q$,
or if $\alpha\in(0,1/2]\cup(1,2]$ and $p/q=\alpha$.
\end{itemize}
\end{theorem}

Indeed, (a) and (b) are the log-majorizations due to Araki \cite{Ar} and Ando and Hiai \cite[Theorem 2.1]{AH}),
respectively. In \cite{KS} the matrix perspective $G_\beta(A,B):=A^{1/2}(A^{-1/2}BA^{-1/2})^\beta A^{1/2}$ for
$-1\le\beta<0$ was treated, where $G_\beta(A,B)$ is denoted by $A\natural_\beta B$. Since
$G_\beta(A,B)=G_{1-\beta}(B,A)$ and $1<1-\beta\le2$, the log-majorization in (c) is equivalent to that shown
by Kian and Seo \cite[Theorem 3.1]{KS}. (d) was shown in \cite[Corollary 5.2]{Hi2}, and (e) and (f) were
shown in \cite[Proposition 5.1(a), (b)]{Hi2}. (g) is a rewriting of the log-majorization due to Gan and Tam
\cite[Theorem 3.7]{GT}. Finally, (h) is rewriting of the log-majorizations shown by Furuichi and Seo
\cite[Theorems 4.3 and 4.8]{FS}. By Proposition \ref{P-2.1}(2) note that the log-majorizations in the
$0<\alpha<1$ case of (a), (b), (e) hold for all $A,B\ge0$, and those in the $\alpha>1$ case of
(a), (c), (d), (e), (f) as well as those in (g), (h) hold for $A,B\ge0$ with $s(A)\ge s(B)$.

In this section we will consider log-majorizations $\cM_{\alpha,p}\prec_{\log}\cN_{\alpha,q}$ for pairs
$(\cM_{\alpha,p},\cN_{\alpha,q})$ of quasi-geometric type matrix means other than the above known cases.
Our discussions will be divided into four subsections.


\subsection{$SG_{\alpha,p}\prec_{\log}R_{\alpha,q}$ and
$R_{\alpha,q}\prec_{\log}SG_{\alpha,p}$}\label{Sec-3.1}

First we recall that the log-majorizations in (e) and (f) of Theorem \ref{T-3.1} were indeed presented in
\cite{Hi2} in the `if and only if' statement. For the convenience to make explicit comparison with the
log-majorizations shown in Section \ref{Sec-3} below, we include \cite[Proposition 5.1]{Hi2} in its complete
form in the following:

\begin{theorem}\label{T-3.2}
Let $\alpha\in(0,\infty)\setminus\{1\}$ and $p,q>0$.
\begin{itemize}
\item[(1)] The following conditions are equivalent:
\begin{itemize}
\item[(i)] $G_{\alpha,p}(A,B)\prec_{\log}R_{\alpha,q}(A,B)$ for all $A,B\ge0$ with $s(A)\ge s(B)$;
\item[(ii)] $G_{\alpha,p}(A,B)\prec_{\log}R_{\alpha,q}(A,B)$ for all $A,B\in\bM_2^{++}$;
\item[(iii)] either $0<\alpha<1$ and $p,q>0$ are arbitrary, or $\alpha>1$ and $p/q\le\min\{\alpha/2,\alpha-1\}$.
\end{itemize}
\item[(2)] The following conditions are equivalent:
\begin{itemize}
\item[(i)] $R_{\alpha,q}(A,B)\prec_{\log}G_{\alpha,p}(A,B)$ for all $A,B\ge0$ with $s(A)\ge s(B)$;
\item[(ii)] $R_{\alpha,q}(A,B)\prec_{\log}G_{\alpha,p}(A,B)$ for all $A,B\in\bM_2^{++}$;
\item[(iii)] $\alpha>1$ and $p/q\ge\max\{\alpha/2,\alpha-1\}$.
\end{itemize}
\end{itemize}
\end{theorem}

When $0<\alpha<1$, the next theorem gives the necessary and sufficient condition on $p,q$ for
$SG_{\alpha,p}(A,B)\prec_{\log}R_{\alpha,q}(A,B)$ (resp., $R_{\alpha,q}(A,B)\prec_{\log}SG_{\alpha,p}(A,B)$)
to hold for all $A,B\ge0$ with $s(A)\ge s(B)$, thus strengthening the assertion in (h) and (g) above.

\begin{theorem}\label{T-3.3}
Let $0<\alpha<1$ and $p,q>0$.
\begin{itemize}
\item[(1)] The following conditions are equivalent:
\begin{itemize}
\item[(i)] $SG_{\alpha,p}(A,B)\prec_{\log}R_{\alpha,q}(A,B)$ for all $A,B\ge0$ with $s(A)\ge s(B)$;
\item[(ii)] $SG_{\alpha,p}(A,B)\prec_{\log}R_{\alpha,q}(A,B)$ for all $A,B\in\bM_2^{++}$;
\item[(iii)] $p/q\le\min\{\alpha,1-\alpha\}$.
\end{itemize}
\item[(2)] The following conditions are equivalent:
\begin{itemize}
\item[(i)] $R_{\alpha,q}(A,B)\prec_{\log}SG_{\alpha,p}(A,B)$ for all $A,B\ge0$ with $s(A)\ge s(B)$;
\item[(ii)] $R_{\alpha,q}(A,B)\prec_{\log}SG_{\alpha,p}(A,B)$ for all $A,B\in\bM_2^{++}$;
\item[(iii)] $p/q\ge\max\{\alpha,1-\alpha\}$.
\end{itemize}
\end{itemize}
\end{theorem}

\begin{proof}
(1)\enspace
(i)$\implies$(ii) is trivial.

(ii)$\implies$(iii).\enspace
For $A,B>0$ set $Y:=A^p$ and $X:=A^{-p}\#B^p$. Since $X=Y^{-1}\#B^p$ and hence $B^p=XYX$ by the
Riccati lemma, it follows that $SG_{\alpha,p}(A,B)\prec_{\log}R_{\alpha,q}(A,B)$ is equivalent to
\[
(X^\alpha YX^\alpha)^r\prec_{\log}Y^{{1-\alpha\over2}r}(XYX)^{\alpha r}Y^{{1-\alpha\over2}r},
\]
where $r:=q/p$. Conversely, for any $X,Y>0$, setting $A:=Y^{1/p}$ and $B:=(XYX)^{1/p}$ we have $Y=A^p$
and $X=A^{-p}\#B^p$. Hence in view of Remark \ref{R-2.4}(4) we see that condition (ii) is
equivalent to saying that
\begin{align}\label{F-3.1}
\lambda_1^r(X^\alpha YX^\alpha)
\le\lambda_1\bigl(Y^{{1-\alpha\over2}r}(XYX)^{\alpha r}Y^{{1-\alpha\over2}r}\bigr)
\end{align}
for all $X,Y\in\bM_2^{++}$.

Now let $X:=A_0$ and $Y:=B_\theta$ in \eqref{F-2.10} for $x,y>0$ with $x^2y\ne1$ and $x^{2\alpha}y\ne1$,
and argue as in the proof of \cite[Theorem 4.23]{Hiai}. Since
\begin{align}\label{F-3.2}
XYX=\begin{bmatrix}1-\theta^2(1-y)&\theta x(1-y)\\\theta x(1-y)&x^2y+\theta^2x^2(1-y)\end{bmatrix}
+o(\theta^2)\quad\mbox{as $\theta\to0$},
\end{align}
one can apply \cite[Example 3.2(1)]{Hiai} to compute
\[
(XYX)^{\alpha r}=\begin{bmatrix}1+\theta^2a_{11}&\theta a_{12}\\
\theta a_{12}&x^{2\alpha r}y^{\alpha r}+\theta^2a_{22}\end{bmatrix}+o(\theta^2),
\]
where
\[
\begin{cases}
a_{11}:=-\alpha r(1-y)+{x^2(1-y)^2(\alpha r-1-\alpha rx^2y+x^{2\alpha r}y^{\alpha r})\over(1-x^2y)^2}, \\
a_{12}:={x(1-y)(1-x^{2\alpha r}y^{\alpha r})\over1-x^2y}.
\end{cases}
\]
(The explicit form of $a_{22}$ is unnecessary below.) Since
\[
Y^{{1-\alpha\over2}r}=\begin{bmatrix}1-\theta^2\bigl(1-y^{{1-\alpha\over2}r}\bigr)&
\theta\bigl(1-y^{{1-\alpha\over2}r}\bigr)\\\theta\bigl(1-y^{{1-\alpha\over2}r}\bigr)&
y^{{1-\alpha\over2}r}+\theta^2\bigl(1-y^{{1-\alpha\over2}r}\bigr)\end{bmatrix}
+o(\theta^2),
\]
one further computes
\begin{align}\label{F-3.3}
Y^{{1-\alpha\over2}r}(XYX)^{\alpha r}Y^{{1-\alpha\over2}r}
=\begin{bmatrix}1+\theta^2b_{11}&\theta b_{12}\\
\theta b_{12}&x^{2\alpha r}y^r+\theta^2b_{22}\end{bmatrix}+o(\theta^2),
\end{align}
where
\[
\begin{cases}
b_{11}:=-2\bigl(1-y^{{1-\alpha\over2}r}\bigr)+x^{2\alpha r}y^{\alpha r}\bigl(1-y^{{1-\alpha\over2}r}\bigr)^2
+a_{11}+2\bigl(1-y^{{1-\alpha\over2}r}\bigr)a_{12}, \\
b_{12}:=1-y^{{1-\alpha\over2}r}+x^{2\alpha r}\bigl(y^{{1+\alpha\over2}r}-y^r\bigr)
+y^{{1-\alpha\over2}r}a_{12}.
\end{cases}
\]
On the other hand, one has
\begin{align}\label{F-3.4}
X^\alpha YX^\alpha=\begin{bmatrix}1-\theta^2(1-y)&\theta x^\alpha(1-y)\\
\theta x^\alpha(1-y)&x^{2\alpha}y+\theta^2x^{2\alpha}(1-y)\end{bmatrix}+o(\theta^2),
\end{align}
When $x^{2\alpha}y<1$, we apply \cite[Lemma 3.6]{Hiai} to \eqref{F-3.4} and \eqref{F-3.3} to obtain
\begin{align}
\lambda_1^r(X^\alpha YX^\alpha)
&=1+\theta^2r\biggl(-1+y+{x^{2\alpha}(1-y)^2\over1-x^{2\alpha}y}\biggr)+o(\theta^2), \label{F-3.5}\\
\lambda_1\bigl(Y^{{1-\alpha\over2}r}(XYX)^{\alpha r}Y^{{1-\alpha\over2}r}\bigr)
&=1+\theta^2\biggl(b_{11}+{b_{12}^2\over1-x^{2\alpha r}y^r}\biggr)+o(\theta^2). \nonumber
\end{align}
Hence, for any $y>0$, whenever $x>0$ is sufficiently small, it must follows from \eqref{F-3.1} that
\[
r\biggl(-1+y+{x^{2\alpha}(1-y)^2\over1-x^{2\alpha}y}\biggr)
\le b_{11}+{b_{12}^2\over1-x^{2\alpha r}y^r}.
\]
As $x\searrow0$, since
\begin{align*}
&a_{11}\to-\alpha r(1-y),\qquad a_{12}\to0, \\
&b_{11}\to-2\bigl(1-y^{{1-\alpha\over2}r}\bigr)-\alpha r(1-y),\qquad b_{12}\to1-y^{{1-\alpha\over2}r},
\end{align*}
we must have
\begin{align}\label{F-3.6}
-r+ry\le-2\bigl(1-y^{{1-\alpha\over2}r}\bigr)-\alpha r(1-y)+\bigl(1-y^{{1-\alpha\over2}r}\bigr)^2,
\qquad y>0.
\end{align}
Letting $y\searrow0$ gives $-r\le-1-\alpha r$ so that $1\le(1-\alpha)r$, i.e., $p/q\le1-\alpha$. Since
$R_{\alpha,p}(A,B)=R_{1-\alpha,p}(B,A)$ and $SG_{\alpha,p}(A,B)=SG_{1-\alpha,p}(A,B)$ (see
\cite[Proposition 4.2(iii)]{LL}) for all $A,B>0$, note that (ii) implies condition (ii) with $1-\alpha$ in place of
$\alpha$. Hence we have $p/q\le\alpha$ too so that (iii) follows.

(iii)$\implies$(i).\enspace
This part was shown in \cite[Proposition 4.25(1)]{Hiai}, while we give a proof in a different way here. Since
$R_{\alpha,q}(A,B)=R_{\alpha,1}(A^q,B^q)^{1/q}$ and
$SG_{\alpha,p}(A,B)=SG_{\alpha,p/q}(A^q,B^q)^{1/q}$, we may assume that $q=1$. Moreover, by continuity
(\cite[Proposition 2.2]{Hiai}, Proposition \ref{P-2.1}(2)) we may assume that $A,B>0$. Note that
$\det SG_{\alpha,p}(A,B)=(\det A)^{1-\alpha}(\det B)^\alpha=\det R_{\alpha,1}(A,B)$. Hence, based on the
anti-symmetric tensor power technique, it suffices to show that
\[
\|SG_{\alpha,p}(A,B)\|_\infty\le\|R_{\alpha,1}(A,B)\|_\infty,
\]
or equivalently,
\begin{align}\label{F-3.7}
B^\alpha\le A^{\alpha-1}\implies SG_{\alpha,p}(A,B)\le I.
\end{align}
Now, we divide the proof into the two cases of $0<\alpha\le1/2$ and $1/2\le\alpha<1$.

{\it Case $0<\alpha\le1/2$}.\enspace
Since ${p\over\alpha}\le1$ by assumption, we have $B^p\le A^{{\alpha-1\over\alpha}p}$ and hence
\[
A^{-p}\#B^p\le A^{-p}\#A^{{\alpha-1\over\alpha}p}=A^{-{p\over2\alpha}}.
\]
Since $2\alpha\le1$, we have $(A^{-p}\#B^p)^{2\alpha}\le A^{-p}$ so that
$A^{p/2}(A^{-p}\#B^p)^{2\alpha}A^{p/2}\le I$. This gives $(A^{-p}\#B^p)^\alpha A^p(A^{-p}\#B^p)^\alpha\le I$,
showing \eqref{F-3.7}.

{\it Case $1/2\le\alpha<1$}.\enspace
Since ${p\over1-\alpha}\le1$, we have $B^{{\alpha\over1-\alpha}p}\le A^{-p}$ so that
$A^p\le B^{-{\alpha\over1-\alpha}p}$. Hence
\[
B^{-p}\#A^p\le B^{-p}\#B^{-{\alpha\over1-\alpha}p}=B^{-{p\over2(1-\alpha)}}.
\]
Since $2(1-\alpha)\le1$, we have $(B^{-p}\#A^p)^{2(1-\alpha)}\le B^{-p}$ so that
$SG_{\alpha,p}(A,B)=SG_{1-\alpha,p}(B,A)\le I$ as in the case $0<\alpha\le1/2$.

(2)\enspace
(iii)$\implies$(i) was shown in \cite[Theorem 3.7]{GT}, as stated in (g) above. (i)$\implies$(ii) is trivial, and
(ii)$\implies$(iii) can be similarly proved by just reversing inequalities in the above proof of (ii)$\implies$(iii)
of (1).
\end{proof}

When $\alpha>1$, the next theorem says that we have a sufficient condition for
$SG_{\alpha,p}\prec_{\log}R_{\alpha,q}$ and that $R_{\alpha,q}\prec_{\log}SG_{\alpha,p}$ fails to hold for
any $p,q>0$.

\begin{theorem}\label{T-3.4}
Let $\alpha>1$ and $p,q>0$.
\begin{itemize}
\item[(1)] If $p/q\le\alpha$, then $SG_{\alpha,p}(A,B)\prec_{\log}R_{\alpha,q}(A,B)$ holds for all $A,B\ge0$
with $s(A)\ge s(B)$.
\item[(2)] For any $\alpha>1$ and any $p,q>0$ there exist $A,B\in\bM_2^{++}$ such that
$R_{\alpha,q}(A,B)\not\prec_{\log}SG_{\alpha,p}(A,B)$.
\end{itemize}
\end{theorem}

\begin{proof}
(1)\enspace
Similarly to the argument in the first part of the proof (ii)$\implies$(iii) of Theorem \ref{T-3.3}(1) and by
continuity we see that the assertion to prove is equivalent to
\[
(X^\alpha YX^\alpha)^r\prec_{\log}Y^{{1-\alpha\over2}r}(XYX)^{\alpha r}Y^{{1-\alpha\over2}r}
\]
for all $X,Y>0$, where $r:=q/p$. To show this, it suffices to prove that
\begin{align}\label{F-3.8}
(XYX)^{\alpha r}\le Y^{(\alpha-1)r}\implies X^\alpha YX^\alpha\le I.
\end{align}
Now assume that $p/q\le\alpha$ and so ${1\over\alpha r}\le1$. Then the inequality in the LHS of
\eqref{F-3.8} implies that $XYX\le Y^{\alpha-1\over\alpha}$ and hence
\[
X^\alpha YX^\alpha=X^{\alpha-1}(XYX)X^{\alpha-1}\le X^{\alpha-1}Y^{\alpha-1\over\alpha}X^{\alpha-1}.
\]
Since $0<{\alpha-1\over\alpha}<1$, Araki's log-majorization (Theorem \ref{T-3.1}(a)) gives
\[
X^{\alpha-1}Y^{\alpha-1\over\alpha}X^{\alpha-1}\prec_{\log}(X^\alpha YX^\alpha)^{\alpha-1\over\alpha}.
\]
Therefore we have
\[
\|X^\alpha YX^\alpha\|_\infty\le\big\|(X^\alpha YX^\alpha)^{\alpha-1\over\alpha}\big\|_\infty
=\|X^\alpha YX^\alpha\|_\infty^{\alpha-1\over\alpha},
\]
which implies that $\|X^\alpha YX^\alpha\|_\infty\le1$, showing \eqref{F-3.8}.

(2)\enspace
Assume by contradiction that $R_{\alpha,p}(A,B)\prec_{\log}SG_{\alpha,q}(A,B)$ holds for all
$A,B\in\bM_2^{++}$. Then we must have the reversed inequality of \eqref{F-3.1}, so that the computations
in the proof (ii)$\implies$(iii) of Theorem \ref{T-3.3}(1) providing \eqref{F-3.6} are all valid with inequalities
reversed in the present setting of $\alpha>1$. Hence from the reversed inequality of \eqref{F-3.6} we have
\begin{align}\label{F-3.9}
y^{(1-\alpha)r}-(1-\alpha)r(y-1)-1\le0,\qquad y>0.
\end{align}
But this inequality is impossible since the above LHS goes to $+\infty$ as $y\searrow0$, so the result has
been shown.
\end{proof}

\begin{problem}\label{Q-3.5}\rm
When $\alpha>1$, no necessary condition on $p,q>0$ for $SG_{\alpha,p}\prec_{\log}R_{\alpha,q}$ to
hold is known. In fact, when $SG_{\alpha,p}(A,B)\prec_{\log}R_{\alpha,q}(A,B)$ holds for all
$A,B\in\bM_2^{++}$, one has the reversed inequality of \eqref{F-3.9} but nothing follows from that. This
suggests us that $SG_{\alpha,p}\prec_{\log}R_{\alpha,q}$ might hold for any $p,q>0$.
\end{problem}

\subsection{$\widetilde SG_{\alpha,p}\prec_{\log}R_{\alpha,q}$ and
$R_{\alpha,q}\prec_{\log}\widetilde SG_{\alpha,p}$}\label{Sec-3.2}

The next theorem is concerned with the log-majorizations between $\widetilde SG_{\alpha,p}$ and
$R_{\alpha,q}$ when $0<\alpha<1$.

\begin{theorem}\label{T-3.6}
Let $0<\alpha<1$ and $p,q>0$.
\begin{itemize}
\item[(1)] The following conditions are equivalent:
\begin{itemize}
\item[(i)] $\widetilde SG_{\alpha,p}(A,B)\prec_{\log}R_{\alpha,q}(A,B)$ for all $A,B\ge0$ with $s(A)\ge s(B)$;
\item[(ii)] $\widetilde SG_{\alpha,p}(A,B)\prec_{\log}R_{\alpha,q}(A,B)$ for all $A,B\in\bM_2^{++}$;
\item[(iii)] $p/q\le\alpha$.
\end{itemize}
\item[(2)] If $\alpha\le1/2$ and $q\le p$, then $R_{\alpha,q}(A,B)\prec_{\log}\widetilde SG_{\alpha,p}(A,B)$
holds for all $A,B\ge0$ with $s(A)\ge s(B)$.
\item[(3)] Assume that $R_{\alpha,q}(A,B)\prec_{\log}\widetilde SG_{\alpha,p}(A,B)$ holds for all
$A,B\in\bM_2^{++}$. Then we have $\alpha\le1/2$ and $p/q\ge1/2$.
\end{itemize}
\end{theorem}

\begin{proof}
(1)\enspace
(i)$\implies$(ii) is trivial.

(ii)$\implies$(iii).\enspace
For $A,B>0$ set $Y:=A^p$ and $X:=A^{-p}\#_\alpha B^p=Y^{-1}\#_\alpha B^p$; then
$B^p=Y^{-1}\#_{1/\alpha}X$ (see definition \eqref{F-2.1} though $1/\alpha>1$). Then
$\widetilde SG_{\alpha,p}(A,B)\prec_{\log}R_{\alpha,q}(A,B)$ is equivalent to
\[
(X^{1/2}Y^{2(1-\alpha)}X^{1/2})^r\prec_{\log}
Y^{{1-\alpha\over2}r}(Y^{-1}\#_{1/\alpha}X)^{\alpha r}Y^{{1-\alpha\over2}r},
\]
where $r:=q/p$. Hence, similarly to the first paragraph of the proof (ii)$\implies$(iii) of Theorem \ref{T-3.3}(1),
we see that condition (ii) is equivalent to saying that
\begin{align}\label{F-3.10}
\lambda_1^r(Y^{1-\alpha}XY^{1-\alpha})
\le\lambda_1\bigl(Y^{{1-\alpha\over2}r}(Y^{-1}\#_{1/\alpha}X)^{\alpha r}Y^{{1-\alpha\over2}r}\bigr)
\end{align}
for all $X,Y\in\bM_2^{++}$, noting that
$\lambda(X^{1/2}Y^{2(1-\alpha)}X^{1/2})=\lambda(Y^{1-\alpha}XY^{1-\alpha})$.

Now let $Y:=A_0$ and $X:=B_\theta$ in \eqref{F-2.10} for $x,y>0$ with $xy\ne1$, $x^{1-\alpha}y\ne1$ and
$x^{2(1-\alpha)}y\ne1$. From the proof of \cite[Theorem 4.35]{Hiai} we have
\begin{align}\label{F-3.11}
Y^{1-\alpha}XY^{1-\alpha}=\begin{bmatrix}1-\theta^2(1-y)&\theta x^{1-\alpha}(1-y)\\
\theta x^{1-\alpha}(1-y)&x^{2(1-\alpha)}y+\theta^2x^{2(1-\alpha)}(1-y)\end{bmatrix}+o(\theta^2)
\quad\mbox{as $\theta\to0$},
\end{align}
and
\begin{align}\label{F-3.12}
Y^{-1}\#_{1/\alpha}X=\begin{bmatrix}1+\theta^2 u_{11}&\theta u_{12}\\
\theta u_{12}&x^{1-\alpha\over\alpha}y^{1\over\alpha}+\theta^2u_{22}\end{bmatrix}+o(\theta^2),
\end{align}
where
\begin{align}\label{F-3.13}
\begin{cases}
u_{11}:=-{1-y\over\alpha}+{x(1-y)^2\bigl(1-\alpha-xy+\alpha x^{1\over\alpha}y^{1\over\alpha}\bigr)
\over\alpha(1-xy)^2}, \\
u_{12}:={(1-y)\bigl(1-x^{1\over\alpha}y^{1\over\alpha}\bigr)\over1-xy}.
\end{cases}
\end{align}
Hence by \cite[Example 3.2(1)]{Hiai} we compute
\[
(Y^{-1}\#_{1/\alpha}X)^{\alpha r}=\begin{bmatrix}1+\theta^2v_{11}&\theta v_{12}\\
\theta v_{12}&x^{(1-\alpha)r}y^r+\theta^2v_{22}\end{bmatrix}+o(\theta^2),
\]
where
\[
\begin{cases}
v_{11}:=\alpha ru_{11}+{\alpha r-1-\alpha rx^{1-\alpha\over\alpha}y^{1\over\alpha}
+x^{(1-\alpha)r}y^r\over\bigl(1-x^{1-\alpha\over\alpha}y^{1\over\alpha}\bigr)^2}\,u_{12}^2, \\
v_{12}:={1-x^{(1-\alpha)r}y^r\over1-x^{1-\alpha\over\alpha}y^{1\over\alpha}}\,u_{12},
\end{cases}
\]
so that
\begin{align}\label{F-3.14}
Y^{{1-\alpha\over2}r}(Y^{-1}\#_{1/\alpha}X)^{\alpha r}Y^{{1-\alpha\over2}r}
=\begin{bmatrix}1+\theta^2v_{11}&\theta x^{{1-\alpha\over2}r}v_{12}\\
\theta x^{{1-\alpha\over2}r}v_{12}&x^{2(1-\alpha)r}y^r+\theta^2v_{22}\end{bmatrix}+o(\theta^2)
\quad\mbox{as $\theta\to0$}.
\end{align}
When $x^{1-\alpha}y<1$ and $x^{2(1-\alpha)y}<1$, applying \cite[Lemma 3.6]{Hiai} to \eqref{F-3.11} and
\eqref{F-3.14} we have
\begin{align}
\lambda_1^r(Y^{1-\alpha}XY^{1-\alpha})
&=1+\theta^2r\biggl(-1+y+{x^{2(1-\alpha)}(1-y)^2\over1-x^{2(1-\alpha)}y}\biggr)+o(\theta^2), \label{F-3.15}\\
\lambda_1\bigl(Y^{{1-\alpha\over2}r}(Y^{-1}\#_{1/\alpha}X)^{\alpha r}Y^{{1-\alpha\over2}r}\bigr)
&=1+\theta^2\biggl(v_{11}+{x^{(1-\alpha)r}v_{12}^2\over1-x^{2(1-\alpha)r}y^r}\biggr)+o(\theta^2). \nonumber
\end{align}
Therefore, for any $x>0$, whenever $y>0$ is sufficiently small, it follows from \eqref{F-3.10} that
\[
r\biggl(-1+y+{x^{2(1-\alpha)}(1-y)^2\over1-x^{2(1-\alpha)}y}\biggr)
\le v_{11}+{x^{(1-\alpha)r}v_{12}^2\over1-x^{2(1-\alpha)r}y^r}.
\]
As $y\searrow0$, since
\begin{align*}
&u_{11}\to-{1\over\alpha}+{1-\alpha\over\alpha}\,x,\qquad u_{12}\to1, \\
&v_{11}\to-r+r(1-\alpha)x+\alpha r-1,\qquad v_{12}\to1,
\end{align*}
we must have
\begin{align}\label{F-3.16}
rx^{2(1-\alpha)}\le r(1-\alpha)x+\alpha r-1+x^{(1-\alpha)r},\qquad x>0.
\end{align}
Letting $x\searrow0$ gives $0\le\alpha r-1$, i.e., $1/r\le\alpha$. Hence (iii) follows.

(iii)$\implies$(i).\enspace
This part was shown in \cite[Proposition 4.33(1)]{Hiai}, while a proof is given in a different way here.
Assume that $p\le\alpha q$. By continuity we may assume that $A,B>0$. Since
$\det\widetilde SG_{\alpha,p}(A,B)=\det R_{\alpha,q}(A,B)$, in view of the anti-symmetric tensor power
technique it suffices to show that $\big\|\widetilde SG_{\alpha,p}(A,B)\big\|_\infty\le\|R_{\alpha,q}(A,B)\|_\infty$;
equivalently,
\begin{align}\label{F-3.17}
B^{\alpha q}\le A^{(\alpha-1)q}\implies\widetilde SG_{\alpha,p}(A,B)\le I.
\end{align}
Assume that $B^{\alpha q}\le A^{(\alpha-1)q}$; then we have $B^p\le A^{{\alpha-1\over\alpha}p}$ since
${p\over\alpha q}\le1$. Therefore,
\begin{align*}
A^{-p}\#_\alpha B^p\le A^{-p}\#_\alpha A^{{\alpha-1\over\alpha}p}=A^{-2(1-\alpha)p},
\end{align*}
so that we have $A^{2(1-\alpha)p}\le(A^{-p}\#_\alpha B^p)^{-1}$ and
$(A^{-p}\#_\alpha B^p)^{1/2}A^{2(1-\alpha)p}(A^{-p}\#_\alpha B^p)^{1/2}\le I$, showing \eqref{F-3.17}.

(2)\enspace
When $\alpha=1/2$, since $\lambda(\widetilde SG_{1/2,p}(A,B))=\lambda(SG_{1/2,p}(A,B))=
\lambda(R_{1/2,2p}(A,B))$, the assertion in this case follows from Theorem \ref{T-3.1}(a) if $q\le2p$ (weaker
than $q\le p$). Next assume that $\alpha<1/2$ and $q\le p$. By Theorem \ref{T-3.1}(a) again we may show
that $R_{\alpha,p}(A,B)\prec_{\log}\widetilde SG_{\alpha,p}(A,B)$, which is more explicitly written as
\begin{align}\label{F-3.18}
A^{{1-\alpha\over2}p}B^{\alpha p}A^{{1-\alpha\over2}p}
\prec_{\log}A^{{1-2\alpha\over2}p}(A^{p/2}B^pA^{p/2})^\alpha A^{{1-2\alpha\over2}p}.
\end{align}
Set $p_1:={1\over1-2\alpha}$, $q_1:={\alpha\over1-2\alpha}=\alpha p_1$, $A_1:=A^{(1-2\alpha)p}$ and
$B_1:=B^{(1-2\alpha)p}$. Since
\[
A^p=A_1^{p_1},\qquad B^p=B_1^{p_1},\qquad
A^{(1-\alpha)p}=A_1^{1-\alpha\over1-2\alpha}=A_1^{1+q_1},\qquad
B^{\alpha p}=B_1^{\alpha\over1-2\alpha}=B_1^{q_1},
\]
the log-majorization in \eqref{F-3.18} is equivalently written as
\[
A_1^{1+q_1\over2}B_1^{q_1}A_1^{1+q_1\over2}\prec_{\log}
A_1^{1/2}\bigl(A_1^{p_1/2}B_1^{p_1}A_1^{p_1/2}\bigr)^{q_1/p_1}A_1^{1/2},
\]
which is exactly the BLP log-majorization due to Bebiano, Lemos and Provid\^encia \cite[Theorem 2.1]{BLP}.

(3)\enspace
We can repeat the above proof (ii)$\implies$(iii) of (1) with inequalities reversed, so that inequality \eqref{F-3.16}
is reversed as
\[
rx^{2(1-\alpha)}\ge r(1-\alpha)x+\alpha r-1+x^{(1-\alpha)r},\qquad x>0,
\]
where $r:=q/p$. Letting $x\searrow0$ gives $0\ge\alpha r-1$, i.e., $1/r\ge\alpha$. Also, looking at the order
as $x\to\infty$ we have $2(1-\alpha)\ge1$ and $2(1-\alpha)\ge(1-\alpha)r$, so that $\alpha\le1/2$ and
$1/r\ge1/2$. Hence the result follows.
\end{proof}

The next theorem is concerned with the log-majorizations between $\widetilde SG_{\alpha,p}$ and
$R_{\alpha,q}$ when $\alpha>1$.

\begin{theorem}\label{T-3.7}
Let $\alpha>1$ and $p,q>0$.
\begin{itemize}
\item[(1)] For any $p,q>0$ there exist $A,B\in\bM_2^{++}$ such that
$\widetilde SG_{\alpha,p}(A,B)\not\prec_{\log}R_{\alpha,q}(A,B)$.
\item[(2)] If $p/q\ge\alpha$, then $R_{\alpha,q}(A,B)\prec_{\log}\widetilde SG_{\alpha,p}(A,B)$ holds for all
$A,B\ge0$ with $s(A)\ge s(B)$.
\item[(3)] Assume that $R_{\alpha,q}(A,B)\prec_{\log}\widetilde SG_{\alpha,p}(A,B)$ holds for all
$A,B\in\bM_2^{++}$. Then we have $p/q\ge1/2$.
\end{itemize}
\end{theorem}

\begin{proof}
(1)\enspace
Assume by contradiction that $\widetilde SG_{\alpha,p}(A,B)\prec_{\log}R_{\alpha,q}(A,B)$ holds for all
$A,B\in\bM_2^{++}$. Then similarly to the first paragraph of the proof (ii)$\implies$(iii) of
Theorem \ref{T-3.6}(1), we have inequality \eqref{F-3.10} for all $X,Y\in\bM_2^{++}$. Hence the computations
in the proof of (ii)$\implies$(iii) of Theorem \ref{T-3.6}(1) show inequality \eqref{F-3.16} for all $x>0$. Since
$\alpha>1$ in the present case, letting $y\to\infty$ gives $0\le-\infty$, a contradiction.

(2)\enspace
Assume that $p\ge\alpha q$. By continuity we may assume that $A,B>0$. As in the proof (iii)$\implies$(i) of
Theorem \ref{T-3.6}(1), it suffices to show that
\[
\widetilde SG_{\alpha,p}(A,B)\le I\implies B^{\alpha q}\le A^{(\alpha-1)q}.
\]
If $\widetilde SG_{\alpha,p}(A,B)\le I$, then one has $A^{-p}\#_\alpha B^p\le A^{2(\alpha-1)p}$, so that
$(A^{p/2}B^pA^{p/2})^\alpha\le A^{(2\alpha-1)p}$. Since $\alpha>1$, this gives
$A^{p/2}B^pA^{p/2}\le A^{{2\alpha-1\over\alpha}p}$ and hence $B^p\le A^{{\alpha-1\over\alpha}p}$. Since
${\alpha q\over p}\le1$, we have $B^{\alpha q}\le A^{(\alpha-1)q}$.

(3)\enspace
Repeating the proof (ii)$\implies$(iii) of Theorem \ref{T-3.6}(1) with inequalities reversed, we have
\[
rx^{2(1-\alpha)}\ge r(1-\alpha)x+\alpha r-1+x^{(1-\alpha)r},\qquad x>0,
\]
where $r:=q/p$. Now assume that $p/q<1/2$. Then, since $0>2(1-\alpha)>(1-\alpha)r$, the above inequality
is impossible in the limit $x\searrow0$. Hence $p/q\ge1/2$ must hold.
\end{proof}

\begin{problem}\label{Q-3.8}\rm
When $0<\alpha<1$, there is a gap between the sufficient condition in Theorem \ref{T-3.6}(2) and the
necessary condition in Theorem \ref{T-3.6}(3) for $R_{\alpha,q}\prec_{\log}\widetilde SG_{\alpha,p}$. When
$\alpha>1$, there is also a big gap between the sufficient condition in Theorem \ref{T-3.7}(2) and the
necessary condition in Theorem \ref{T-3.7}(3) for $R_{\alpha,q}\prec_{\log}\widetilde SG_{\alpha,p}$. Hence
the problem of characterizing $R_{\alpha,q}\prec_{\log}\widetilde SG_{\alpha,p}$ is far from completed.
When $\alpha=1/2$, note that the condition in Theorem \ref{T-3.6}(2) is not sharp as mentioned in its proof,
while the sufficient condition in Theorem \ref{T-3.7}(2) is sharp.
\end{problem}

\begin{remark}\label{R-3.9}\rm
Let $\alpha\in(0,\infty)\setminus\{1\}$ and $p,q>0$. For any pair $(\cM_{\alpha,p},\cN_{\alpha,q})$
such as in Remark \ref{R-2.4}, one can easily find a necessary condition for
$\cM_{\alpha,p}\prec_{\log}\cN_{\alpha,q}$ by use of expressions given in Lemma \ref{L-2.5}. For instance,
if $SG_{\alpha,p}(A,B)\prec_{\log}R_{\alpha,q}(A,B)$ for all $A,B\in\bM_2^{++}$, then from \eqref{F-2.11}
and \eqref{F-2.13} one must have
\[
-{2\alpha(1-\alpha)\over p}\cdot{1-x^p\over1+x^p}
\le{-1-x^q+x^{\alpha q}+x^{(1-\alpha)q}\over q(1-x^q)},\qquad x\in(0,1).
\]
Letting $x\searrow0$ gives $p/q\le2\alpha(1-\alpha)$ if $0<\alpha<1$. Similarly, if
$R_{\alpha,q}(A,B)\prec_{\log}SG_{\alpha,p}(A,B)$ for all $A,B\in\bM_2^{++}$, then one has
$p/q\ge2\alpha(1-\alpha)$ if $0<\alpha<1$, and a contradiction if $\alpha>1$. Hence Theorem \ref{T-3.4}(2)
follows again, while we have necessary conditions weaker than Theorem \ref{T-3.3} since
$\min\{\alpha,1-\alpha\}\le2\alpha(1-\alpha)\le\max\{\alpha,1-\alpha\}$ for $0<\alpha<1$.

Also, if $\widetilde SG_{\alpha,p}(A,B)\prec_{\log}R_{\alpha,q}(A,B)$ for all $A,B\in\bM_2^{++}$, then from
\eqref{F-2.11} and \eqref{F-2.14} one has
\[
-{1\over p}\biggl\{\alpha\,{1-x^p\over1+x^p}
+{x^p+x^{2p}-x^{(2\alpha+1)p}-x^{2(1-\alpha)p}\over(1-x^p)(1+x^p)^2}\biggr\}
\le{-1-x^q+x^{\alpha q}+x^{(1-\alpha)q}\over q(1-x^q)},\qquad x\in(0,1).
\]
Letting $x\searrow0$ gives $p/q\le\alpha$ if $0<\alpha<1$, and $p/q\ge1/2$ if $\alpha>1$. Similarly, if
$R_{\alpha,q}(A,B)\prec_{\log}\widetilde SG_{\alpha,p}(A,B)$ for all $A,B\in\bM_2^{++}$, then one has
$p/q\ge\alpha$ if $0<\alpha<1$, and $p/q\ge1/2$ if $\alpha\ge1/2$. These show (ii)$\implies$(iii) of
Theorem \ref{T-3.6}(1) and Theorem \ref{T-3.7}(3) again and a partial necessary condition in
Theorem \ref{T-3.6}(3). In this way, an easy way of using Lemma \ref{L-2.5} is sometimes (though not always)
enough to obtain the necessary conditions shown in the theorems of Sections \ref{Sec-3.1} and \ref{Sec-3.2}.
\end{remark}

\subsection{Regarding $G_{\alpha,p}\prec_{\log}G_{\alpha,q}$, $SG_{\alpha,p}\prec_{\log}SG_{\alpha,q}$
and $\widetilde SG_{\alpha,p}\prec_{\log}\widetilde SG_{\alpha,q}$}\label{Sec-3.3}

While the log-majorizations $G_{\alpha,p}\prec_{\log}G_{\alpha,q}$ for $\alpha\in(0,2]\setminus\{1\}$ were
completely characterized as stated in Theorem \ref{T-3.1}(b) and (c), the characterizations of
$SG_{\alpha,p}\prec_{\log}SG_{\alpha,q}$ and $\widetilde SG_{\alpha,p}\prec_{\log}\widetilde SG_{\alpha,q}$,
as well as $G_{\alpha,p}\prec_{\log}G_{\alpha,q}$ for $\alpha>2$, have not been obtained yet.\footnote{
It was claimed in \cite[Theorem 3.3]{GT} that $SG_{\alpha,p}(A,B)\prec_{\log}SG_{\alpha,q}(A,B)$ for all
$A,B>0$ if $0<\alpha<1$ and $0<p\le q$; however the proof contains a flaw.}
In this subsection we discuss the problem though without complete success.

The proof of Theorem \ref{T-3.1}(d) for $\alpha\ge2$ was given in \cite{Hi2} in an indirect way combining
(e) and (f), so the sufficient condition in (d) is probably not best possible. In a similar way, a certain sufficient
condition for $SG_{\alpha,p}\prec_{\log}SG_{\alpha,q}$ (resp.,
$\widetilde SG_{\alpha,p}\prec_{\log}\widetilde SG_{\alpha,q}$) is given as follows by combining the
log-majorizations in Section \ref{Sec-3.1} (resp., Section \ref{Sec-3.2}).

\begin{proposition}\label{P-3.10}
Let $0<\alpha<1$ and $p,q>0$.
\begin{itemize}
\item[(1)] If $p/q\le\min\{{\alpha\over1-\alpha},{1-\alpha\over\alpha}\}$, then we have
$SG_{\alpha,p}(A,B)\prec_{\log}SG_{\alpha,q}(A,B)$ for all $A,B\ge0$ with $s(A)\ge s(B)$.
\item[(2)] If $\alpha\le1/2$ and $p/q\le\alpha$, then we have
$\widetilde SG_{\alpha,p}(A,B)\prec_{\log}\widetilde SG_{\alpha,q}(A,B)$ for all $A,B\ge0$ with
$s(A)\ge s(B)$.
\end{itemize}
\end{proposition}

\begin{proof}
(1)\enspace
Assume that $0<\alpha\le1/2$ and $p/q\le{\alpha\over1-\alpha}$. Since ${q\over p/\alpha}\ge1-\alpha$,
by Theorem \ref{T-3.3} we have
\[
SG_{\alpha,p}(A,B)\prec_{\log}R_{\alpha,{p/\alpha}}(A,B)\prec_{\log}SG_{\alpha,q}(A,B)
\]
Assume that $1/2\le\alpha<1$ and $p/q\le{1-\alpha\over\alpha}$. Since ${q\over p/(1-\alpha)}\ge\alpha$,
we have similarly
\[
SG_{\alpha,p}(A,B)\prec_{\log}R_{\alpha,{p/(1-\alpha)}}(A,B)\prec_{\log}SG_{\alpha,q}(A,B).
\]

(2)\enspace
Assume that $\alpha\le1/2$ and $p/\alpha\le q$. Then by Theorem \ref{T-3.6} we have
\[
\widetilde SG_{\alpha,p}(A,B)\prec_{\log}R_{\alpha,p/\alpha}(A,B)
\prec_{\log}\widetilde SG_{\alpha,q}(A,B).
\]
\end{proof}

The next proposition shows that $p\le q$ or $p\ge q$ is anyhow necessary for the log-majorizations of our
concern to hold.

\begin{proposition}\label{P-3.11}
Let $\alpha\in(0,\infty)\setminus\{1\}$ and $p,q>0$.
\begin{itemize}
\item[(1)] Assume that $G_{\alpha,p}(A,B)\prec_{\log}G_{\alpha,q}(A,B)$ holds for all $A,B\in\bM_2^{++}$.
Then we have $p\ge q$ if $0<\alpha<1$, and $p\le q$ if $\alpha>1$.
\item[(2)] Assume that $SG_{\alpha,p}(A,B)\prec_{\log}SG_{\alpha,q}(A,B)$ holds for all $A,B\in\bM_2^{++}$.
Then we have $p\le q$ if $0<\alpha<1$, and $p\ge q$ if $\alpha>1$.
\item[(3)] Assume that $\widetilde SG_{\alpha,p}(A,B)\prec_{\log}\widetilde SG_{\alpha,q}(A,B)$ holds for all
$A,B\in\bM_2^{++}$. Then we have $p\le q$ whichever $0<\alpha<1$ or $\alpha>1$.
\end{itemize}
\end{proposition}

\begin{proof}
These are easily checked by use of expressions \eqref{F-2.12}--\eqref{F-2.14} in Lemma \ref{L-2.5} similarly
to the discussions in Remark \ref{R-3.9}. The details are left to the reader.
\end{proof}

%
%

\begin{problem}\label{Q-3.12}\rm
What we are most interested in is whether a necessary condition $p\le q$ or $p\ge q$ confirmed in
Proposition \ref{P-3.11} is indeed sufficient for $SG_{\alpha,p}\prec_{\log}SG_{\alpha,q}$ and
$\widetilde SG_{\alpha,p}\prec_{\log}\widetilde SG_{\alpha,q}$ for relevant $\alpha$ to hold or not. The
problem is to show the variants of Ando--Hiai's inequality \cite{AH} for $SG_{\alpha,p}$ and
$\widetilde SG_{\alpha,q}$. Regarding this, see remark (3) of Section \ref{Sec-6} as well.
\end{problem}

Note that for $0<\alpha<1$ the condition on $p,q$ in Proposition \ref{P-3.11}(1) and that in
Proposition \ref{P-3.11}(2) and (3) are opposite. This suggests us that two types of weighted spectral
geometric means are quite different from Kubo--Ando's weighted geometric mean $\#_\alpha$.

\subsection{More log-majorizations}\label{Sec-3.4}

In this subsection we will examine log-majorizations between $G_\alpha$, $SG_\alpha$ and
$\widetilde SG_\alpha$ in an easy way of putting $R_\alpha$ in the middle of two of them for sufficient
conditions, or the outside of the two for necessary conditions, otherwise by use of Lemma \ref{L-2.5} as in
Remark \ref{R-3.9}.

\begin{proposition}\label{P-3.13}
\
\begin{itemize}
\item[(1)] Let $0<\alpha<1$. For any $p,q>0$ we have $G_{\alpha,q}(A,B)\prec_{\log}SG_{\alpha,p}(A,B)$
for all $A,B\ge0$ with $s(A)\ge s(B)$.
\item[(2)] Let $\alpha>1$. For any $p,q>0$ there exist $A,B\in\bM_2^{++}$ such that
$G_{\alpha,q}(A,B)\not\prec_{\log}SG_{\alpha,p}(A,B)$.
\item[(3)] Let $\alpha>1$. If $p/q\le\max\bigl\{2,{\alpha\over\alpha-1}\bigr\}$, then we have
$SG_{\alpha,p}(A,B)\prec_{\log}G_{\alpha,q}(A,B)$ for all $A,B\ge0$ with $s(A)\ge s(B)$.
\end{itemize}
\end{proposition}

\begin{proof}
(1)\enspace
Thanks to Theorem \ref{T-3.2}(1) and \ref{T-3.3}(2), taking an $r>0$ with $p/r\ge\max\{\alpha,1-\alpha\}$
we have
\[
G_{\alpha,q}(A,B)\prec_{\log}R_{\alpha,r}(A,B)\prec_{\log}SG_{\alpha,p}(A,B).
\]

(2) is easily verified by using expressions \eqref{F-2.12} and \eqref{F-2.13} in Lemma \ref{L-2.5}
similarly to the discussions in Remark \ref{R-3.9}.

%

(3)\enspace
Assume that $\alpha>1$ and $p/q\le\min\bigl\{2,{\alpha\over\alpha-1}\bigr\}$. For every $A,B\ge0$ with
$s(A)\ge s(B)$, letting $r:=p/\alpha$, we have $SG_{\alpha,p}(A,B)\prec_{\log}R_{\alpha,r}(A,B)$ by
Theorem \ref{T-3.4}(1). Moreover, since $q/r=q\alpha/p\ge\max\{\alpha/2,\alpha-1\}$, we have
$R_{\alpha,r}(A,B)\prec_{\log}G_{\alpha,q}(A,B)$ by Theorem \ref{T-3.2}(2). Therefore,
$SG_{\alpha,p}(A,B)\prec_{\log}G_{\alpha,q}(A,B)$.
\end{proof}

\begin{proposition}\label{P-3.14}
\
\begin{itemize}
\item[(1)] For any $\alpha\in(0,\infty)\setminus\{1\}$ and any $p,q>0$ there exist $A,B\in\bM_2^{++}$ such that
$\widetilde SG_{\alpha,p}(A,B)\not\prec_{\log}G_{\alpha,q}(A,B)$.
\item[(2)] Let $0<\alpha\le1/2$. For any $p,q$ we have
$G_{\alpha,q}(A,B)\prec_{\log}\widetilde SG_{\alpha,p}(A,B)$ for all $A,B\ge0$ with $s(A)\ge s(B)$.
\item[(3)] Let $\alpha>1$. If $q/p\le\min\bigl\{1/2,{\alpha-1\over\alpha}\bigr\}$, then we have
$G_{\alpha,q}(A,B)\prec_{\log}\widetilde SG_{\alpha,p}(A,B)$ for all $A,B\ge0$ with $s(A)\ge s(B)$.
\end{itemize}
\end{proposition}

\begin{proof}
(1)\enspace
Assume that $\widetilde SG_{\alpha,p}(A,B)\prec_{\log}G_{\alpha,q}(A,B)$ for all $A,B\in\bM_2^{++}$. When
$0<\alpha<1$, by Theorem \ref{T-3.2}(1) we have
\[
\widetilde SG_{\alpha,p}(A,B)\prec_{\log}G_{\alpha,q}(A,B)\prec_{\log}R_{\alpha,r}(A,B)
\]
for any $r>0$ and all $A,B\in\bM_2^{++}$, contradicting Theorem \ref{T-3.6}(1). (The $0<\alpha<1$ case is
seen also by use of expressions \eqref{F-2.12} and \eqref{F-2.14}.) When $\alpha>1$, by Theorem \ref{T-3.2}(1)
again the above holds for any sufficiently large $r>0$ and all $A,B\in\bM_2^{++}$, contradicting
Theorem \ref{T-3.7}(1).
 
(2)\enspace
Let $0<\alpha\le1/2$. For any $p,q>0$, by Theorems \ref{T-3.2}(1) and \ref{T-3.6}(2) we have
\[
G_{\alpha,q}(A,B)\prec_{\log}R_{\alpha,p}(A,B)\prec_{\log}\widetilde SG_{\alpha,p}(A,B)
\]
for all $A,B\ge0$ with $s(A)\ge s(B)$.

(3)\enspace
Assume that $\alpha>1$ and $q/p\le\min\bigl\{1/2,{\alpha-1\over\alpha}\bigr\}$. Since
$q/(p/\alpha)\le\min\{\alpha/2,\alpha-1\}$, by Theorems \ref{T-3.2}(1) and  \ref{T-3.7}(2) we have
\[
G_{\alpha,q}(A,B)\prec_{\log}R_{\alpha,p/\alpha}(A,B)\prec_{\log}\widetilde SG_{\alpha,p}(A,B)
\]
for all $A,B\ge0$ with $s(A)\ge s(B)$.
\end{proof}

\begin{proposition}\label{P-3.15}
\
\begin{itemize}
\item[(1)] Let $0<\alpha<1$ and $p,q>0$. If $SG_{\alpha,p}(A,B)\prec_{\log}\widetilde SG_{\alpha,q}(A,B)$
for all $A,B\in\bM_2^{++}$, then we have $\alpha\le1/2$ and $p/q\le2(1-\alpha)$.
\item[(2)] Let $0<\alpha<1$. If $p/q\ge\max\bigl\{1,{1-\alpha\over\alpha}\bigr\}$, then
$\widetilde SG_{\alpha,q}(A,B)\prec_{\log}SG_{\alpha,p}(A,B)$ for all $A,B\ge0$ with $s(A)\ge s(B)$.
Morevoer, if $\widetilde SG_{\alpha,q}(A,B)\prec_{\log}SG_{\alpha,p}(A,B)$ for all $A,B\in\bM_2^{++}$, then
we have $p/q\ge2(1-\alpha)$. (Note that $2(1-\alpha)\le\max\bigl\{1,{1-\alpha\over\alpha}\bigr\}$ for
$0<\alpha<1$.)
\item[(3)] Let $\alpha>1$. If $p\le q$, then $SG_{\alpha,p}(A,B)\prec_{\log}\widetilde SG_{\alpha,q}(A,B)$
for all $A,B\ge0$ with $s(A)\ge s(B)$.
\item[(4)] Let $\alpha>1$. For any $p,q>0$ there exist $A,B\in\bM_2^{++}$ such that
$\widetilde SG_{\alpha,q}(A,B)\not\prec_{\log}SG_{\alpha,p}(A,B)$.
\end{itemize}
\end{proposition}

\begin{proof}
(1)\enspace
Assume that $SG_{\alpha,p}(A,B)\prec_{\log}\widetilde SG_{\alpha,q}(A,B)$ for all $A,B\in\bM_2^{++}$.
For $r:=p/\max\{\alpha,1-\alpha\}$, by Theorem \ref{T-3.3}(2) we have
\[
R_{\alpha,r}(A,B)\prec_{\log}SG_{\alpha,p}(A,B)\prec_{\log}\widetilde SG_{\alpha,q}(A,B)
\]
for all $A,B\in\bM_2^{++}$. Hence by Theorem \ref{T-3.6}(3) we must have $\alpha\le1/2$ and $q/r\ge1/2$
so that $p/q\le2\max\{\alpha,1-\alpha\}=2(1-\alpha)$.

(2)\enspace
Assume that $0<\alpha<1$ and $p/q\ge\max\bigl\{1,{1-\alpha\over\alpha}\bigr\}$. Let $r:=q/\alpha$ and so
$p/r\ge\max\{\alpha,1-\alpha\}$. Thanks to Theorem \ref{T-3.6}(1) and \ref{T-3.3}(2) we have
\[
\widetilde SG_{\alpha,q}(A,B)\prec_{\log}R_{\alpha,r}(A,B)\prec_{\log}SG_{\alpha,p}(A,B)
\]
for all $A,B\ge0$ with $s(A)\ge s(B)$. Next, assume that
$\widetilde SG_{\alpha,q}(A,B)\prec_{\log}SG_{\alpha,p}(A,B)$ for all $A,B\in\bM_2^{++}$. Then
$p/q\ge2(1-\alpha)$ is easily verified from expressions \eqref{F-2.13} and \eqref{F-2.14}.

(3)\enspace
Letting $r:=p/\alpha$ and so $q/r\ge\alpha$, by Theorems \ref{T-3.4}(1) and \ref{T-3.7}(2) we have
\[
SG_{\alpha,p}(A,B)\prec_{\log}R_{\alpha,r}(A,B)\prec_{\log}\widetilde SG_{\alpha,q}(A,B)
\]
for all $A,B\ge0$ with $s(A)\ge s(B)$.

(4) is easily verified from expressions \eqref{F-2.13} and \eqref{F-2.14} again.
\end{proof}

The section ends with log-majorizations of $LE_\alpha$ with other quasi-geometric type means. The next
two remarks are easy but convenient.

\begin{remark}\label{R-3.16}\rm
Let $\cM\in\{R,G,SG,\widetilde SG\}$, $\alpha\in(0,\infty)\setminus\{1\}$ and $p>0$.

(1)\enspace
The log-majorization $\cM_{\alpha,p}(A,B)\prec_{\log}LE_\alpha(A,B)$ (resp.,
$LE_\alpha(A,B)\prec_{\log}\cM_{\alpha,p}(A,B)$) holds for all $A,B>0$ if and only if
$\cM_{\alpha,1}(A,B)\prec_{\log}LE_\alpha(A,B)$ (resp., $LE_\alpha(A,B)\prec_{\log}\cM_{\alpha,1}(A,B)$) for
all $A,B>0$. Indeed, if $\cM_{\alpha,1}(A,B)\prec_{\log}LE_\alpha(A,B)$ for all $A,B>0$, then
\[
\cM_{\alpha,1}(A^p,B^p)\prec_{\log}LE_\alpha(A^p,B^p)=LE_\alpha(A,B)^p
\]
so that we have $\cM_{\alpha,p}(A,B)\prec_{\log}LE_\alpha(A,B)$ for all $A,B>0$.

(2)\enspace
If $\cM_{\alpha,p}(A,B)\prec_{\log}LE_\alpha(A,B)$ (resp., $LE_\alpha(A,B)\prec_{\log}\cM_{\alpha,p}(A,B)$)
holds for all $A,B\in\bM_2^{++}$, then the reverse log-majorization
$LE_\alpha(A,B)\prec_{\log}\cM_{\alpha,p}(A,B)$ (resp., $\cM_{\alpha,p}(A,B)\prec_{\log}LE_\alpha(A,B)$)
does not hold for some $A,B\in\bM_2$. Otherwise, we must have
$\lambda(\cM_{\alpha,p}(A,B))=\lambda(LE_\alpha(A,B))$ for all $A,B\in\bM_2^{++}$, which contradicts the
assertion of Remark \ref{R-2.4}(5).
\end{remark}

The known cases are summarized in the following proposition.

\begin{proposition}\label{P-3.17}
\
\begin{itemize}
\item[(a)] $LE_\alpha(A,B)\prec_{\log}R_{\alpha,p}(A,B)$ for all $\alpha\in(0,\infty)\setminus\{1\}$, $p>0$ and
$A,B\ge0$ (with $s(A)\ge s(B)$ for $\alpha>1$).
\item[(b)] $G_{\alpha,p}(A,B)\prec_{\log}LE_\alpha(A,B)$ for all $\alpha\in(0,1)$, $p>0$ and $A,B\ge0$.
\item[(c)] $LE_\alpha(A,B)\prec_{\log}G_{\alpha,p}(A,B)$ for all $\alpha\in(1,2]$, $p>0$ and $A,B\ge0$ with
$s(A)\ge s(B)$.
\item[(d)] $LE_\alpha(A,B)\prec_{\log}SG_{\alpha,p}(A,B)$ for all $\alpha\in(0,1)$, $p>0$ and $A,B\ge0$ with
$s(A)\ge s(B)$.
\end{itemize}
\end{proposition}

See \cite{AH,BG} for (a) and (b), \cite{KS} for (c), and \cite{GLT,GK} for (d). When $0<\alpha<1$, it is worth
noting that for every (a) and (b) above are supplemented with the quasi-arithmetic and the quasi-harmonic
means as follows: if $r\ge p/2$ and $s\ge q$ then
\[
\cH_{\alpha,s}(A,B)\le_\lambda
G_{\alpha,q}(A,B)\prec_{\log}LE_\alpha(A,B)\prec_{\log}R_{\alpha,p}(A,B)
\le_\lambda\cA_{\alpha,r}(A,B),\qquad A,B\ge0,
\]
where the additional inequalities in both sides are seen from
\cite[Propositions 4.9, 4.19 and Remark 2.7(3)]{Hiai}.

From log-majorizations given in Sections \ref{Sec-3.1} and \ref{Sec-3.2} we have more results as follows:

\begin{corollary}\label{C-3.18}
\
\begin{itemize}
\item[(1)] $LE_\alpha(A,B)\prec_{\log}G_{\alpha,p}(A,B)$ for all $\alpha>1$, $p>0$ and $A,B\ge0$ with
$s(A)\ge s(B)$.
\item[(2)] For any $\alpha\in(0,1)$ and $p>0$ there exist $A,B\in\bM_2^{++}$ such that
$SG_{\alpha,p}(A,B)\not\prec_{\log}LE_\alpha(A,B)$.
\item[(3)] For any $\alpha>1$ and $p>0$ there exist $A,B\in\bM_2^{++}$ such that
$LE_\alpha(A,B)\not\prec_{\log}SG_{\alpha,p}(A,B)$.
\item[(4)] For any $\alpha\in(0,\infty)\setminus\{1\}$ and $p>0$ there exist
$A,B\in\bM_2^{++}$ such that $\widetilde SG_{\alpha,p}(A,B)\not\prec_{\log}LE_\alpha(A,B)$.
\item[(5)] $LE_\alpha(A,B)\prec_{\log}\widetilde SG_{\alpha,p}(A,B)$ for all $\alpha\in(0,1/2]$, $p>0$
and $A,B\ge0$ with $s(A)\ge s(B)$.
\item[(6)] $LE_\alpha(A,B)\prec_{\log}\widetilde SG_{\alpha,p}(A,B)$ for all $\alpha>1$, $p>0$ and
$A,B\ge0$ with $s(A)\ge s(B)$.
\end{itemize}
\end{corollary}

\begin{proof}
(1) follows by letting $p\searrow0$ in (c) and (d) of Theorem \ref{T-3.1}. (2) and (3) are immediately seen by
comparing \eqref{F-2.15} with \eqref{F-2.13}, and (4) is by comparing \eqref{F-2.15} with
\eqref{F-2.14}. (5) and (6) follow as the $q\searrow0$ limits (due to Theorem \ref{T-2.2}) of the
log-majorizations in Theorems \ref{T-3.6}(2) and \ref{T-3.7}(2), respectively.
\end{proof}

The next theorem says that $LE_\alpha\prec_{\log}\widetilde SG_{\alpha,p}$ for $1/2<\alpha<1$ fails to hold.

\begin{theorem}\label{T-3.19}
For any $\alpha\in(1/2,1)$ and $p>0$ there exist $A,B\in\bM_2^{++}$ such that
$LE_\alpha(A,B)\not\prec_{\log}\widetilde SG_{\alpha,p}(A,B)$.
\end{theorem}

\begin{proof}
By Remark \ref{R-3.16}(1) we may assume that $p=1$. So we may show that if $0<\alpha<1$ and
$LE_\alpha(A,B)\prec_{\log}\widetilde SG_{\alpha,1}(A,B)$ for all $A,B\in\bM_2^{++}$, then $\alpha\le1/2$.
Similarly to the first paragraph of the proof (ii)$\implies$(iii) of Theorem \ref{T-3.6}(1) the above
log-majorization for all $A,B\in\bM_2^{++}$ is equivalent to
\[
LE_\alpha(Y,Y^{-1}\#_{1/\alpha}X)\prec_{\log}X^{1/2}Y^{2(1-\alpha)}X^{1/2},\qquad X,Y\in\bM_2^{++},
\]
equivalently,
\begin{align}\label{F-3.19}
\exp\lambda_1((1-\alpha)\log Y+\alpha\log(Y^{-1}\#_{1/\alpha}X))
\le\lambda_1(Y^{1-\alpha}XY^{1-\alpha}),\qquad X,Y\in\bM_2^{++}.
\end{align}
Now let $Y:=A_0$ and $X:=B_\theta$ in \eqref{F-2.10} for $x,y>0$ with $xy<1$ and $x^{2(1-\alpha)}y<1$.
By \eqref{F-3.15} (with $r=1$) we have
\begin{align}\label{F-3.20}
\lambda_1(Y^{1-\alpha}XY^{1-\alpha})
=1+\theta^2\biggl(-1+y+{x^{2(1-\alpha)}(1-y)^2\over1-x^{2(1-\alpha)}y}\biggr)+o(\theta^2)
\quad\mbox{as $\theta\to0$}.
\end{align}
On the other hand, apply \cite[Example 3.2(3)]{Hiai} to \eqref{F-3.12} and \eqref{F-3.13} to compute
\[
\log(Y^{-1}\#_{1/\alpha}X)=\begin{bmatrix}\theta^2w_{11}&\theta w_{12}\\
\theta w_{12}&\log x^{1-\alpha\over\alpha}y^{1\over\alpha}+\theta^2w_{22}\end{bmatrix}+o(\theta^2),
\]
where
\[
\begin{cases}
w_{11}:=u_{11}+{1-x^{1-\alpha\over\alpha}y^{1\over\alpha}+\log x^{1-\alpha\over\alpha}y^{1\over\alpha}
\over\bigl(1-x^{1-\alpha\over\alpha}y^{1\over\alpha}\bigr)^2}\,u_{12}^2, \\
w_{12}:=-{\log x^{1-\alpha\over\alpha}y^{1\over\alpha}\over1-x^{1-\alpha\over\alpha}y^{1\over\alpha}}\,u_{12}.
\end{cases}
\]
Hence it follows that
\[
(1-\alpha)\log Y+\alpha\log(Y^{-1}\#_{1/\alpha}X)
=\begin{bmatrix}\theta^2\alpha w_{11}&\theta\alpha w_{12}\\
\theta\alpha w_{12}&\log x^{2(1-\alpha)}y+\theta^2\alpha w_{22}\end{bmatrix}+o(\theta^2),
\]
so that by \cite[Lemma 3.6]{Hiai},
\[
\lambda_1((1-\alpha)\log Y+\alpha\log(Y^{-1}\#_{1/\alpha}X))
=\theta^2\biggl(\alpha w_{11}-{\alpha^2w_{12}^2\over\log x^{2(1-\alpha)}y}\biggr)+o(\theta^2),
\]
which implies that
\begin{align}\label{F-3.21}
\exp\lambda_1((1-\alpha)\log Y+\alpha\log(Y^{-1}\#_{1/\alpha}X))
=1+\theta^2\biggl(\alpha w_{11}-{\alpha^2w_{12}^2\over\log x^{2(1-\alpha)}y}\biggr)+o(\theta^2).
\end{align}
Therefore, by \eqref{F-3.19}--\eqref{F-3.21} we must have
\begin{align}\label{F-3.22}
\alpha w_{11}-{\alpha^2w_{12}^2\over\log x^{2(1-\alpha)}y}
\le-1+y+{x^{2(1-\alpha)}(1-y)^2\over1-x^{2(1-\alpha)}y}
\end{align}
for any $x>0$, whenever $y>0$ is sufficiently small. A direct computation yields that the LHS of
\eqref{F-3.22} is arranged as
\[
\alpha u_{11}+{2\alpha\big(1-x^{1-\alpha\over\alpha}y^{1\over\alpha}\bigr)\log x^{1-\alpha}
+\log^2x^{1-\alpha}-\alpha x^{1-\alpha\over\alpha}y^{1\over\alpha}\log y+(\alpha+\log x^{1-\alpha})\log y
\over\bigl(1-x^{1-\alpha\over\alpha}y^{1\over\alpha}\bigr)^2(2\log x^{1-\alpha}+\log y)}\,u_{12}^2,
\]
whose limit as $y\searrow0$ is equal to $-1+(1-\alpha)x+\alpha+\log x^{1-\alpha}$, since
$u_{11}\to-{1\over\alpha}+{1-\alpha\over\alpha}\,x$ and $u_{12}\to1$ as $y\searrow0$ thanks to
\eqref{F-3.13}. Hence, letting $y\searrow0$ in \eqref{F-3.22} gives
\[
(1-\alpha)x+\alpha+\log x^{1-\alpha}\le x^{2(1-\alpha)},\qquad x>0,
\]
which implies that $1\le2(1-\alpha)$, i.e., $\alpha\le1/2$, as desired.
\end{proof}

\begin{problem}\label{Q-3.20}\rm
The case remaining open is whether $SG_{\alpha,p}\prec_{\log}LE_\alpha$ for $\alpha>1$. Note that this is
affirmative if $SG_{\alpha,p}\prec_{\log}SG_{\alpha,q}$ if $\alpha>1$ and $p\ge q$ (see Problem \ref{Q-3.12}).
\end{problem}

For the convenience of the reader, Proposition \ref{P-3.17}, Corollary \ref{C-3.18} (also Remark \ref{R-3.16})
and Theorem \ref{T-3.19} together are summarized as follows. Here, ``all $p$'' says that the designated
log-majorization holds for all $A,B>0$ and all $p>0$, and ``none'' says that for any $p>0$ it fails for some
$A,B\in\bM_2^{++}$. An asymmetric behavior of $\widetilde SG_{\alpha,p}$ for $0<\alpha<1$ around
$\alpha=1/2$ is worth noting, that is a reflection of asymmetry of $\widetilde SG_{\alpha,p}$ under
interchanging $\alpha$ and $1-\alpha$.

\medskip
\begin{table}[htb]
\centering
\begin{tabular}{|c|l|l|} \hline
& $0<\alpha<1$ & $\alpha>1$\phantom{///} \\ \hline

$R_{\alpha,p}\prec_{\log}LE_\alpha$ & none & none \\ \hline

$LE_\alpha\prec_{\log}R_{\alpha,p}$ & all $p$ & all $p$ \\ \hline

$G_{\alpha,p}\prec_{\log}LE_\alpha$ & all $p$ &none \\ \hline

$LE_\alpha\prec_{\log}G_{\alpha,p}$ & none & all $p$ \\ \hline

$SG_{\alpha, p}\prec_{\log}LE_\alpha$ & none & \ \ ? \\ \hline

$LE_\alpha\prec_{\log}SG_{\alpha,p}$ & all $p$ & none \\ \hline

$\widetilde SG_{\alpha,p}\prec_{\log}LE_\alpha$ & none & none \\ \hline

$LE_\alpha\prec_{\log}\widetilde SG_{\alpha,p}$ &
$\hskip-1.8mm\begin{array}{ll}\mbox{all $p$\, for $0<\alpha\le1/2$} \\
\mbox{none for $1/2<\alpha<1$}\end{array}$ & all $p$ \\ \hline
\end{tabular}
\end{table}

\section{Equality cases in norm inequalities derived from log-majorizations}\label{Sec-4}

Recall that if $X,Y\in\bM_n^+$ and $X\prec_{\log}Y$, then $\|X\|\le\|Y\|$ holds for all unitarily invariant
norms $\|\cdot\|$ on $\bM_n$; see, e.g., \cite[Proposition 4.4.13]{Hi}. Therefore, if a pair
$(\cM_{\alpha,p},\cN_{\alpha,p})$ of quasi-geometric type means satisfies
$\cM_{\alpha,p}\prec_{\log}\cN_{\alpha,q}$, then we have $\|\cM_{\alpha,p}(A,B)\|\le\|\cN_{\alpha,q}(A,B)\|$
for all $A,B\ge0$ (with $s(A)\ge s(B)$) and for any unitarily invariant norm $\|\cdot\|$. In this section we aim
at characterizing the equality case in the above norm inequality in terms of commutativity $AB=BA$ in certain
cases of the pair $(\cM_{\alpha,p},\cN_{\alpha,p})$ and the norm $\|\cdot\|$. This is reasonable since
$AB=BA$ implies that $\cM_{\alpha,p}(A,B)=\cN_{\alpha,q}(A,B)$ holds for any pair of quasi-geometric
type means. We say that a norm $\|\cdot\|$ on $\bM_n$ is \emph{strictly increasing} if $A\ge B\ge0$ and
$\|A\|=\|B\|$ imply $A=B$. The Schatten $p$-norms where $1\le p<\infty$ (particularly, the trace norm) are
typical examples of strictly increasing unitarily invariant norms.

\begin{theorem}\label{T-4.1}
Let $\alpha\in(0,\infty)\setminus\{1\}$ and $p,q>0$. Consider the following pairs of quasi-geometric matrix
means:
\begin{align}
&(R_{\alpha,p},R_{\alpha,q})\quad\mbox{for}\ \alpha\in(0,\infty)\setminus\{1\},\ p\ne q, \label{F-4.1}\\
&(G_{\alpha,p},R_{\alpha,q})\quad\mbox{for}\ \begin{cases}
0<\alpha<1,\ p,q>0,\ \mbox{or} \\
\alpha>1,\ p/q<\min\{\alpha/2,\alpha-1\},\ \mbox{or} \\
\alpha>1,\ p/q>\max\{\alpha/2,\alpha-1\},\end{cases} \label{F-4.2}\\
&(SG_{\alpha,p},R_{\alpha,q})\quad\mbox{for}\ \begin{cases}
0<\alpha<1,\ p/q>\max\{\alpha,1-\alpha\},\ \mbox{or} \\
0<\alpha<1,\ p/q<\min\{\alpha,1-\alpha\},\ \mbox{or} \\
\alpha>1,\ p/q<\alpha,\end{cases} \label{F-4.3}\\
&(\widetilde SG_{\alpha,p},R_{\alpha,q})\quad\mbox{for}\ \begin{cases}
0<\alpha<1,\ p/q<\alpha,\ \mbox{or} \\
0<\alpha\le1/2,\ q<p,\ \mbox{or} \\
\alpha>1,\ p/q>\alpha,\end{cases} \label{F-4.4}\\
&(LE_\alpha,R_{\alpha,p})\quad\mbox{for}\ \alpha\in(0,\infty)\setminus\{1\},\ p>0, \label{F-4.5}\\
&(G_{\alpha,p},G_{\alpha,q})\quad\mbox{for}\ \alpha\in(0,2]\setminus\{1\},\ p\ne q, \label{F-4.6}\\
&(SG_{\alpha,p},G_{\alpha,q})\quad\mbox{for}\ \begin{cases}
0<\alpha<1,\ p,q>0,\ \mbox{or} \\
1<\alpha\le2,\ p/q<\max\bigl\{2,{\alpha\over\alpha-1}\bigr\},\end{cases} \label{F-4.7}\\
&(\widetilde SG_{\alpha,p},G_{\alpha,q})\quad\mbox{for}\ \begin{cases}
0<\alpha\le1/2,\ p,q>0,\ \mbox{or} \\
1<\alpha\le2,\ q/p<\min\bigl\{1/2,{\alpha-1\over\alpha}\bigr\},\end{cases} \label{F-4.8}\\
&(LE_\alpha,G_{\alpha,p})\quad\mbox{for}\ \alpha\in(0,2]\setminus\{1\},\ p>0. \label{F-4.9}
\end{align}
Let $\|\cdot\|$ be any strictly increasing unitarily invariant norm. Then the following hold:
\begin{itemize}
\item[(1)] Let $(\cM_{\alpha,p},\cN_{\alpha,q})$ be any of \eqref{F-4.1}--\eqref{F-4.5} and $A,B\ge0$, with
an additional assumption $s(A)\ge s(B)$ for the $\alpha>1$ case of \eqref{F-4.1}, \eqref{F-4.2} and
\eqref{F-4.5}, and for \eqref{F-4.3} and \eqref{F-4.4}. If $\|\cM_{\alpha,p}(A,B)\|=\|\cN_{\alpha,q}(A,B)\|$,
then $AB=BA$
\item[(2)] Let $(\cM_{\alpha,p},\cN_{\alpha,q})$ be any of \eqref{F-4.6}--\eqref{F-4.9} and $A,B>0$. If
$\|\cM_{\alpha,p}(A,B)\|=\|\cN_{\alpha,q}(A,B)\|$, then $AB=BA$
\end{itemize}
\end{theorem}

\begin{proof}
(1)\enspace
The main of this part is the assertion for \eqref{F-4.1}, which was formerly shown in \cite[Theorem 2.1]{Hi1}
and the proof has been updated in \cite[Appendix A]{Hi3}.\footnote{
We note that the claim \cite[(A.3)]{Hi3} and its reasoning in the last five lines of \cite[p.\ 544]{Hi3} are incorrect.
But (A.3) is unnecessary in the proof of \cite[Appendix]{Hi3}. Therefore, these lines and Remark A.2 of
\cite{Hi3} should be deleted.}
The other cases in \eqref{F-4.2}--\eqref{F-4.5} are all reduced to \eqref{F-4.1}. Indeed, assume that
$\|G_{\alpha,p}(A,B)\|=\|R_{\alpha,q}(A,B)\|$ for $\alpha,p,q$ satisfying one of the three conditions in
\eqref{F-4.2}. When either $0<\alpha<1$ with $p,q>0$ arbitrary, or $\alpha>1$ and
$p/q<\min\{\alpha/2,\alpha-1\}$, choose a $q'\in(0,q)$ such that $\alpha,p,q'$ satisfy the same condition.
By Theorem \ref{T-3.1}(a) and (e) one has
\[
\|G_{\alpha,p}(A,B)\|\le\|R_{\alpha,q'}(A,B)\|\le\|R_{\alpha,q}(A,B)\|.
\]
Hence $\|R_{\alpha,q'}(A,B)\|=\|R_{\alpha,q}(A,B)\|$, which gives $AB=BA$ by the \eqref{F-4.1} case. When
$\alpha>1$ and $p/q>\max\{\alpha/2,\alpha-1\}$, choose a $q'>q$ such that $\alpha,p,q'$ satisfy the same
condition. By Theorem \ref{T-3.1}(a) and (f) one has
\[
\|R_{\alpha,q}(A,B)\|\le\|R_{\alpha,q'}(A,B)\|\le\|G_{\alpha,p}(A,B)\|,
\]
so that $\|R_{\alpha,q}(A,B)\|=\|R_{\alpha,q'}(A,B)\|$ and $AB=BA$ follows. The proof is similar for
\eqref{F-4.3} by using Theorems \ref{T-3.1}(g) (or \ref{T-3.3}(2)), \ref{T-3.3}(1) and \ref{T-3.4}(1) for the three
cases in \eqref{F-4.3}, respectively. The proof for \eqref{F-4.4} is also similar by
Theorems \ref{T-3.6}(1), \ref{T-3.6}(2) and \ref{T-3.7}(2) for the three cases in \eqref{F-4.4}, respectively.
For \eqref{F-4.5} we may use Proposition \ref{P-3.17}(a).

(2)\enspace
The main of this part is the assertion for \eqref{F-4.6}, which was formerly shown for $0<\alpha<1$ in
\cite[Theorem 3.1]{Hi1}. The assertion for $1<\alpha\le2$ is new. We defer its details to Appendix A because
it seems more instructive to present them independently. The other cases in \eqref{F-4.7}--\eqref{F-4.9} are
all reduced to \eqref{F-4.6}. For instance, assume that $\|SG_{\alpha,p}(A,B)\|=\|G_{\alpha,q}(A,B)\|$ for
$\alpha,p,q$ satisfying one of the two conditions in \eqref{F-4.7}. When $0<\alpha<1$ with $p,q>0$ arbitrary,
for a $q'\in(0,q)$ by Theorem \ref{T-3.1}(b) and Proposition \ref{P-3.13}(1) one has
\[
\|G_{\alpha,q}(A,B)\|\le\|G_{\alpha,q'}(A,B)\|\le\|SG_{\alpha,p}(A,B)\|,
\]
implying that $\|G_{\alpha,q}(A,B)\|=\|G_{\alpha,q'}(A,B)\|$ so that $AB=BA$ from the \eqref{F-4.6} case.
When $1<\alpha\le2$ and $p/q<\max\bigl\{2,{\alpha\over\alpha-1}\bigr\}$, choose a $q'\in(0,q)$ such that
$\alpha,p,q'$ satisfy the same condition. By Theorem \ref{T-3.1}(c) and Proposition \ref{P-3.13}(3) one has
\[
\|SG_{\alpha,p}(A,B)\|\le\|G_{\alpha,q'}(A,B)\|\le\|G_{\alpha,q}(A,B)\|,
\]
implying that $\|G_{\alpha,q'}(A,B)\|=\|G_{\alpha,q}(A,B)\|$ so that $AB=BA$. The proof is similar for
\eqref{F-4.8} by using Proposition \ref{P-3.14}(2) and (3). For \eqref{F-4.9}, (b) and (c) of
Proposition \ref{P-3.17} are used for $0<\alpha<1$ and $1<\alpha\le2$, respectively.
\end{proof}

The next proposition refines the log-majorizations given in (a), (b) and (c) of Theorem \ref{T-3.1} into the
local characterization in the sense that those log-majorizations are characterized for any fixed
non-commuting $A,B>0$.

\begin{proposition}\label{P-4.2}
Let $p,q>0$ and $A,B>0$ be such that $AB\ne BA$.
\begin{itemize}
\item[(1)] For any $\alpha\in(0,\infty)\setminus\{1\}$ we have
\[
R_{\alpha,p}(A,B)\prec_{\log}R_{\alpha,q}(A,B)\iff\Tr\,R_{\alpha,p}(A,B)\le\Tr\,R_{\alpha,q}(A,B)\iff p\le q.
\]
For $0<\alpha<1$ the above equivalences hold for $A,B\ge0$ with $AB\ne BA$.
\item[(2)] For $0<\alpha<1$ we have
\[
G_{\alpha,q}(A,B)\prec_{\log}G_{\alpha,p}(A,B)\iff\Tr\,G_{\alpha,q}(A,B)\le\Tr\,G_{\alpha,p}(A,B)\iff p\le q.
\]
\item[(3)] For $1<\alpha\le2$ we have
\[
G_{\alpha,p}(A,B)\prec_{\log}G_{\alpha,q}(A,B)\iff\Tr\,G_{\alpha,p}(A,B)\le\Tr\,G_{\alpha,q}(A,B)\iff p\le q.
\]
\end{itemize}
\end{proposition}

\begin{proof}
(1)\enspace
Let $A,B\ge0$ when $0<\alpha<1$, while $A,B>0$ when $\alpha>1$. If $p\le q$, then
$R_{\alpha,p}(A,B)\prec_{\log}R_{\alpha,q}(A,B)$ by Theorem \ref{T-3.1}(a). Obviously,
$R_{\alpha,p}(A,B)\prec_{\log}R_{\alpha,q}(A,B)$ implies $\Tr\,R_{\alpha,p}(A,B)\le\Tr\,R_{\alpha,q}(A,B)$.
Finally, assume that $\Tr\,R_{\alpha,p}(A,B)\le\Tr\,R_{\alpha,q}(A,B)$. If $p>q$, then we must have
$\Tr\,R_{\alpha,p}(A,B)=\Tr\,R_{\alpha,q}(A,B)$ by Theorem \ref{T-3.1}(a) again. This implies $AB=BA$ by
case \eqref{F-4.1} of Theorem \ref{T-4.1}(1), contradicting the assumption $AB\ne BA$. Hence (1) follows.

(2) and (3) are proved similarly to (1) by use of Theorem \ref{T-3.1}(b), (c) and case \eqref{F-4.6} of
Theorem \ref{T-4.1}(2) (i.e., Theorem \ref{T-A.1} in Appendix A).
\end{proof}

\begin{remark}\label{R-4.3}\rm
It is not possible to have the local characterization similar to Proposition \ref{P-4.2} for log-majorizaions
between $G_{\alpha,p}$ and $R_{\alpha,q}$ for $\alpha>1$. In fact, log-majorizations between
$G_{\alpha,p}$ and $R_{\alpha,q}$ are indefinite when $\alpha>1$ and
$\min\{\alpha/2,\alpha-1\}<p/q<\max\{\alpha/2,\alpha-1\}$; see Theorem \ref{T-3.1}(e) and (f).
\end{remark}

\section{Joint concavity/convexity}\label{Sec-5}

The aim of this section is to examine the joint concavity/convexity of the trace functions for quasi-geometric
type matrix means of our concern.

\subsection{Joint concavity/convexity and monotonicity}\label{Sec-5.1}

Before going to our main discussions of this section, we first review a known result on
the joint concavity of trace functions for Kubo--Ando's operator means and then a strong connection of the
joint concavity/convexity and the monotonicity property of our target trace functions. 

The next  theorem was shown in \cite[Theorem 3.2]{Hi4} in a more general form by a complex function
method using Pick functions (called Epstein's method).

\begin{theorem}\label{T-5.1}
Let $\sigma$ be a Kubo--Ando's operator mean. If $0\le p,q\le1$ and $0\le s\le1/\max\{p,q\}$, then the function
$(A,B)\mapsto\Tr(A^p\sigma B^q)^s$ is jointly concave on $\bM_n^+\times\bM_n^+$ for any $n\ge1$.
In particular, if $0<p\le1$, then $(A,B)\mapsto\Tr(A^p\sigma B^p)^{1/p}$ is jointly concave on
$\bM_n\times\bM_n^+$.
\end{theorem}

In particular, $\Tr\,\cA_{\alpha,p}$, $\Tr\,\cH_{\alpha,p}$ and $\Tr\,G_{\alpha,p}$ are jointly concave for
$0<\alpha\le1$ and $0<p\le1$. More specified results on the joint concavity/convexity of $\Tr\,\cA_{\alpha,p}$
and $\Tr\,\cH_{\alpha,p}$ are known as follows (see \cite{Bek,CL} and \cite[p.\,1583]{Hi4}):

\begin{proposition}\label{P-5.2}
Let $0<\alpha<1$, $p>0$, and let $n\in\bN$ with $n\ge2$ be arbitrarily fixed. Then the following hold:
\begin{itemize}
\item[(1)] $(A,B)\mapsto\Tr\,\cA_{\alpha,p}(A,B)$ is jointly concave on $\bM_n^+\times\bM_n^+$ if and
only if $0<p\le1$.
\item[(2)] $(A,B)\mapsto\Tr\,\cA_{\alpha,p}(A,B)$ is jointly convex on $\bM_n^+\times\bM_n^+$ if and
only if $1\le p\le2$.
\item[(3)] $(A,B)\mapsto\Tr\,\cH_{\alpha,p}(A,B)$ is jointly concave on $\bM_n^+\times\bM_n^+$ if and
only if $0<p\le1$.
\item[(4)] $(A,B)\mapsto\Tr\,\cH_{\alpha,p}(A,B)$ is not jointly convex on $\bM_n^+\times\bM_n^+$ for any
$p>0$.
\end{itemize}
\end{proposition}

Below we write
\[
(\bM_n^+\times\bM_n^+)_\ge:=\{(A,B)\in\bM_n^+\times\bM_n^+:s(A)\ge s(B)\}.
\]
Consider a function $Q:(A,B)\in\dom\,Q\mapsto Q(A,B)\in[0,+\infty)$ with the domain
\begin{align}\label{F-5.1}
\dom\,Q=\bigsqcup_{n\in\bN}(\bM_n^+\times\bM_n^+)\quad\mbox{or}\quad
\bigsqcup_{n\in\bN}(\bM_n^+\times\bM_n^+)_\ge.
\end{align}
Note that $\dom_nQ:=\bM_n^+\times\bM_n^+$ or $(\bM_n^+\times\bM_n^+)_\ge$ is a convex cone, and
if $(A_1,B_1)\in\dom_nQ$ and $(A_2,B_2)\in\dom_mQ$, then
$(A_1\oplus A_2,B_1\oplus B_2)\in\dom_{n+m}Q$ and $(A_1\otimes A_2,B_1\otimes B_2)\in\dom_{nm}Q$.
Consider the following properties of $Q$:
\begin{itemize}
\item[(a)] Normalization identity: $Q(A,A)=\Tr\,A$ for all $A\ge0$.
\item[(b)] Positive homogeneity: $Q(\lambda A,\lambda B)=\lambda Q(A,B)$ for all
$(A,B)\in\dom\,Q$ and any $\lambda\ge0$.
\item[(c)] Additivity under direct sums: $Q(A_1\oplus A_2,B_1\oplus B_2)=
Q(A_1,B_1)+Q(A_2,B_2)$ for all $(A_i,B_i)\in\dom\,Q$, $i=1,2$.
\item[(d)] Multiplicativity under tensor products: $Q(A_1\otimes A_2,B_1\otimes B_2)
=Q(A_1,B_1)Q(A_2,B_2)$ for all $(A_i,B_i)\in\dom\,Q$, $i=1,2$.
\item[(e)] Unitary invariance: $Q(UAU^*,UBU^*)=Q(A,B)$ for all $(A,B)\in\dom_n\,Q$
and all unitaries $U\in\bM_n$, $n\in\bN$.
\end{itemize}

A linear map $\Phi:\bM_n\to\bM_m$ is positive if $A\in\bM_n^+$ implies $\Phi(A)\in\bM_m^+$. It is said
to be \emph{completely positive} if for each $k\in\bN$ the map $\id_k\otimes\Phi$ defined by
$(\id_k\otimes \Phi)([A_{ij}]_{i,j=1}^k)=[\Phi(A_{ij})]_{i,j=1}^k$ is positive. When $\Phi$ is completely positive
and \emph{trace-preserving}, i.e., $\Tr\,\Phi(A)=\Tr\,A$ for all $A\in\bM_n$, $\Phi$ is called a
\emph{CPTP map} or a \emph{quantum channel}. This is a very important notion in quantum information.
Note that if $\Phi:\bM_n\to\bM_m$ is positive and $(A,B)\in(\bM_n^+\times\bM_n^+)_\ge$, then
$(\Phi(A),\Phi(B))\in(\bM_m^+\times\bM_m^+)_\ge$.

The next theorem is rather well known to experts on quantum R\'enyi divergences, while we give a proof
in Appendix \ref{Sec-B} for the convenience of the reader.

\begin{theorem}\label{T-5.3}
Let $Q:\dom\,Q\to[0,+\infty)$ be as given above with the domain in \eqref{F-5.1}. Consider the following
conditions:
\begin{itemize}
\item[(i)] $Q(\Phi(A),\Phi(B))\ge Q(A,B)$ for all $(A,B)\in\dom_n\,Q$ and all CPTP maps
$\Phi:\bM_n\to\bM_m$, $n,m\in\bN$.
\item[(i\,$'$)] $Q(\Phi(A),\Phi(B))\le Q(A,B)$ for all $(A,B)\in\dom_n\,Q$ and all CPTP maps
$\Phi:\bM_n\to\bM_m$, $n,m\in\bN$.
\item[(ii)] $Q$ is jointly concave on $\dom_n\,Q$ for all $n\in\bN$.
\item[(ii\,$'$)] $Q$ is jointly convex on $\dom_n\,Q$ for all $n\in\bN$.
\end{itemize}
Then the following hold:
\begin{itemize}
\item[(1)] If (b) and (c) are assumed, then we have (i)$\implies$(ii) and (i\,$'$)$\implies$(ii\,$'$).
\item[(2)] If (a), (d) and (e) are assumed, then we have (ii)$\implies$(i) and (ii\,$'$)$\implies$(i\,$'$).
\item[(3)] Hence we have (i)$\iff$(ii) and (i\,$'$)$\iff$(ii\,$'$) under (a)--(e).
\end{itemize}
\end{theorem}

\subsection{Quantum R\'enyi type divergences}\label{Sec-5.2}

In this subsection we briefly review quantum $\alpha$-R\'enyi type divergences from quantum information
and present certain necessary conditions for the joint concavity/convexity of our quasi-geometric type matrix
means (see Theorem \ref{T-5.8} below).

\begin{definition}\label{D-5.4}\rm
Let $\alpha\in(0,\infty)\setminus\{1\}$. The classical $\alpha$-R\'enyi divergence $D_\alpha^\cl(b\|a)$ is
defined for non-negative vectors $a,b\in[0,\infty)^n$ with $b\ne0$ by
\[
D_\alpha^\cl(b\|a):=\begin{cases}
{1\over\alpha-1}\log{\sum_{i=1}^nb_i^\alpha a_i^{1-\alpha}\over\sum_{i=1}^nb_i}
& \text{if $0<\alpha<1$ or $\supp\,a\supset\supp\,b$}, \\
+\infty & \text{if $\alpha>1$ and $\supp\,a\not\supset\supp\,b$},\end{cases}
\]
where $\supp\,a:=\{i:a_i>0\}$ and $0^{1-\alpha}:=0$ for $\alpha>1$. We say that a function
\[
D_\alpha^\q:(A,B)\in\bigsqcup_{n\in\bN}(\bM_n^+\times(\bM_n^+\setminus\{0\}))
\mapsto D_\alpha^\q(B\|A)\in(-\infty,+\infty]
\]
is a \emph{quantum $\alpha$-R\'enyi type divergence} if it is invariant under isometries, i.e.,
\[
D_\alpha^\q(VBV^*\|VAV^*)=D_\alpha^\q(B\|A),\qquad A,B\in\bM_n^+,\ B\ne0,\ n\in\bN
\]
for any isometry $V:\bC^n\to\bC^m$ and satisfies
\[
D_\alpha^\q(\diag(b)\|\diag(a))=D_\alpha^\cl(b\|a),\qquad a,b\in[0,\infty)^n,\ b\ne0,\ n\in\bN,
\]
where $\diag(a)$ is the diagonal matrix whose diagonals are entries of $a$.
\end{definition}

When $A,B\in\bM_n^+$ are commuting, one can write $A=\sum_{i=1}^na_i|e_i\>\<e_i|$ and
$B=\sum_{i=1}^nb_i|e_i\>\<e_i|$ for some orthonormal basis $(e_i)_{i=1}^n$ of $\bC^n$ and by unitary
invariance one has
\[
D_\alpha^q(B\|A)=D_\alpha^q(\diag(b)\|\diag(a))=D_\alpha^\cl(b\|a).
\]
In view of this, we may write $D_\alpha^q(B\|A)=D_\alpha^\cl(B\|A)$ when $AB=BA$. There are the notions
of the minimal and the maximal quantum $\alpha$-R\'enyi divergences. The maximal one is Matsumoto's
\emph{maximal $\alpha$-R\'enyi divergence} \cite{Ma} defined by
\begin{align}\label{F-5.2}
D_\alpha^{\max}(B\|A):=\inf\{D_\alpha^\cl(b\|a):\Gamma(a)=A,\ \Gamma(b)=B\},
\end{align}
where the infimum is taken over triplets $(\Gamma,a,b)$ (called \emph{reverse tests}) consisting of
$a,b\in[0,\infty)^m$ and a (completely) positive trace-preserving map $\Gamma:\bC^m\to\bM_n$ satisfying
$\Gamma(a)=A$, $\Gamma(b)=B$ for some $m\in\bN$. The minimal one is the so-called
\emph{measured $\alpha$-R\'enyi divergence} defined for $A,B\in\bM_n^+$ by
\begin{align}\label{F-5.3}
D_\alpha^\meas(B\|A)
:=\sup\bigl\{D_\alpha^\cl\bigl((\Tr\,M_iB)_{i=1}^k\big\|(\Tr\,M_iA)_{i=1}^k\bigr):
\mbox{$(M_i)_{i=1}^k$ is a POVM on $\bC^n$, $k\in\bN$}\bigr\},
\end{align}
where a POVM on $\bC^n$ is a family $(M_i)_{i=1}^k\subset\bM_n^+$ with $\sum_{i=1}^kM_i=I_n$ (i.e., a
$k$-outcome positive operator-valued measure). Furthermore, the \emph{regularized measured
$\alpha$-R\'enyi divergence} is defined by
\begin{align}\label{F-5.4}
\overline D_\alpha^\meas(B\|A):=\sup_{m\ge1}{1\over m}\,D_\alpha^\meas(B^{\otimes m}\|A^{\otimes m})
=\lim_{m\to\infty}{1\over m}\,D_\alpha^\meas(B^{\otimes m}\|A^{\otimes m}),
\end{align}
where the last equality holds since $m\in\bN\mapsto D_\alpha^\meas(B^{\otimes m}\|A^{\otimes m})$ is a
superadditive sequence.

Now assume that a quantum $\alpha$-R\'enyi type divergence $D_\alpha^\q$ satisfies the monotonicity
under CPTP maps, i.e.,
\[
D_\alpha^\q(\Phi(B)\|\Phi(A))\le D_\alpha^\q(B\|A),\qquad A,B\ge0
\]
for any CPTP map $\Phi$. Then from definitions \eqref{F-5.2} and \eqref{F-5.3} it follows that
\begin{align}\label{F-5.5}
D_\alpha^\meas(B\|A)\le D_\alpha^\q(B\|A)\le D_\alpha^{\max}(B\|A),\qquad A,B\ge0.
\end{align}
Indeed, for any POVM $(M_i)_{i=1}^k$ on $\bC^n$, since $\Phi(A):=\diag((\Tr\,M_iA)_{i=1}^k)$ for
$A\in\bM_n$ is a CPTP map, we have
\[
D_\alpha^\cl\bigl((\Tr\,M_iB)_{i=1}^k\big\|(\Tr\,M_iA)_{i=1}^k\bigr)
=D_\alpha^\q\bigl(\diag((\Tr\,M_iB)_{i=1}^k)\big\|\diag((\Tr\,M_iA)_{i=1}^k)\bigr)
\le D_\alpha^\q(B\|A),
\]
so that the first inequality in \eqref{F-5.5} follows from \eqref{F-5.3}. On the other hand, for any reverse
test $(\Gamma,a,b)$ with $\Gamma:\bC^m\to\bM_n$ and $a,b\in[0,\infty)^m$, define $\cE:\bM_m\to\bC^m$
by $A=[a_{ij}]_{i,j=1}^m\mapsto(a_{11},\dots,a_{mm})$ and $\Phi:=\Gamma\circ\cE:\bM_m\to\bM_n$; then
$\Phi$ is a CPTP map and $\Phi(\diag(a))=\Gamma(a)$, $a\in\bC^m$. Since
\[
D_\alpha^\q(B\|A)=D_\alpha^\q(\Gamma(b)\|\Gamma(a))=D_\alpha^\q(\Phi(\diag(b))\|\Phi(\diag(a))\bigr)
\le D_\alpha^\q(\diag(b)\|\diag(a))=D_\alpha^\cl(b\|a),
\]
we have the second inequality in \eqref{F-5.5}. If, in addition, $D_\alpha^\q$ is additive under tensor
products, then by \eqref{F-5.5} with definition \eqref{F-5.4} we furthermore have
\begin{align}\label{F-5.6}
\overline D_\alpha^\meas(B\|A)\le D_\alpha^\q(B\|A),\qquad A,B\ge0.
\end{align}

The so-called \emph{$\alpha$-$z$-R\'enyi divergence} $D_{\alpha,z}$ \cite{AD} for
$\alpha\in(0,\infty)\setminus\{1\}$ and $z>0$ is typical among quantum $\alpha$-R\'enyi type divergences.
Their most important special cases are the \emph{Petz type $\alpha$-R\'enyi divergence}
$D_\alpha=D_{\alpha,1}$ \cite{Pe} for $z=1$ and the \emph{sandwiched $\alpha$-R\'enyi divergence}
$\widetilde D_\alpha=D_{\alpha,\alpha}$ \cite{MDSFT,WWY} for $z=\alpha$. See, e.g., \cite{MoHi} for further
details on quantum $\alpha$-R\'enyi divergences.


One can easily check the next lemma, whose proofs may be omitted.

\begin{lemma}\label{L-5.5}
Let $\cM_{\alpha,p}$ be any of $R_{\alpha,p}$, $G_{\alpha,p}$, $SG_{\alpha,p}$, $\widetilde SG_{\alpha,p}$
and $LE_\alpha$ with any $\alpha\in(0,\infty)\setminus\{1\}$ and $p>0$. Then $Q:=\Tr\,\cM_{\alpha,p}$
satisfies all properties of (a)--(e) above.
\end{lemma}

For a quasi matrix mean $\cM_{\alpha,p}$ as in Lemma \ref{L-5.5} it is meaningful to associate the quantum
divergence of R\'enyi type with two parameters $\alpha\in(0,\infty)\setminus\{1\}$ and $p>0$ as follows: for
$A,B\in\bM_n^+$ with $B\ne0$,
\begin{align}\label{F-5.7}
D^{\cM_{\alpha,p}}(B\|A):=\begin{cases}
{1\over\alpha-1}\log{\Tr\,\cM_{\alpha,p}(A,B)\over\Tr\,B} &
\text{if $(A,B)\in\dom\,\cM_{\alpha,p}$}, \\
+\infty & \text{otherwise},
\end{cases}
\end{align}
where $\dom\,\cM_{\alpha,p}$ has been fixed in Section \ref{Sec-2} (see the paragraph after \eqref{F-2.8}).
Here, note that the orders of the variables $A,B$ in $\cM_{\alpha,p}(A,B)$ and $D^{\cM_{\alpha,p}}(B\|A)$
are opposite; we follow the convention of Kubo--Ando's operator means for the former and the usual
convention of quantum divergences for the latter. The next lemma is also easy to verify.

\begin{lemma}\label{L-5.6}
Let $p>0$. Either when $\cM\in\{R,G,LE\}$ and $\alpha\in(0,\infty)\setminus\{1\}$, or when
$\cM\in\{SG,\widetilde SG\}$ and $\alpha>1$, the function $D^{\cM_{\alpha,p}}$ defined in \eqref{F-5.7} is
a quantum $\alpha$-R\'enyi type divergence in the sense of Definition \ref{D-5.4}, and is additive under tensor
products.

When $\cM\in\{SG,\widetilde SG\}$ and $0<\alpha<1$, $D^{\cM_{\alpha,p}}$ is invariant under isometries
and satisfies $D^{\cM_{\alpha,p}}(\diag(b)\|\diag(a))=D_\alpha^\cl(b\|a)$ for all $a,b\in[0,\infty)$, $b\ne0$,
with $\supp\,a\supset\supp\,b$, and it is additive under tensor product.
\end{lemma}

We note that $D^{R_{\alpha,1/z}}$ becomes the $\alpha$-$z$-R\'enyi divergence $D_{\alpha,z}$; in particular,
$D^{R_{\alpha,1}}$ is the Petz type $D_\alpha$ and $D^{R_{\alpha,1/\alpha}}$ is the sandwiched
$\widetilde D_\alpha$. Also, $D^{G_{\alpha,1}}$ is the \emph{maximal $f$-divergence} \cite{HM} associated to
the function $f(x):=x^\alpha$ ($x>0$). Here, in the next theorem we recall a few results related to the minimal
and the maximal bounds in \eqref{F-5.5} and \eqref{F-5.6}, as we will use those below.

\begin{theorem}\label{T-5.7}
Let $\alpha\in(0,\infty)\setminus\{1\}$ and $A,B\ge0$, $B\ne0$.
\begin{itemize}
\item[(1)] (\cite{Ma})\enspace
When $0<\alpha\le2$, $D_\alpha^{\max}(B\|A)=D^{G_{\alpha,1}}(B\|A)$.
\item[(2)] (\cite[Remark II.6]{MoHi})\enspace
When $\alpha>2$, $D_\alpha^{\max}(B\|A)\le D^{G_{\alpha,1}}(B\|A)$ and
$D_\alpha^{\max}\ne D^{G_{\alpha,1}}$.
\item[(3)] (\cite[Theorem 3.7]{MoOg}, \cite{HaTo})\enspace
When $1/2\le\alpha<\infty$,
\[
\overline{D}_\alpha^\meas(B\|A)=\widetilde D_\alpha(B\|A)
=\lim_{m\to\infty}{1\over m}D_\alpha^\cl(\cE_{A^{\otimes m}}(B^{\otimes m})\|A^{\otimes m}),
\]
where $\cE_{A^{\otimes m}}$ is the pinching with respect to $A^{\otimes m}$.
\item[(4)] (\cite[Theorem 7]{BFT})\enspace
When $0<\alpha<1/2$,
$\widetilde D_\alpha(B\|A)\le D_\alpha^\meas(B\|A)\le\overline{D}_\alpha^\meas(B\|A)$, and
$\widetilde D_\alpha(B\|A)<D_\alpha^\meas(B\|A)$ if $AB\ne BA$.
\end{itemize}
\end{theorem}

For any $\cM_{\alpha,p}$ in Lemma \ref{L-5.5}, since $\cM_{\alpha,p}(a,b)=a^{1-\alpha}b^\alpha$ in
$a,b\in(0,\infty)$ is not convex when $0<\alpha<1$ and not concave when $\alpha>1$, we may consider only
the joint concavity of $\Tr\,\cM_{\alpha,p}$ when $0<\alpha<1$, and the joint convexity of $\Tr\,\cM_{\alpha,p}$
when $\alpha>1$. Below we will simply write ``$\Tr\,\cM_{\alpha,p}$ is jointly concave/convex'' if it is so on the
domain of $\cM_{\alpha,p}$, equivalently, if it is so on $\bM_n^{++}\times\bM_n^{++}$ for any $n\in\bN$
(thanks to Proposition \ref{P-2.1}(2)).

The next theorem will repeatedly be used in the rest of this section.

\begin{theorem}\label{T-5.8}
Let $\cM_{\alpha,p}$ be as in Lemma \ref{L-5.5}.
\begin{itemize}
\item[(i)] If $\Tr\,\cM_{\alpha,p}$ is jointly concave, then $0<\alpha<1$ and
\[
\Tr\,G_{\alpha,1}(A,B)\le\Tr\,\cM_{\alpha,p}(A,B)\le\Tr\,R_{\alpha,1/\alpha}(A,B)
\]
for all $(A,B)\in\dom\,\cM_{\alpha,p}$.
\item[(ii)] If $\Tr\,\cM_{\alpha,p}$ is jointly convex, then $\alpha>1$ and
\[
\Tr\,R_{\alpha,1/\alpha}(A,B)\le\Tr\,\cM_{\alpha,p}(A,B)\le\Tr\,G_{\alpha,1}(A,B)
\]
for all $(A,B)\in\dom\,\cM_{\alpha,p}$
\end{itemize}
\end{theorem}

\begin{proof}
Assume that $\Tr\,\cM_{\alpha,p}$ is jointly concave as in (i), or jointly convex as in (ii). From the remark
mentioned just above the theorem, the parameter $\alpha$ is restricted to $0<\alpha<1$ in (i) and to
$\alpha>1$ in (ii). Unless $\cM\in\{SG,\widetilde SG\}$ and $0<\alpha<1$, it follows from Lemmas \ref{L-5.5},
\ref{L-5.6} and Theorem \ref{T-5.3} that $D^{\cM_{\alpha,p}}$ is a quantum $\alpha$-R\'enyi type divergence
satisfying the monotonicity under CPTP maps and it is moreover additive under tensor products. Therefore,
we can apply \eqref{F-5.5} and \eqref{F-5.6} to $D^{\cM_{\alpha,p}}$ to obtain
\begin{align}\label{F-5.8}
\overline D_\alpha^\meas(B\|A)\le D^{\cM_{\alpha,p}}(B\|A)\le D_\alpha^{\max}(B\|A)
\end{align}
for all $(A,B)\in\dom\,\cM_{\alpha,p}$, $B\ne0$.
When $\cM\in\{SG,\widetilde SG\}$ and $0<\alpha<1$, even though $D^{\cM_{\alpha,p}}$ is not precisely a
quantum $\alpha$-R\'enyi type divergence in Definition \ref{D-5.4}, we can still have \eqref{F-5.8} for all
$(A,B)\in\dom\,\cM_{\alpha,p}$. In fact, Theorem \ref{T-5.3} implies that $D^{\cM_{\alpha,p}}$ is monotone
under CPTP maps with restriction to $\dom\,\cM_{\alpha,p}$. Hence by Lemma \ref{L-5.6} the first inequality
in \eqref{F-5.8} holds on $\dom\,\cM_{\alpha,p}$. Also, note (see \cite{Ma} and the proof of
\cite[Proposition 4.1]{HM}) that for any $(A,B)\in\dom\,\cM_{\alpha,p}$ (hence $s(A)\ge s(B)$) the infimum in
\eqref{F-5.2} is attained by a reverse test (called the ``minimal reverse test'') $(\Gamma,a,b)$ with
$\supp\,a\supset\supp\,b$. Hence the second inequality in \eqref{F-5.8} holds on $\dom\,\cM_{\alpha,p}$.

From \eqref{F-5.8} and (1)--(4) of Theorem \ref{T-5.7} altogether we have
\[
D^{R_{\alpha,1/\alpha}}(B\|A)=\widetilde D_\alpha(B\|A)\le D^{\cM_{\alpha,p}}(B\|A)
\le D^{G_{\alpha,1}}(B\|A)
\]
for all $(A,B)\in\dom\,\cM_{\alpha,p}$, $B\ne0$. In view of \eqref{F-5.7} this shows the inequalities in
(i) ($0<\alpha<1$) and in (ii) ($\alpha>1$) for all $(A,B)\in\dom\,\cM_{\alpha,p}$ except for the case $B=0$.
But note that $\Tr\,\cM_{\alpha,p}(A,0)=0$ for all $A\ge0$ and all $\cM_{\alpha,p}$'s.
\end{proof}

\subsection{Joint concavity/convexity of trace functions}\label{Sec-5.3}

In this subsection we examine the joint concavity/convexity of $\Tr\,R_{\alpha,p}$, $\Tr\,LE_\alpha$,
$\Tr\,G_{\alpha,p}$, $\Tr\,SG_{\alpha,p}$ and $\Tr\,\widetilde SG_{\alpha,p}$ with the domains described in
Section \ref{Sec-2}.

The question of the joint concavity/convexity of $\Tr\,R_{\alpha,p}$ was an intriguing issue because it is
equivalent to the monotonicity under CPTP maps of $D_{\alpha,z}$ (see Theorem \ref{T-5.3}). It was finally
settled by Zhang \cite{Zh} as follows:

\begin{theorem}[\cite{Zh}]\label{T-5.9}
Let $\alpha\in(0,\infty)\setminus\{1\}$ and $p>0$.
\begin{itemize}
\item[(1)] $\Tr\,R_{\alpha,p}$ is jointly concave if and only if $0<\alpha<1$ and $1/p\ge\max\{\alpha,1-\alpha\}$.
\item[(2)] $\Tr\,R_{\alpha,p}$ is jointly convex if and only if $\alpha>1$ and
$\max\{\alpha/2,\alpha-1\}\le1/p\le\alpha$.
\end{itemize}
\end{theorem}

It is worthwhile to compare the condition of Theorem \ref{T-5.9}(2) with the $p=1$ (and $q$ replaced with $p$)
case of (iii) of Theorem \ref{T-3.2}(2). Notice that the former condition is strictly stronger than the latter though
they are quite similar.

\begin{remark}\label{R-5.10}\rm
Note that the `only if' parts of (1) and (2) of Theorem \ref{T-5.9} are verified by Theorem \ref{T-5.8}. Indeed,
if $\Tr\,R_{\alpha,p}$ is jointly concave, then we have $0<\alpha<1$ and $1/p\ge\alpha$ by
Theorem  \ref{T-5.8} and Proposition \ref{P-4.2}(1), so we have $1/p\ge1-\alpha$ too since
$R_{\alpha,p}(A,B)=R_{1-\alpha,p}(B,A)$. If $\Tr\,R_{\alpha,p}$ is jointly convex, then we have $\alpha>1$
and $\max\{\alpha/2,\alpha-1\}\le1/p\le\alpha$ by Proposition \ref{P-4.2}(1) and Theorem \ref{T-3.2}(2)
(with Remark \ref{R-2.4}(4)).
\end{remark}

\begin{remark}\label{R-5.11}\rm
Although Theorem \ref{T-5.9} implies the monotonicity of $D_{\alpha,z}$ under CPTP maps for $z=1/p$
satisfying the conditions in (1) and (2), it is in fact known \cite{Ka,HiJe} (even in the von Neumann algebra
setting) that $D_{\alpha,z}$ satisfies the monotonicity under general positive (not necessarily completely
positive) trace-preserving maps for the same $z$.
\end{remark}

The next theorem is concerned with the joint concavity/convexity of $\Tr\,LE_\alpha$.

\begin{theorem}\label{T-5.12}
Let $\alpha$ be as above.
\begin{itemize}
\item[(1)] For any $\alpha\in(0,1)$, $\Tr\,LE_\alpha$ is jointly concave.
\item[(2)] For any $\alpha>1$, $\Tr\,LE_\alpha$ is not jointly convex on $\bM_2^{++}\times\bM_2^{++}$.
\end{itemize}
\end{theorem}

\begin{proof}
(1)\enspace
It follows from Theorem \ref{T-5.9}(1) that $\Tr\,R_{\alpha,p}$ is jointly concave for all $p>0$ with
$1/p\ge\max\{\alpha,1-\alpha\}$. Letting $p\searrow0$ shows the assertion due to Theorem \ref{T-2.2}.

(2)\enspace
Let $\alpha>1$. Let $A:=\begin{bmatrix}a&0\\0&b\end{bmatrix}$ with $a,b>0$, $a\ne b$, and
$B:={1\over2}\begin{bmatrix}1&1\\1&1\end{bmatrix}$, and let $\cE_A$ be the pinching with respect to $A$.
It was shown in \cite[Lemma 3.17]{MoOg2} that $\Tr\,LE_\alpha(A,\cE_A(B))>\Tr\,LE_\alpha(A,B)$ for some
choice of $a,b$. (Note that $\Tr\,LE_\alpha(A,B)$ is denoted as $Q_\alpha^\flat(B\|A)$ in \cite{MoOg2}.) Let
$U:=\begin{bmatrix}1&0\\0&-1\end{bmatrix}$. Since $A={1\over2}(A+UAU^*)$,
$\cE_A(B)={1\over2}(B+UBU^*)$ and $\Tr\,LE_\alpha(A,B)=\Tr\,LE_\alpha(UAU^*,UBU^*)$, we have
\[
\Tr\,LE_\alpha\Bigl({A+UAU^*\over2},{B+UBU^*\over2}\Bigr)
>{\Tr\,LE_\alpha(A,B)+\Tr\,LE_\alpha(UAU^*,UBU^*)\over2},
\]
showing the result.
\end{proof}

\begin{remark}\label{R-5.13}\rm
In \cite[Theorem 3.6]{MoOg2} Mosonyi and Ogawa showed the variational expression
\begin{align}\label{F-5.9}
\Tr\,LE_\alpha(A,B)=\max_{X\in\bM_n^+,s(X)\le s(B)}\{\Tr\,X-(1-\alpha)D(X\|A)-\alpha D(X\|B)\}
\end{align}
for any $\alpha\in(0,\infty)\setminus\{1\}$ and for every $A,B\in\bM_n^+$, where $D$ is the Umegaki relative
entropy. Thanks to the well known joint convexity of $D$, expression \eqref{F-5.9} implies the joint concavity
in Theorem \ref{T-5.12}(1). Moreover, since $D$ satisfies the monotonicity under general positive
trace-preserving maps (see \cite{Bei,MuRe}), it easily follows from \eqref{F-5.9} that for every $\alpha\in(0,1)$
and for any positive trace-preserving map $\Phi:\bM_n\to\bM_m$,
\begin{align}\label{F-5.10}
\Tr\,LE_\alpha(\Phi(A),\Phi(B))\ge\Tr\,LE_\alpha(A,B),\qquad A,B\in\bM_n^+.
\end{align}
From Remark \ref{R-5.11} the same inequality holds for $\Tr\,R_{\alpha,p}$ if $1/p\ge\max\{\alpha,1-\alpha\}$.
Letting $p\searrow0$ as in the proof of Theorem \ref{T-5.12}(1) gives \eqref{F-5.10} as well.
\end{remark}

\begin{remark}\label{R-5.14}\rm
For $\alpha>1$, since $\Tr\,R_{\alpha,1/\alpha}\le\Tr\,LE_\alpha$ fails to hold (see
Proposition \ref{P-3.17}(a) and Remark \ref{R-3.16}(2)), we also confirm by
Theorem \ref{T-5.8}(ii) that $\Tr\,LE_\alpha$ is not jointly convex.
\end{remark}


The next theorem determines when $\Tr\,G_{\alpha,p}$ is jointly concave.

\begin{theorem}\label{T-5.15}
Let $\alpha,p$ be as in Theorem \ref{T-5.9}. Then $\Tr\,G_{\alpha,p}$ is jointly concave if and only if
$0<\alpha<1$ and $p\le1$.
\end{theorem}

\begin{proof}
If $0<\alpha<1$ and $0<p\le1$, then Theorem \ref{T-5.1} implies that $\Tr\,G_{\alpha,p}$ is jointly concave.
Conversely, assume the joint concavity of $\Tr\,G_{\alpha,p}$. Then by Theorem \ref{T-5.8}(i) we have
$0<\alpha<1$ and $\Tr\,G_{\alpha,1}\le\Tr\,G_{\alpha,p}$, which yields $p\le1$ by Proposition \ref{P-4.2}(2).
\end{proof}


For $\alpha>1$ the complete characterization of the joint convexity of $\Tr\,G_{\alpha,p}$ is not known,
while we give some partial results in the following:

\begin{proposition}\label{P-5.16}
Let $\alpha>1$ and $p>0$.
\begin{itemize}
\item[(1)] Assume that $\Tr\,G_{\alpha,p}$ is jointly convex. Then we have
$\max\bigl\{1/2,{\alpha-1\over\alpha}\bigr\}\le p\le1$.
\item[(2)] For $\alpha=2$, $\Tr\,G_{2,p}$ is jointly convex if and only if $1/2\le p\le1$.
\item[(3)] For $p=1$, $\Tr\,G_{\alpha,1}$ is jointly convex if and only if $1<\alpha\le2$.
\item[(4)] If $1<\alpha\le2$ and $0<p\le1$, then for any fixed $B\in\bM_n^+$ with any $n\in\bN$,
$A\mapsto\Tr\,G_{\alpha,p}(A,B)$ is convex on $\{A\in\bM_n^+:s(A)\ge s(B)\}$.
\end{itemize}
\end{proposition}

\begin{proof}
(1)\enspace
By Theorem \ref{T-5.8}(ii) the assumption implies that
\[
\Tr\,R_{\alpha,1/\alpha}(A,B)\le\Tr\,G_{\alpha,p}(A,B)\le\Tr\,G_{\alpha,1}(A,B)
\]
for all $A,B\in\bM_n^{++}$, $n\in\bN$. This in particular yields that
\[
R_{\alpha,1/\alpha}(A,B)\prec_{\log}G_{\alpha,p}(A,B)\prec_{\log}G_{\alpha,1}(A,B),
\qquad A,B\in\bM_2^{++}.
\]
From the above two log-majorizations we must have $\alpha p\ge\max\{\alpha/2,\alpha-1\}$ by
Theorem \ref{T-3.2}(2) and $p\le1$ by Proposition \ref{P-3.11}(1). Hence the result follows.

(2)\enspace
Since $\Tr\,G_{2,p}=\Tr\,R_{2,p}$ (Remark \ref{R-2.4}(3)), we see by Theorem \ref{T-5.9}(2) that
$\Tr\,G_{2,p}$ is jointly convex if and only if $1/2\le p\le1$.

(3)\enspace
The `if' part is a general result for the operator perspective for operator convex functions on $(0,\infty)$
due to \cite{ENG,Effros}. For the `only if' part, a more general and stronger result was shown in
\cite[Proposition A.1]{HM}; see Remark \ref{R-5.17} below.

(4)\enspace
In view of Proposition \ref{P-2.1}(2) 
it suffices to show that for any fixed $B\in\bM_n^{++}$, $A\mapsto\Tr\,G_{\alpha,p}(A,B)$ is convex on
$\bM_n^{++}$. For any $A,B\in\bM_n^{++}$ we can write
\[
\Tr\,G_{\alpha,p}(A,B)=\big\|B^{p/2}(B^{-p/2}A^pB^{-p/2})^{1-\alpha}B^{p/2}\big\|_{1/p}^{1/p}.
\]
Now let $A_1,A_2,B\in\bM_n^{++}$ and $0<\lambda<1$. Since
$(\lambda A_1+(1-\lambda)A_2)^p\ge\lambda A_1^p+(1-\lambda)A_2^p$ thanks to $0<p\le1$, one has
\[
B^{-p/2}(\lambda A_1+(1-\lambda)A_2)^pB^{-p/2}
\ge\lambda B^{-p/2}A_1^pB^{-p/2}+(1-\lambda)B^{-p/2}A_2^pB^{-p/2}.
\]
Since $-1\le1-\alpha<0$, note that $x^{1-\alpha}$ ($x>0$) is operator monotone decreasing and operator
convex. Hence one has
\begin{align*}
&B^{p/2}\bigl(B^{-p/2}(\lambda A_1+(1-\lambda)A_2)^pB^{-p/2}\bigr)^{1-\alpha}B^{p/2} \\
&\quad\le B^{p/2}\bigl(\lambda B^{-p/2}A_1^pB^{-p/2}
+(1-\lambda)B^{-p/2}A_2^pB^{-p/2}\bigr)^{1-\alpha}B^{p/2} \\
&\quad\le\lambda B^{p/2}\bigl(B^{-p/2}A_1^pB^{-p/2}\bigr)^{1-\alpha}B^{p/2}
+(1-\lambda)B^{p/2}\bigl(B^{-p/2}A_2^pB^{-p/2}\bigr)^{1-\alpha}B^{p/2}.
\end{align*}
Since $\|\cdot\|_{1/p}^{1/p}$ is convex on $\bM_n^+$, it follows that
\[
\Tr\,G_{\alpha,p}(\lambda A_1+(1-\lambda)A_2,B)
\le\lambda\Tr\,G_{\alpha,p}(A_1,B)+(1-\lambda)\Tr\,G_{\alpha,p}(A_2,B),
\]
as desired.
\end{proof}

\begin{remark}\label{R-5.17}\rm
In view of the proof of \cite[Proposition A.1]{HM} together with \cite[Proposition 3.1]{HT},
Proposition \ref{P-5.16}(3) is in fact improved in such a way that either if $A\mapsto\Tr\,G_{\alpha,1}(A,B)$
on $\bM_2^{++}$ for any fixed $B\in\bM_2^{++}$ is convex or if $B\mapsto\Tr\,G_{\alpha,1}(A,B)$ on
$\bM_2^{++}$ for any fixed $A\in\bM_2^{++}$ is convex, then $1<\alpha\le2$.
\end{remark}

\begin{problem}\label{Q-5.18}\rm
No sufficient condition for the joint convexity of $\Tr\,G_{\alpha,p}$ is known except the particular cases in
(2) and (3) of Proposition \ref{P-5.16}. It is of our special concern whether or not $\Tr\,G_{\alpha,p}$ is jointly
convex for $1<\alpha<2$ and $1/2\le p<1$ (see remark (7) in Section \ref{Sec-6}).
\end{problem}

As for $\Tr\,SG_{\alpha,p}$ and $\Tr\,\widetilde SG_{\alpha,p}$ we have the following:

\begin{proposition}\label{P-5.19}
Let $\alpha\in(0,\infty)\setminus\{1\}$ and $p>0$.
\begin{itemize}
\item[(1)] Let $0<\alpha<1$. If $\Tr\,SG_{\alpha,p}$ is jointly concave, then
$p\le\min\bigl\{{1-\alpha\over\alpha},{\alpha\over1-\alpha}\bigr\}$. If $\Tr\,\widetilde SG_{\alpha,p}$ is jointly
concave, then $p\le1$.
\item[(2)] For any $\alpha>1$ and any $p>0$, neither $\Tr\,SG_{\alpha,p}$ nor $\Tr\,\widetilde SG_{\alpha,p}$
is jointly convex.
\end{itemize}
\end{proposition}

\begin{proof}
(1)\enspace
For $0<\alpha<1$, if $\Tr\,SG_{\alpha,p}$ is jointly concave, then by Theorem \ref{T-5.8}(i),
\[
\Tr\,G_{\alpha,1}(A,B)\le\Tr\,SG_{\alpha,p}(A,B)\le\Tr\,R_{\alpha,1/\alpha}(A,B)
\]
for all $A,B\ge0$ with $s(A)\ge s(B)$. The first inequality gives no restriction (see Proposition \ref{P-3.13}(1)),
but by Theorem \ref{T-3.3}(1) the second does the restriction $\alpha p\le\min\{\alpha,1-\alpha\}$ so that
$p\le\min\bigl\{1,{1-\alpha\over\alpha}\bigr\}$. One can replace $\alpha$ with $1-\alpha$ to have
$p\le\min\bigl\{1,{\alpha\over1-\alpha}\}$. Therefore,
$p\le\min\bigl\{{1-\alpha\over\alpha},{\alpha\over1-\alpha}\bigr\}$ must hold. Next, if
$\Tr\,\widetilde SG_{\alpha,p}$ is jointly concave, then
$\Tr\,\widetilde SG_{\alpha,p}(A,B)\le\Tr\,R_{\alpha,1/\alpha}(A,B)$ for all $A,B\ge0$ with $s(A)\ge s(B)$,
which gives $p\le1$ by Theorem \ref{T-3.6}(1).

(2)\enspace
For $\alpha>1$, assume that $\Tr\,SG_{\alpha,p}$ is jointly convex; then by Theorem \ref{T-5.8}(ii),
$\Tr\,R_{\alpha,1/\alpha}\le\Tr\,SG_{\alpha,p}$ on $\bM_2^{++}\times\bM_2^{++}$ must hold. But this
contradicts Theorem \ref{T-3.4}(2). Assume that $\Tr\,\widetilde SG_{\alpha,p}$ is jointly convex. Then we
must have $\Tr\,\widetilde SG_{\alpha,p}\le\Tr\,G_{\alpha,1}$ on $\bM_2^{++}\times\bM_2^{++}$,
which contradicts Proposition \ref{P-3.14}(1).
\end{proof}

\begin{problem}\label{Q-5.20}\rm
We have no sufficient condition for the joint concavity of $\Tr\,SG_{\alpha,p}$ or
$\Tr\,\widetilde SG_{\alpha,p}$, except for the particular case $\alpha=1/2$, the case that
$\Tr\,SG_{\alpha,p}=\Tr\,\widetilde SG_{\alpha,p}=\Tr\,R_{1/2,2p}$.
\end{problem}

\section{Concluding remarks}\label{Sec-6}

(1)\enspace
In the present paper we have followed the previous paper \cite{Hiai} to discuss the $\alpha$-weighted
quasi-geometric type matrix means $R_{\alpha,p}$, $G_{\alpha,p}$, $SG_{\alpha,p}$,
$\widetilde SG_{\alpha,p}$ and $LE_\alpha$ not only for $0<\alpha<1$ but also $\alpha>1$ as defined in
Section \ref{Sec-2}. We are concerned with log-majorizations $\cM_{\alpha,p}\prec_{\log}\cN_{\alpha,q}$ for
pairs $(\cM,\cN)$ from $\{R,G,SG,\widetilde SG,LE\}$ and for $\alpha\in(0,\infty)\setminus\{1\}$, $p>0$. Already
known cases are summarized in Theorem \ref{T-3.1}. In Sections \ref{Sec-3.1}--\ref{Sec-3.4} we have
examined unknown cases with the aim of finding the necessary and sufficient condition on $p,q,\alpha$
under which $\cM_{\alpha,p}(A,B)\prec_{\log}\cN_{\alpha,q}(A,B)$ holds for all $A,B>0$. But in many cases
we have only a certain necessary condition and/or a certain sufficient condition for that, and the problem is still
left open as explained in Problems \ref{Q-3.5}, \ref{Q-3.8}, \ref{Q-3.12} and \ref{Q-3.20}.

(2)\enspace
A natural and interesting question is to find $\cM,\cN$ from $\{R,G,SG,\widetilde SG,LE\}$ and $\alpha,p,q$
such that the following (a) fails to hold but (b) holds:
(a) $\cM_{\alpha,p}(A,B)\prec_{\log}\cN_{\alpha,q}(A,B)$ for all $A,B\in\bM_n^{++}$;
(b) $\Tr\,\cM_{\alpha,p}(A,B)\le\Tr\,\cN_{\alpha,q}(A,B)$ for all $A,B\in\bM_n^{++}$. Note that (a) and (b) are
equivalent when $n=2$ by Remark \ref{R-2.4}(4), and are also equivalent for all $n$ in certain cases of
$(\cM_{\alpha,p},\cN_{\alpha,q})$ due to Theorems \ref{T-3.2}, \ref{T-3.3}, \ref{T-3.6} and Proposition \ref{P-4.2}.

(3)\enspace
In a recent paper \cite{SWY} Seo, Wada and Yamazaki have,
appealing to the grand Furuta inequality \cite{Fu}, shown Proposition \ref{P-3.10}(1) and improved
Proposition \ref{P-3.10}(2) in such a way that if $0<\alpha\le1/2$ and $p/q\le2\alpha$, then
$\widetilde SG_{\alpha,p}(A,B)\prec_{\log}\widetilde SG_{\alpha,q}(A,B)$ for all $A,B>0$.
Moreover, it has been shown in \cite[Theorem 3.13]{SWY} that for some $\alpha\in(1/2,1)$ and for any
$p\in(0,1)$ there exist $A,B\in\bM_2^{++}$ such that
$\widetilde SG_{\alpha,p}(A,B)\not\prec_{\log}\widetilde SG_{\alpha,1}(A,B)$. Indeed, in view of
Theorem \ref{T-3.19} (with Theorem \ref{T-2.2}), the same holds true for any $\alpha\in(1/2,1)$. From this
and Proposition \ref{P-3.11}(3) we conclude that $\widetilde SG_{\alpha,p}\prec_{\log}\widetilde SG_{\alpha,q}$
fails to hold for both $p<q$ and $p>q$ when $1/2<\alpha<1$. Thus, concerning Problem \ref{Q-3.12} on
$\widetilde G_{\alpha,p}$ for $0<\alpha<1$, the remaining is whether or not
$\widetilde SG_{\alpha,p}\prec_{\log}\widetilde SG_{\alpha,q}$ holds if $0<\alpha<1/2$ and $0<p<q$,
extending the above mentioned case $p/q\le2\alpha$.

(4)\enspace
Another topic we have addressed is the joint concavity/convexity of the trace functions $\Tr\,\cM_{\alpha,p}$
for each $\cM$ from $\{R,G,SG,\widetilde SG,LE\}$. The cases of $\Tr\,R_{\alpha,p}$ ($\alpha>0$)
and $\Tr\,G_{\alpha,p}$ ($0<\alpha<1$) are already known \cite{Zh,Hi4}. In Section \ref{Sec-5} we have aimed
at characterizing the joint concavity/convexity of $\Tr\,\cM_{\alpha,p}$ for unknown cases. Our method is
based on the close connection between the joint concavity/convexity of geometric type trace functions and the
monotonicity under CPTP maps of the generated R\'enyi type divergences described in Theorem \ref{T-5.3}.
Appealing to theory of quantum $\alpha$-R\'enyi divergences we have obtained a necessary condition for
$\Tr\,\cM_{\alpha,p}$ to be jointly concave/convex. The condition is the sandwiched inequalities of
$\Tr\,\cM_{\alpha,p}$ with $\Tr\,R_{\alpha,1/\alpha}$ and $\Tr\,G_{\alpha,1}$ (see Theorem \ref{T-5.8}), which
are connected with the log-majorizations discussed in Section \ref{Sec-3}. However, we cannot obtain a
sufficient condition for the joint concavity/convexity of $\Tr\,\cM_{\alpha,p}$ in this way, while the necessary
condition obtained can be best possible, for example, as exemplified in Remark \ref{R-5.10}. Thus the
problem is still left open for $\Tr\,G_{\alpha,p}$ ($\alpha>1$), $\Tr\,SG_{\alpha,p}$ and
$\Tr\,\widetilde SG_{\alpha,p}$ as mentioned in Problems \ref{Q-5.18} and \ref{Q-5.20}.

(5)\enspace
We have used the Lie--Trotter--Kato product formula (Theorem \ref{T-2.2} and \cite[Theorem 2.3]{Hiai})
in the proofs of (1), (4) and (5) of Corollary \ref{C-3.18} and Theorem \ref{T-5.12}(1) for example. In fact,
in view of Proposition \ref{P-2.1}(2) we may assume that $A,B>0$ in the proofs of those, in which case
the proof of Theorem \ref{T-2.2} becomes much simpler. Note however that in the proof of
Proposition \ref{P-2.1}(2) we have used \cite[Lemma A.3]{Hiai} that is essential in the proof of
\cite[Theorems A.1 and 2.3]{Hiai}.

(6)\enspace
Related to Theorems \ref{T-5.3} and \ref{T-5.8}, in Appendix \ref{Sec-C} below we clarify the structure of
quantum $\alpha$-R\'enyi type divergences that satisfy the monotonicity under a subclass of quantum
channels consisting of quantum-classical and classical-quantum channels. At the moment we have no explicit
example of quantum $\alpha$-R\'enyi type divergence that is not monotone under all quantum channels but
monotone under the above subclass. Although Appendix \ref{Sec-C} has not been used in the main body of
this paper, it might be of independent interest from the point of view of quantum information.

(7)\enspace
It is well known \cite{KA,Effros,ENG,HM,HUW} that the operator perspective
\[
\cP_f(B,A):=A^{1/2}f(A^{-1/2}BA^{-1/2})A^{1/2}
\]
associated to a function $f$ on $(0,\infty)$ is jointly operator convex on $\bM_n^{++}\times\bM_n^{++}$,
$n\in\bN$, if $f$ is operator convex, while it is jointly operator concave if $f$ is operator monotone. Thus it
might be expected to obtain a general joint convexity theorem for $\Tr\,\cP_f(B^p,A^p)^{1/p}$ for certain
$p>0$, that is the convexity counterpart of Theorem \ref{T-5.1}, when $f$ is operator convex on $(0,\infty)$.
However, this is not possible except for $p=1$, in the generic case that $\cP_f$ is jointly operator convex.
Indeed, when $f$ is a linear function $f(x)=1-\alpha+\alpha x$ ($0<\alpha<1$), the joint convexity of
$\Tr\,\cP_f(B^p,A^p)^{1/p}=\Tr\,\cA_{\alpha,p}(A,B)$ restricts to $1\le p\le2$ by Proposition \ref{P-5.2}(2).
When $f(x)=x^\alpha$ ($1<\alpha\le2$), the joint convexity of $\Tr\,\cP_f(B^p,A^p)^{1/p}=\Tr\,G_{\alpha,p}$
restricts to $p\le1$ by Proposition \ref{P-5.16}(1). The intersection is only $p=1$. It might be possible that
$\Tr\,\cP_f(B^p,A^p)^{1/p}$ is jointly convex for a certain $p<1$ under the constraint of $f(0^+)=0$.

(8)\enspace
We have checked that all of $\cA_{\alpha,p}$ for $0<\alpha<1$ and $G_{\alpha,p}$, $R_{\alpha,p}$,
$LE_\alpha$ for $\alpha\in(0,\infty)\setminus\{1\}$, where $p>0$, are neither jointly operator concave
nor jointly operator convex, except for $\cA_{\alpha,1}$, $G_{\alpha,1}$ for $0<\alpha<1$ and
$G_{\alpha,1}$ for $1<\alpha\le2$. Therefore, two variable operator/matrix means other than Kubo--Ando's
operator means or operator perspectives for operator convex functions on $(0,\infty)$ can hardly be jointly
operator concave/convex. Indeed, this is the reason why we have treated the joint concavity/convexity of the
trace functions $\Tr\,\cM_{\alpha,p}$ in Section \ref{Sec-5}. Nevertheless, it might be interesting to consider
whether or not $\cM_{\alpha,p}$ are jointly concave/convex in other orderings such as the entrywise
eigenvalue order or the weak majorization (see \cite[Sec.~2.2]{Hiai}).

\subsection*{Declaration of competing interest}

The author has no competing interests to disclose.

\subsection*{Acknowledgements}

The author thanks Milan Mosonyi for invitation to the workshop at the Erd\H os Center in July, 2024 and for
valuable suggestions which helped to improve this paper.

\appendix

\section{Equality case in norm inequality for $G_{\alpha,p}$}\label{Sec-A}

\begin{theorem}\label{T-A.1}
Let $\alpha\in(0,2]\setminus\{1\}$ and $\|\cdot\|$ be a strictly increasing unitarily invariant norm on
$\bM_n$. Then for every $A,B\in\bM_n^{++}$ the following conditions are equivalent:
\begin{itemize}
\item[(i)] $\|G_{\alpha,p}(A,B)\|=\|G_{\alpha,q}(A,B)\|$ for some $0<p<q$;
\item[(ii)] $\|G_{\alpha,p}(A,B)\|=\|LE_\alpha(A,B)\|$ for all $p>0$;
\item[(iii)] $G_{\alpha,p}(A,B)=G_{\alpha,q}(A,B)$ for some $0<p<q$;
\item[(iv)] $G_{\alpha,p}(A,B)=LE_\alpha(A,B)$ for all $p>0$;
\item[(v)] $AB=BA$.
\end{itemize}
\end{theorem}

\begin{proof}
Let us prove the case $1<\alpha\le2$. As the idea of the proof of the case $0<\alpha<1$ in
\cite[Theorem 3.1]{Hi1} does not work well, our proof below is similar to that of \cite[Theorem 2.1]{Hi1}. Since
the implications (v)$\implies$(iv)$\implies$\{(ii), (iii)\}$\implies$(i) are clear, it suffices to prove that (i)$\implies$(v).
Assume that (i) holds for some $0<p<q$. Then by \cite[Theorem 3.2]{KS} we have
$\|G_{\alpha,p}(A,B)\|=\|G_{\alpha,r}(A,B)\|$ for all $r\in[p,q]$. (Here, recall that the matrix perspective
$A\natural_\beta B:=A^{1/2}(A^{-1/2}BA^{-1/2})^\beta A^{1/2}$ ($A,B>0$) for $\beta\in[-1,0)$ was treated in
\cite{KS}, while we note that $A\natural_\beta B=B\natural_{1-\beta}A$ and $1<1-\beta\le2$.) Hence it follows
from \cite[Theorem 3.1]{KS} and \cite[Lemma 2.2]{Hi1} that
$\lambda\bigl(G_\alpha(A^p,B^p)^{1/p}\bigr)=\lambda\bigl(G_\alpha(A^r,B^r)^{1/r}\bigr)$ for all $r\in[p,q]$. Hence,
letting $A_0:=A^p$, $B_0:=B^p$ and $t:=r/p$ we have
\begin{align}\label{F-A.1}
\Tr\,G_\alpha(A_0,B_0)=\Tr\,G_\alpha(A_0^t,B_0^t)^{1/t},\qquad1\le t\le q/p.
\end{align}
Since the function $t\mapsto G_\alpha(A_0^t,B_0^t)^{1/t}$ is analytic in $(0,\infty)$, it follows that \eqref{F-A.1}
extends to all $t\in(0,\infty)$. Now write $A_0=e^H$ and $B_0=e^K$ with Hermitian $H,K\in\bM_n$. Since
$G_\alpha(A_0^t,B_0^t)^{1/t}\to e^{\alpha H+(1-\alpha)K}$ as $t\to0$, we obtain
\[
\Tr\,G_\alpha(e^{tH},e^{tK})^{1/t}=\Tr\,e^{\alpha H+(1-\alpha)K},\qquad t>0.
\]
By \cite[Theorem 3.1]{KS} and \cite[Lemma 2.2]{Hi1} again this implies that
$\lambda\bigl(G_\alpha(e^{tH}t,e^{tK})^{1/t}\bigr)=\lambda(e^{\alpha H+(1-\alpha)K})$ for all $t>0$, so that
\begin{align}\label{F-A.2}
\Tr\,G_\alpha(e^{tH},e^{tK})=\Tr\,e^{t(\alpha H+(1-\alpha)K)},\qquad t\in(-\infty,\infty).
\end{align}

Below we compute the coefficients up to the fourth of the Taylor expansion in $t$ of the LHS of \eqref{F-A.2}.
From the Taylor expansions of $e^{tH}$ and $e^{-tK/2}$ one can write
\[
e^{-tK/2}e^{tH}e^{-tK/2}=I+tX_1+t^2X_2+t^3X_3+t^4X_4+\cdots,
\]
where
\begin{equation}\label{F-A.3}
\begin{aligned}
X_1&:=H-K, \\
X_2&:={(H-K)^2\over2}, \\
X_3&:={H^3\over6}-{H^2K+KH^2\over4}+{HK^2+2KHK+K^2H\over8}-{K^3\over6}, \\
X_4&:={H^4\over24}-{H^3K+KH^3\over12}+{H^2K^2+2KH^2K+K^2H^2\over16} \\
&\qquad-{HK^3+3KHK^2+3K^2HK+K^3H\over48}+{K^4\over24}.
\end{aligned}
\end{equation}
Using
\[
(1+x)^\alpha=1+\alpha x+{\alpha(\alpha-1)\over2}x^2+{\alpha(\alpha-1)(\alpha-2)\over6}x^3
+{\alpha(\alpha-1)(\alpha-2)(\alpha-3)\over24}x^4+\cdots,
\]
one can write
\[
(e^{-tK/2}e^{tH}e^{-tK/2})^\alpha=I+tY_1+t^2Y^2+t^3Y_3+t^4Y_4+\cdots,
\]
where
\begin{equation}\label{F-A.4}
\begin{aligned}
Y_1&:=\alpha X_1, \\
Y_2&:=\alpha X_2+{\alpha(\alpha-1)\over2}X_1^2, \\
Y_3&:=\alpha X_3+{\alpha(\alpha-1)\over2}(X_1X_2+X_2X_1)+{\alpha(\alpha-1)(\alpha-2)\over6}X_1^3, \\
Y_4&:=\alpha X_4+{\alpha(\alpha-1)\over2}(X_1X_3+X_2^2+X_3X_1) \\
&\qquad+{\alpha(\alpha-1)(\alpha-2)\over6}(X_1^2X_2+X_1X_2X_1+X_2X_1^2)
+{\alpha(\alpha-1)(\alpha-2)(\alpha-3)\over24}X_1^4.
\end{aligned}
\end{equation}
Therefore, one can write
\[
\Tr\,G_\alpha(e^{tH},e^{tK})=\Tr(e^{tK/2}e^{tH}e^{-tK/2})^\alpha e^{tK}
=\Tr\,I+tz_1+t^2z_2+t^3z_3+t^4z_4+\cdots,
\]
where
\begin{equation}\label{F-A.5}
\begin{aligned}
z_1&:=\Tr(Y_1+K), \\
z_2&:=\Tr\biggl(Y_2+Y_1K+{K^2\over2}\biggr), \\
z_3&:=\Tr\biggl(Y_3+Y_2K+Y_1{K^2\over2}+{K^3\over6}\biggr), \\
z_4&:=\Tr\biggl(Y_4+Y_3K+Y_2{K^2\over2}+Y_1{K^3\over6}+{K^4\over24}\biggr).
\end{aligned}
\end{equation}
Direct though quite tedious computations from \eqref{F-A.3}--\eqref{F-A.5} give
\begin{align}
z_1&:=\Tr(\alpha H+(1-\alpha)K), \nonumber\\
z_2&:={1\over2}\Tr(\alpha H+(1-\alpha)K)^2, \nonumber\\
z_3&:={1\over6}\Tr(\alpha H+(1-\alpha)K)^3, \nonumber\\
z_4&:={1\over24}\Tr\Bigl(\alpha^4H^4+4\alpha^3(1-\alpha)H^3K
+4\alpha(\alpha-1)(\alpha^2-\alpha+1)H^2K^2 \nonumber\\
&\qquad\qquad+2\alpha(\alpha-1)(\alpha^2-\alpha-2)HKHK+4\alpha(1-\alpha)^3HK^3
+(1-\alpha)^4K^4\Bigr). \label{F-A.6}
\end{align}
Comparing the Taylor coefficients of $t^4$ of both sides of \eqref{F-A.2}, we have
\begin{align*}
24z_4&=\Tr(\alpha H+(1-\alpha)K)^4 \\
&=\Tr\Bigl(\alpha^4H^4+4\alpha^3(1-\alpha)H^3K+4\alpha^2(1-\alpha)^2H^2K^2 \\
&\qquad\quad+2\alpha^2(1-\alpha)^2HKHK+4\alpha(1-\alpha)^3HK^3+(1-\alpha)^4K^4\Bigr).
\end{align*}
From this and \eqref{F-A.6} we finally arrive at
\[
\Tr\,H^2K^2=\Tr\,HKHK,
\]
which is equivalently written as $\Tr(HK-KH)^*(HK-KH)=0$. Hence $HK=KH$, so that $A_0B_0=B_0A_0$,
equivalently $AB=BA$.
\end{proof}

\section{Proof of Theorem \ref{T-5.3}}\label{Sec-B}

\begin{proof}
(i)$\implies$(ii) under [(b), (c)].\enspace
Define a CPTP map $\Phi:\bM_{2n}=\bM_n\otimes\bM_2\to\bM_n$ by
\[
\Phi\biggl(\begin{bmatrix}X_{11}&X_{12}\\X_{21}&X_{22}\end{bmatrix}\biggr):=X_{11}+X_{22},
\qquad X_{ij}\in\bM_n.
\]
For any $A_i,B_i\in\dom_n\,Q$, $i=1,2$, and $\lambda\in(0,1)$ set
$A:=\lambda A_1\oplus(1-\lambda)A_2$ and $B:=\lambda B_1\oplus(1-\lambda)B_2$ in
$\bM_{2n}$. We then have
\begin{align*}
&Q(\lambda A_1+(1-\lambda)A_2,\lambda B_1+(1-\lambda)B_2) \\
&\quad=Q(\Phi(A),\Phi(B)) \\
&\quad\ge Q(\lambda A_1\oplus(1-\lambda)A_2,\lambda B_1\oplus(1-\lambda)B_2)
\qquad\mbox{(by (i))} \\
&\quad=\lambda Q(A_1,B_1)+(1-\lambda)Q(A_2,B_2)\qquad\mbox{(by (c), (b))}.
\end{align*}
The proof of (i$'$)$\implies$(ii$'$) is similar.

(ii)$\implies$(i) under [(a), (d), (e)].\enspace
Let $\Phi:\bM_n\to\bM_m$ be a CPTP map. According to the representation theorem of CPTP maps
(see, e.g., \cite[Theorem 5.1]{Ha}), one can choose an $l\in\bN$, a rank one projection
$\eta\in\bM_l\otimes\bM_m$ and a unitary $V\in\bM_n\otimes\bM_l\otimes\bM_m$ such that
\begin{align}\label{F-B.3}
\Phi(X)=\Tr_{nl}V(X\otimes\eta)V^*,\qquad X\in\bM_n,
\end{align}
where $\Tr_{nl}:\bM_n\otimes\bM_l\otimes\bM_m\to\bM_m$ is the partial trace. Let
$\tau_0:=(nl)^{-1}I_{nl}\in\bM_n\otimes\bM_l$. Then $\tau_0\otimes\Tr_{nl}$ is the conditional expectation from
$\bM_n\otimes\bM_l\otimes\bM_m$ onto $I_{nl}\otimes\bM_m$ with respect to the trace, and it is well known
(see, e.g., \cite[Sec.~7]{Ru}) that for every $Z\in\bM_n\otimes\bM_l\otimes\bM_m$,
\begin{align}\label{F-B.4}
\tau_0\otimes\Tr_{nl}(Z)=(nl)^{-2}\sum_{\mu,\nu=1}^{nl}
(S^\mu\otimes I_m)(W^\nu\otimes I_m)Z(W^{*\mu}\otimes I_m)(S^{*\nu}\otimes I_m),
\end{align}
where $S$ and $W$ are the Weyl--Heisenberg matrices in $\bM_d$, $d:=nl$, i.e., $S_{jk}:=\delta_{j+1,k}$
(with $j+1$ mod $d$) and $W_{jk}:=\omega^j\delta_{jk}$ with $\omega:=e^{2\pi i/d}$. Note that
$U_{\mu\nu}:=(S^\mu\otimes I_m)(W^\nu\otimes I_m)$ is a unitary and
$U_{\mu\nu}^*=(W^{*\mu}\otimes I_m)(S^{*\nu}\otimes I_m)$. By \eqref{F-B.3} and \eqref{F-B.4} we have
\begin{align}\label{F-B.5}
\tau_0\otimes\Phi(X)=(nl)^{-2}\sum_{\mu,\nu=1}^{nl}U_{\mu\nu}V(X\otimes\eta)V^*U_{\mu\nu}^*,
\qquad X\in\bM_n.
\end{align}
Therefore,
\begin{align*}
Q(\Phi(A),\Phi(B))&=Q(\tau_0\otimes\Phi(A),\tau_0\otimes\Phi(B))
\qquad\mbox{(by (d), (a))} \\
&=Q\Biggl((nl)^{-2}\sum_{\mu,\nu=1}^{nl}U_{\mu\nu}V(A\otimes\eta)V^*U_{\mu\nu}^*,
(nl)^{-2}\sum_{\mu,\nu=1}^{nl}U_{\mu\nu}V(B\otimes\eta)V^*U_{\mu\nu}^*\Biggr) \\
&\ge(nl)^{-2}\sum_{\mu,\nu=1}^{nl}Q\bigl(U_{\mu\nu}V(A\otimes\eta)V^*U_{\mu\nu}^*,
U_{\mu\nu}V(B\otimes\eta)V^*U_{\mu\nu}^*\bigr)\qquad\mbox{(by (ii))} \\
&=Q(A\otimes\eta,B\otimes\eta)\qquad\mbox{(by (e))} \\
&=Q(A,B)\qquad\mbox{(by (d), (a))}.
\end{align*}
The proof of (ii$'$)$\implies$(i$'$) under [(a), (d), (e)] is similar.
\end{proof}

\section{Monotonicity under restricted CPTP maps}\label{Sec-C}

Let $\alpha\in(0,\infty)\setminus\{1\}$ and
$Q_\alpha:(A,B)\in\dom\,Q_\alpha\mapsto Q_\alpha(A,B)\in[0,+\infty)$ be a function with the domain
$\dom\,Q_\alpha$ given in \eqref{F-5.1}, where $\dom\,Q_\alpha=\bigsqcup_n(\bM_n^+\times\bM_n^+)_\ge$
when $\alpha>1$. Assume that $Q_\alpha$ satisfies properties (d) and (e) in Section \ref{Sec-5.1} and
\begin{align}\label{F-C.1}
Q_\alpha(\diag(a),\diag(b))=\sum_{i=1}^na_i^{1-\alpha}b_i^\alpha
\end{align}
for all $a,b\in[0,\infty)^n$, $n\in\bN$, such that $(\diag(a),\diag(b))\in\dom\,Q_\alpha$. Define the quantum
divergence of R\/enyi type for $A,B\in\bM_n^+$, $n\in\bN$, with $B\ne0$ as in \eqref{F-5.7} by
\begin{align}\label{F-C.2}
D^{Q_\alpha}(B\|A):=\begin{cases}
{1\over\alpha-1}\log{Q_\alpha(A,B)\over\Tr\,B} &
\text{if $(A,B)\in\dom Q_\alpha$}, \\ +\infty. & \text{otherwise}
\end{cases}
\end{align}
For example, if $\cM_{\alpha,p}$ is any of $R_{\alpha,p}$, $G_{\alpha,p}$, $SG_{\alpha,p}$,
$\widetilde SG_{\alpha,p}$ and $LE_\alpha$ for $\alpha\in(0,\infty)\setminus\{1\}$ and $p>0$, then
$Q_\alpha:=\Tr\,\cM_{\alpha,p}$ meets the above requirements. In Theorem \ref{T-5.8} four trace
inequalities have appeared as necessary conditions for $\Tr\,\cM_{\alpha,p}$ to be jointly concave or
jointly convex. In the next theorem we show what is meant by each of those inequalities for more general
$Q_\alpha$ stated above.

We consider two special classes of quantum channels (CPTP maps). If $\Phi:\bM_n\to\bM_m$ is a positive
trace-preserving map such that the range of $\Phi$ is included in a commutative subalgebra of $\bM_m$
is called a \emph{quantum-classical channel}. In this case, there are orthogonal projections $(P_i)_{i=1}^k$
in $\bM_m$ such that $\sum_{i=1}^kP_i=I_m$ and the range of $\Phi$ is included in
$\bigoplus_{i=1}^k\bC P_i$. Then we have a POVM $(M_i)_{i=1}^k$ on $\bC^n$ (see \eqref{F-5.3}) such that
$\Phi(A)=\sum_{i=1}^k(\Tr\,M_iA/\Tr\,P_i)P_i$ for $A\in\bM_n$. In this way, a quantum-classical channel is
essentially a POVM. The other way around, a positive trace-preserving map $\Phi:\bC^n\to\bM_m$ is called
a \emph{classical-quantum channel}. Those quantum-classical and classical-quantum channels are
automatically CPTP maps.

\begin{theorem}\label{T-C.1}
Let $Q_\alpha$ be as mentioned above.
\begin{itemize}
\item[(1)] Let $0\le\alpha<1$. Then, for the following conditions (i)--(iii), we have
(i)$\implies$(ii)$\implies$(iii), and if $1/2\le\alpha<1$ then (i)--(iii) are equivalent.
\begin{itemize}
\item[(i)] $Q_\alpha(A,B)\le\Tr\,R_{\alpha,1/\alpha}(A,B)$ for all $(A,B)\in\dom\,Q_\alpha$.
\item[(ii)] $Q_\alpha(\Phi(A),\Phi(B))\ge Q_\alpha(A,B)$ for all $(A,B)\in\dom_nQ_\alpha$ and for any
quantum-classical channel $\Phi:\bM_n\to\bM_m$, $n,m\in\bN$.
\item[(iii)] $Q_\alpha(A,\cE_A(B))\ge Q_\alpha(A,B)$ for all $(A,B)\in\dom\,Q_\alpha$, where $\cE_A$ is
the pinching with respect to $A$.
\end{itemize}

\item[(2)] Let $1<\alpha<\infty$. The following conditions (i)--(iii) are equivalent:
\begin{itemize}
\item[(i)] $Q_\alpha(A,B)\ge\Tr\,R_{\alpha,1/\alpha}(A,B)$ for all $(A,B)\in\dom\,Q_\alpha$.
\item[(ii)] $Q_\alpha(\Phi(A),\Phi(B))\le Q_\alpha(A,B)$ for all $(A,B)\in\dom_nQ_\alpha$ and for any
quantum-classical channel $\Phi:\bM_n\to\bM_n$, $n,m\in\bN$.
\item[(iii)] $Q_\alpha(A,\cE_A(B))\le Q_\alpha(A,B)$ for all $(A,B)\in\dom\,Q_\alpha$.
\end{itemize}

\item[(3)] Let $0<\alpha<1$. The following conditions (i) and (ii) are equivalent:
\begin{itemize}
\item[(i)] $Q_\alpha(A,B)\ge\Tr\,G_{\alpha,1}(A,B)$ for all $(A,B)\in\dom\,Q_\alpha$.
\item[(ii)] $Q_\alpha(\Phi(a),\Phi(b))\ge Q_\alpha(\diag(a),\diag(b))$ for all $a,b\in[0,\infty)^n$ with
$(\diag(a),\diag(b))\in\dom\,Q_\alpha$ and for any classical-quantum channel $\Phi:\bC^n\to\bM_m$,
$n,m\in\bN$.
\end{itemize}

\item[(4)] Let $1<\alpha<\infty$. Then, for the following conditions (i) and (ii), we have (ii)$\implies$(i), and if
$1<\alpha\le2$ then (i) and (ii) are equivalent.
\begin{itemize}
\item[(i)] $Q_\alpha(A,B)\le\Tr\,G_{\alpha,1}(A,B)$ for all $(A,B)\in\dom\,Q_\alpha$.
\item[(ii)] $Q_\alpha(\Phi(a),\Phi(b))\le Q_\alpha(\diag(a),\diag(b))$ for all $a,b\in[0,\infty)^n$ with
$(\diag(a),\diag(b))\in\dom\,Q_\alpha$ and for any classical-quantum channel $\Phi:\bC^n\to\bM_m$,
$n,m\in\bN$.
\end{itemize}
\end{itemize}
\end{theorem}

\begin{proof}
We give the proofs of (1) and (4) only; those of (2) and (3) are similar to (1) and (4), respectively. First note
that $Q_\alpha(A,0)=0$ for all $A\ge0$. In fact, by property (e) and \eqref{F-C.1} we have
$Q_\alpha(A,0)=Q_\alpha(\diag(\lambda(A)),\diag(0))=0$, where $\lambda(A)$ is the eigenvalue vector of
$A$. Since $R_{\alpha,1/\alpha}(A,0)=G_{\alpha,1}(A,0)=0$ as well, we may always assume $B\ne0$ and
$b\ne0$ in the following.

(1)\enspace
(i)$\implies$(ii).\enspace
For every $(A,B)\in\dom\,Q_\alpha$, $B\ne0$, we have
$\widetilde D_\alpha(B\|A)=D^{R_{\alpha,1/\alpha}}(B\|A)\le D^{Q_\alpha}(B\|A)$ by (i).
If $\Phi$ is a quantum-classical channel, then
\[
D^{Q_\alpha}(\Phi(B)\|\Phi(A))=D_\alpha^\cl(\Phi(B)\|\Phi(A))=\widetilde D_\alpha(\Phi(B)\|\Phi(A))
\le\widetilde D_\alpha(B\|A)\le D^{Q_\alpha}(B\|A).
\]
This implies (ii).

(ii)$\implies$(iii).\enspace
For any $(A,B)\in\dom_nQ_\alpha$, since $A\cE_A(B)=\cE_A(B)A$, we can choose an orthonormal basis
$(e_i)_{i=1}^n$ of $\bC^n$ such that $A=\sum_{i=1}^na_i|e_i\>\<e_i|$ and
$\cE_A(B)=\sum_{i=1}^nb_i|e_i\>\<e_i|$. Let $\Phi:\bM_n\to\bM_n$ be the pinching by the rank one
projections $|e_i\>\<e_i|$, $1\le i\le n$. Then $\Phi$ is a quantum-classical channel. Hence thanks to (e) we
have
\[
Q_\alpha(A,\cE_A(B))=Q_\alpha(\diag(a),\diag(b))=Q_\alpha(\Phi(A),\Phi(B))\ge Q_\alpha(A,B).
\]

(iii)$\implies$(i).\enspace
Assume that $1/2\le\alpha<1$. For every $(A,B)\in\dom\,Q_\alpha$, $B\ne0$, we have
\begin{align*}
D^{Q_\alpha}(B\|A)={1\over m}D^{Q_\alpha}(B^{\otimes m}\|A^{\otimes m})
\ge{1\over m}D^{Q_\alpha}(\cE_{A^{\otimes m}}(B^{\otimes m})\|A^{\otimes m})
={1\over m}D_\alpha^\cl(\cE_{A^{\otimes m}}(B^{\otimes m})\|A^{\otimes m}),
\end{align*}
where the first equality is by definition \eqref{F-C.2} and property (d), the inequality is due to (iii), and the last
equality is by (e) and \eqref{F-C.1}. Since $1/2\le\alpha<1$, letting $m\to\infty$ (see definition \eqref{F-5.4})
and applying Theorem \ref{T-5.7}(3) we have
$D^{Q_\alpha}(B\|A)\ge\widetilde D_\alpha(B\|A)=D^{R_{\alpha,1/\alpha}}(B\|A)$, which implies (i).

(4)\enspace
(ii)$\implies$(i).\enspace
Let $(A,B)\in\dom\,Q_\alpha$, $B\ne0$, and let $(\Gamma,a,b)$ be a reverse
test for $(A,B)$ (see the definition just after \eqref{F-5.2}). Then by (ii) and \eqref{F-C.1} we have
\[
D^{Q_\alpha}(B\|A)=D^{Q_\alpha}(\Gamma(b)\|\Gamma(a))\le D^{Q_\alpha}(\diag(b)\|\diag(a))
=D_\alpha^\cl(b\|a).
\]
Therefore, it follows from \eqref{F-5.2} and Theorem \ref{T-5.7}(1) and (2) that
$D^{Q_\alpha}(B\|A)\le D_\alpha^{\max}(B\|A)\le D^{G_{\alpha,1}}(B\|A)$, implying (i).

(i)$\implies$(ii).\enspace
Assume that $1<\alpha\le2$. Let $a,b\in[0,\infty)^n$, $b\ne0$, with $(\diag(a),\diag(b))\in\dom\,Q_\alpha$.
For any classical-quantum channel $\Phi:\bC^n\to\bM_m$ it follows from (i) and Theorem \ref{T-5.7}(1) that
\[
D^{Q_\alpha}(\Phi(b)\|\Phi(a))\le D^{G_{\alpha,1}}(\Phi(b)\|\Phi(a))=D_\alpha^{\max}(\Phi(b)\|\Phi(a)).
\]
Moreover, we have $D_\alpha^{\max}(\Phi(b)\|\Phi(a))\le D_\alpha^\cl(b\|a)=D^{Q_\alpha}(\diag(b)\|\diag(a))$
by \eqref{F-5.2} and \eqref{F-C.1}, so that $D^{Q_\alpha}(\Phi(b)\|\Phi(a))\le D^{Q_\alpha}(\diag(b)\|\diag(a))$,
implying (ii).
\end{proof}

\begin{example}\label{E-C.2}\rm
Here let us exemplify Theorem \ref{T-C.1} for $Q_\alpha=\Tr\,R_{\alpha,p}$ and $\Tr\,G_{\alpha,p}$. We say
that a quantum channel $\Phi$ is \emph{semi-classical} if it is either quantum-classical or classical-quantum.

(1)\enspace
As for $\Tr\,R_{\alpha,p}$, by Theorems \ref{T-3.1}(a) and \ref{T-3.2} note that for $0<\alpha<1$,
$\Tr\,G_{\alpha,1}\le\Tr\,R_{\alpha,p}\le\Tr\,R_{\alpha,1/\alpha}$ if and only if $p\le1/\alpha$, and that for
$\alpha>1$, $\Tr\,R_{\alpha,1/\alpha}\le\Tr\,R_{\alpha,p}\le\Tr\,G_{\alpha,1}$ if and only if
$\max\{\alpha/2,\alpha-1\}\le1/p\le\alpha$. Theorem \ref{T-C.1} says that if either of these holds, then
$D^{R_{\alpha,p}}$ ($=D_{\alpha,z}$ with $z=1/p$) is monotone under all semi-classical channels, and vice
versa if $\alpha\ge1/2$. Thus, by Theorem \ref{T-5.9} we see that for any $\alpha\in[1/2,\infty)\setminus\{1\}$
the $\alpha$-$z$-R\'enyi divergence $D_{\alpha,z}$ satisfies the monotonicity under quantum channels if
and only if it does under semi-classical channels.

(2)\enspace
As for $\Tr\,G_{\alpha,p}$, by Proposition \ref{P-4.2} and Theorem \ref{T-3.2} note that for $0<\alpha<1$,
$\Tr\,G_{\alpha,1}\le\Tr\,G_{\alpha,p}\le\Tr\,R_{\alpha,1/\alpha}$ if and only if $p\le1$, and that for
$1<\alpha\le2$, $\Tr\,R_{\alpha,1/\alpha}\le\Tr\,G_{\alpha,p}\le\Tr\,G_{\alpha,1}$ if and only if
$\max\bigl\{1/2,{\alpha-1\over\alpha}\bigr\}\le p\le1$. Theorem \ref{T-C.1} says that either of these holds if
and only if $D^{G_{\alpha,p}}$ satisfies the monotonicity under semi-classical channels. Thus, by
Theorem \ref{T-5.15} we see that for $0<\alpha<1$ the monotonicity of $D^{G_{\alpha,p}}$ under quantum
channels is equivalent to that under semi-classical channels, while this is not known for $\alpha>1$ (in view
of Proposition \ref{P-5.16}(1)).
\end{example}

The next corollary shows the `conditional' joint concavity/convexity of the above $Q_\alpha$.

\begin{corollary}\label{C-C.3}
Let $Q_\alpha$ be as above and assume that it satisfies, in addition to (d), (e), properties (b), (c) in
Section \ref{Sec-5.1}. Let $(A_i,B_i)\in\dom\,Q_\alpha$, $i=1,2$, and $0<\lambda<1$. Assume that
$\lambda A_1+(1-\lambda)A_2$ and $\lambda B_1+(1-\lambda)B_2$ commute. If $0<\alpha<1$ and
condition (ii) (or (i)) of Theorem \ref{T-C.1}(1) holds, then
\[
Q_\alpha(\lambda A_1+(1-\lambda)A_2,\lambda B_1+(1-\lambda)B_2)
\ge\lambda Q_\alpha(A_1,B_1)+(1-\lambda)Q_\alpha(A_2,B_2).
\]
If $1<\alpha<\infty$ and condition (ii) (equivalently (i)) of Theorem \ref{T-C.1}(2) holds, then
\[
Q_\alpha(\lambda A_1+(1-\lambda)A_2,\lambda B_1+(1-\lambda)B_2)
\le\lambda Q_\alpha(A_1,B_1)+(1-\lambda)Q_\alpha(A_2,B_2).
\]
\end{corollary}

\begin{proof}
Let $A_i,B_i$ be in $\bM_n$. By assumption we take the commutative subalgebra $\cC$ of $\bM_n$
generated by $\lambda A_i+(1-\lambda)B_i$, $i=1,2$. Let $\cE_\cC$ be the trace-preserving conditional
expectation from $\bM_n$ onto $\cC$. Define a quantum-classical channel
$\Phi_{2n}:\bM_{2n}=\bM_n\otimes\bM_2\to\bM_n$ by
\[
\Phi\biggl(\begin{bmatrix}X_{11}&X_{12}\\X_{21}&X_{22}\end{bmatrix}\biggr)
:=\cE_\cC(X_{11}+X_{22}),\qquad X_{ij}\in\bM_n.
\]
Then both conclusions follow similarly to the proof (i)$\implies$(ii) in Appendix \ref{Sec-B} with use of (ii) of
Theorem \ref{T-C.1}(1) and (2).
\end{proof}

For instance, if $\alpha>1$ and $p\ge\max\bigl\{1/2,{\alpha-1\over\alpha}\}$, then Theorem \ref{T-3.2}(2)
implies that $\Tr\,G_{\alpha,p}(A,B)\ge\Tr\,R_{\alpha,1/\alpha}(A,B)$ for all $A,B\ge0$ with $s(A)\ge s(B)$.
Hence in this case, $\Tr\,G_{\alpha,p}$ satisfies the `conditional' joint convexity by the above corollary,
supplementing Proposition \ref{P-5.16}(1).

\subsection*{Data availability}

No data was used for the research described in the article.

\end{document}